\documentclass[12pt]{amsart}
\usepackage[margin=1in]{geometry}
\usepackage{amsmath,amssymb,amsfonts}

\usepackage[usenames,dvipsnames]{xcolor}
\usepackage[colorlinks=false]{hyperref}

\newtheorem{theorem}{Theorem}
\newtheorem{lemma}{Lemma}
\newtheorem{corollary}{Corollary}
\newtheorem{proposition}{Proposition}
\newtheorem{remark}{Remark}

\newtheorem*{theorem*}{Theorem}

\theoremstyle{definition}

\def\De{\Delta}

\def\<{\langle}
\def\>{\rangle}

\DeclareMathOperator{\cpct}{Cap}
\DeclareMathOperator{\ccpct}{cap}

\DeclareMathOperator{\supp}{supp}
\DeclareMathOperator{\diverg}{div}
\DeclareMathOperator{\diam}{diam}

\DeclareMathOperator{\dist}{dist}
\DeclareMathOperator{\lin}{span}
\DeclareMathOperator{\Lip}{Lip}

\begin{document}
\title[Removable sets and $L^p$-uniqueness]{Removable sets and $L^p$-uniqueness on manifolds and metric measure spaces}

\author{M. Hinz$^1$, J. Masamune$^2$, K. Suzuki$^3$}
\address{$^1$ Department of Mathematics,
Bielefeld University, 33501 Bielefeld, Germany}
\email{mhinz@math.uni-bielefeld.de}
\address{$^2$Department of Mathematics,  Tohoku University,
6-3 Aramaki Aza-Aoba, Aoba-ku, Sendai 980-8578, Japan}
\email{jun.masamune.c3@tohoku.ac.jp}
\address{$^3$ Department of Mathematics,
Bielefeld University, 33501 Bielefeld, Germany}
\email{ksuzuki@math.uni-bielefeld.de}
\thanks{$^1$ Research supported in part by the DFG IRTG 2235: 'Searching for the regular in the irregular: Analysis of singular and random systems' and by the DFG CRC 1283, `Taming uncertainty and profiting from randomness and low regularity in analysis, stochastics and their applications'.}
\thanks{$^3$ Research supported by the Alexander von Humboldt Foundation}

\begin{abstract}
We study symmetric diffusion operators on metric measure spaces. Our main question is whether or not the restriction of the operator to a suitable core continues to be essentially self-adjoint or $L^p$-unique if a small closed set is removed from the space. The effect depends on how large the removed set is, and we provide characterizations of the critical size in terms of capacities and Hausdorff dimension. As a key tool we prove a truncation result for potentials of nonnegative functions. We apply our results to Laplace operators on Riemannian and sub-Riemannian manifolds and on metric measure spaces satisfying curvature dimension conditions. For non-collapsing Ricci limit spaces with two-sided Ricci curvature bounds we observe that the self-adjoint Laplacian is already fully determined by the classical Laplacian on the regular part. 
\end{abstract}


\maketitle

\setcounter{tocdepth}{1}
\tableofcontents

\section{Introduction}

We consider generators $\mathcal{L}$ of symmetric diffusion semigroups on metric measure spaces $M$ and investigate whether or not the removal of a small closed subset $\Sigma\subset M$ leads to a loss of essential self-adjointness or $L^p$-uniqueness.
Recall that in general an operator, given on a dense a priori domain $\mathcal{A}$ in $L^2(M)$, may have different self-adjoint extensions, and each such extension encodes a specific type of boundary condition. We assume that the generator $\mathcal{L}$ itself, together with a suitable a priori domain $\mathcal{A}$, is essentially self-adjoint, i.e., that it has a unique self-adjoint extension. This property may be viewed as the uniqueness of the quantum system determined by $\mathcal L|_{\mathcal A}$, \cite[Section X.1 p.\ 135]{RS80v2}. Since in general no classical notion of differentiability is available, we introduce a number of abstract conditions the a priori domain $\mathcal{A}$ should satisfy. Cutting out a small closed set $\Sigma$ from $M$, one can restrict $\mathcal{L}$ to the space $\mathcal{A}(M\setminus \Sigma)$ of those elements of $\mathcal{A}$ whose support lies in the complement of $\Sigma$. If $\Sigma$ is small enough, it is `ignored by the operator' and $\mathcal{L}|_{\mathcal{A}(M\setminus \Sigma)}$ has the same unique self-adjoint extension $\mathcal{L}$. In this case we say that $\Sigma$ is \emph{removable} (from the point of view of operator extensions). If $\Sigma$ is too large, the operator perceives it as a boundary at which different boundary conditions can be imposed. The critical size of $\Sigma$ can be characterized in terms of capacities and Hausdorff (co-)dimension. A very similar logics applies to $L^p$-uniqueness, that is, the question whether the extension to the generator of a strongly continuous semigroup on $L^p$ is unique.

We first provide a general result, Theorem \ref{T:esa}, it characterizes the critical size of $\Sigma$ in terms of capacities based on $\mathcal{A}$, cf. Section \ref{S:caps}. We then prove the equivalence of these capacities and capacities based on resolvent operators as commonly used in potential theory, these capacities are more {\it amenable} objects in connection to Hausdorff measures. Our key tool for this comparison of capacities is an estimate for truncations 
of potentials with respect to the graph norm of the generator, Theorem \ref{T:Mazja}. After making the link to Hausdorff measures in Sections \ref{S:Hausdorfftocap} and \ref{S:captoHausdorff}, we apply our results to Riemannian manifolds, sub-Riemannian manifolds and {\rm RCD}$^\ast(K,N)$ spaces in Sections \ref{S:manifolds}--\ref{S:RCD}.

For Sobolev spaces and elliptic operators on Euclidean domains the connection between capacities and removable sets is a  classical topic, \cite{AP73, HarveyPolking73, Maz'ya63, Maz'ya70, Maz'ya72, Maz'ya73, Maz'ya75, Maz'yabook}. In this context removability had typically referred to the extendability of solutions, see \cite[Section 2.7]{AH96}. The uniqueness or non-uniqueness of self-adjoint extension of Schr\"odinger operators after the removal of a single point had been addressed in \cite[Theorem X.11]{RS80v2}. The uniqueness of Markovian self-adjoint extensions of operators had been studied in \cite{Takeda92}, see also \cite{Eberle99} and \cite{FOT94}; this type of uniqueness is the correct notion to guarantee the uniqueness of associated symmetric diffusion processes. Clearly essential self-adjointness implies the uniqueness of Markovian self-adjoint extensions. For Laplacians and their powers on Euclidean spaces the connection between uniqueness of self-adjoint extensions, capacities of removed sets and Hausdorff codimension were spelled out in more detail in \cite{AGMST13} and \cite{HiKaMa}. In \cite{HiKa} related $L^p$-uniqueness results were established on infinite dimensional spaces. First results on essential self-adjointness and removable sets for the Laplacian on Riemannian manifolds were provided in \cite{CdV82}, where $\Sigma$ was assumed to be a single point set, and in \cite{Masamune1999}, where $\Sigma$ was assumed to be a smooth submanifold.

Similarly as the mentioned results for the Euclidean cases, our results are {\it structure-free} in the sense that we do not assume $\Sigma$ to have any specific structure except being {\it closed}. For instance, if $(M, g)$ is a complete Riemannian manifold of dimension $d\geq 4$, $\mu$ denotes the Riemannian volume and $\Delta_\mu$ the classical Laplacian, then we observe that for any closed set $\Sigma\subset M$ we have the following:
\begin{enumerate}
\item[(i)] If $\mu(\Sigma)=0$ and $\Delta_\mu|_{C_c^\infty(M\setminus \Sigma)}$ is essentially self-adjoint, then $\mathcal{H}^{d-4+\varepsilon}(\Sigma)=0$ for all $\varepsilon>0$ (if $d>4$) respectively $\mathcal{H}^h(\Sigma)=0$ for any Hausdorff function $h$ satisfying $\int_0^1 (-\log r)dh(r)<+\infty$ (if $d=4$);
\item[(ii)] If $\mathcal{H}^{d-4}(\Sigma)<+\infty$ (if $d>4$) respectively $\mathcal{H}^{h_2}(\Sigma)<+\infty$ (if $d=4$), then $\Delta_\mu|_{C_c^\infty(M\setminus \Sigma)}$ is essentially self-adjoint. 
\end{enumerate}
Here $\mathcal{H}^s$ denotes the $s$-dimensional Hausdorff measure, $\mathcal{H}^h$ a generalized Hausdorff measure with Hausdorff function $h$, and $h_2(r):=(1+\log_+ \frac{1}{r})^{-1}$, $r>0$. See Theorem \ref{T:Hausdorff} in Section \ref{S:manifolds}. This {\it structure-free} type of characterization works well also for metric measure spaces $M$. Indeed our method is robust enough to deal with the singular sets that arise in {\it non-collapsing Ricci limit spaces with two-sided Ricci curvature bounds}, see Theorem \ref{T:codim4}. 

We proceed as follows. In Section \ref{S:caps} we introduce related $(2,p)$-capacities $\ccpct_{2,p}^\mathcal{A}$ of a similar type as in \cite{HarveyPolking73, Maz'ya75} and use them to characterize the critical size of $\Sigma$ at which a loss of $L^p$-uniqueness occurs, see Theorem \ref{T:esa} in Section \ref{S:esa}. This includes the discussion about the essential self-adjointness of $\mathcal{L}|_{\mathcal{A}(M\setminus \Sigma)}$ as the special case $p=2$. Under the assumption that the associated Markov semigroup $(P_t)_{t>0}$ is strong Feller, we then introduce a different well-known type of $(2,p)$-capacities $\cpct_{2,p}$ in Section \ref{S:second}, now based on associated resolvent operators $G_\lambda$, \cite[Section 2.3]{AH96}. These capacities are more suitable to discuss connections to Hausdorff measures and dimensions later on, so it is desirable to show they are equivalent to the capacities $\ccpct_{2,p}^\mathcal{A}$. The dominance of $\ccpct_{2,p}^\mathcal{A}$ over $\cpct_{2,p}$ is easy to see, Corollary \ref{C:firstcomparison}, but the opposite inequality is not at all automatic. It can be proved (Corollary \ref{C:secondcomparison}) if one has a truncation property on the level of the operator domain $\mathcal{D}(\mathcal{L})$; for $p=2$ it is of form 
\[\left\|F\circ G_\lambda f\right\|_{\mathcal{D}(\mathcal{L})}\leq c\:\left\|f\right\|_{L^2(M)},\quad f\in C_c(M),\quad f\geq 0,\]
where $F$ is a suitable function in $C^2(\mathbb{R})$. In Theorem \ref{T:Mazja} in Section \ref{S:second} we establish such a truncation property under the assumption that the semigroup satisfies 
\[\sqrt{\Gamma(P_t f)}\leq \frac{c_1\:e^{c_2t}}{\sqrt{t}}\:P_{\alpha t} f\] 
for $f$ as above and with $1\leq \alpha<2$; we refer to this condition as \eqref{E:log}. Rewritten in terms of heat kernels, it is seen to be a kind of \emph{logarithmic gradient estimate}, see \eqref{E:loghk} in Section \ref{S:second}. The truncation result in Theorem \ref{T:Mazja} is not at all trivial - recall that even in the Euclidean case Sobolev spaces $W^{2,p}$ are not stable under compositions with smooth bounded functions, \cite[Theorem 3.3.2]{AH96}. It is a partial generalization of a well-known truncation property for Bessel-potentials, \cite[Theorem 3.3.3]{AH96}, which follows by a method due to Maz'ya, \cite{Maz'ya72}, combined with a multiplicative estimate for derivatives of potentials in terms of maximal functions. 
The proof for the Euclidean case employs gradient estimates for resolvents and the Hardy-Littlewood maximal inequality, see \cite[Lemma 1]{A76}, \cite[Proposition 3.1.8]{AH96}, \cite{Hedberg72}. For manifolds and metric measure spaces estimates for semigroups or heat kernels and their gradients are well-studied and widely used. This motivates us to formulate a proof of Theorem \ref{T:Mazja} using \eqref{E:log} and the semigroup maximal inequality, see Section \ref{S:truncation}. In Sections \ref{S:Hausdorfftocap} and \ref{S:captoHausdorff} we connect the capacities $\cpct_{2,p}$
to Hausdorff measures and dimensions, provided that volume measure and resolvent densities admit suitable asymptotics respectively estimates, Lemma \ref{L:resolventlower} and Lemma \ref{L:resolventupper}. We provide applications to complete Riemannian manifolds $M$ in Section \ref{S:manifolds}; in this case $\mathcal{A}=C_c^\infty(M)$ is a natural choice. For $p=2$ a characterization of the critical size of $\Sigma$ in terms of $\ccpct_{2,2}^{C_c^\infty(M)}$ is valid without further assumptions, Theorem \ref{T:mfdccpct}. For $p\neq 2$ and for characterizations in terms of the capacities $\cpct_{2,p}$ we assume volume doubling and a gradient estimate, see Theorems \ref{T:mfdcpct}, \ref{T:Lpuni} and \ref{T:mfdcpctp}. For a characterization of essential self-adjointness in terms of Hausdorff measures we can again drop all additional assumptions, see Theorem \ref{T:Hausdorff}. This result is proved using localizations to small enough balls, and as mentioned before, it generalizes the former results in \cite{CdV82} and \cite{Masamune1999} in the sense that $\Sigma$ may now be an arbitrary closed set. 
In the case of manifolds of dimension $4$ the size of $\Sigma$ on a logarithmic scale determines whether essential self-adjointness is lost or not, and one can find positive and negative examples, see Remark \ref{R:mayormaynot}. Applications to sub-Riemannian manifolds are discussed in Section \ref{S:subRiemann}, see Theorems \ref{T:mfdccpctsub} and \ref{T:Carnot}, links to \cite{ABFP21} are pointed out in Remark \ref{R:Heisenberg}. In Section \ref{S:RCD} we discuss the natural Laplacian $\mathcal{L}$ on RCD$^*(K,N)$ spaces and its restriction to an abstract core $\mathcal{A}$, see \eqref{E:alternativeA1} and Proposition \ref{P:Aisnice}. We obtain characterizations of the critical size of $\Sigma$ for essential self-adjointness and $L^p$-uniqueness in terms of capacities, Theorem \ref{T:RCDLpuni}, and for the case of (Ahlfors) regular measures or {\rm CAT}$(0)$-spaces also in terms of Hausdorff measures, Theorems \ref{T:HausdorffRCD} and \ref{T:HausdorffRCDCAT}. An interesting application appears for Gromov-Hausdorff limits of non-collapsing manifolds, \cite{ChC97,ChN15}: Since by the results in \cite{JN21} the singular part $\mathcal{S}$ of the limit space $M$ is known to have finite Hausdorff measure of codimension four, the operator $\mathcal{L}|_{\mathcal{A}_c( \mathcal{R})}$ on the regular part $\mathcal{R}=M\setminus \mathcal{S}$ is essentially self-adjoint with unique self-adjoint extension $\mathcal{L}$. Since $\mathcal{L}|_{\mathcal{A}_c( \mathcal{R})}$ extends the 
classical Laplacian $\Delta_\mu|_{C_c^\infty(\mathcal{R})}$ on the regular part and $C_c^\infty(\mathcal{R})$ is seen to be dense in $\mathcal{A}_c( \mathcal{R})$ with respect to the graph norm, it follows that also $\Delta_\mu|_{C_c^\infty(\mathcal{R})}$ is essentially self adjoint on $L^2(M)$ with unique self-adjoint extension $\mathcal{L}$.

\section*{Acknowledgements} 

We thank Hiroaki Aikawa, Maria Gordina, Alexander Grigor'yan, Wolfhard Hansen, Seunghyun Kang, Martin Kolb, Takashi Kumagai and Qi S. Zhang for helpful discussions and comments.

\section{Basic setup and notation}\label{S:Setup}

Let $(M,\varrho)$ be a locally compact separable metric space and let $\mu$ a nonnegative Radon measure on $M$ with full support. We write $L^p(M):=L^p(M,\mu)$, $1\leq p \leq +\infty$ for the $L^p$-spaces of $\mu$-classes of $p$-integrable functions on $M$ with respect to $\mu$ and similarly, $L^p(A):=L^p(A,\mu|_{A})$ if $A$ is a Borel subset of $M$. 

Let $(\mathcal{L},\mathcal{D}(\mathcal{L}))$ be a non-positive definite densely defined self-adjoint operator on $L^2(M)$ and let $(\mathcal{E},\mathcal{D}(\mathcal{E}))$ be its quadratic form, i.e. the unique densely defined closed quadratic form on $L^2(M)$ satisfying
\begin{equation}\label{E:GG}
\mathcal{E}(f,g)=-\left\langle \mathcal{L}f,g\right\rangle_{L^2(M)}, \quad f\in\mathcal{D}(\mathcal{L}),\ g\in\mathcal{D}(\mathcal{E}).
\end{equation}
Endowed with the norm
\[\left\| f\right\|_{\mathcal{D}(\mathcal{E})}:=\Big(\mathcal{E}(f,f)+\left\| f\right\|_{L^2(M)}^2\Big)^{1/2}\]
the form domain $\mathcal{D}(\mathcal{E})$ is a Hilbert space. Given $\lambda>0$, we equip the operator domain $\mathcal{D}(\mathcal{L})$ with the Hilbert space norm 
\begin{equation}\label{E:DLnorm}
\left\|f\right\|_{\mathcal{D}(\mathcal{L})}:=\left\|(\lambda-\mathcal{L})f\right\|_{L^2(M)}.
\end{equation}
From (\ref{E:GG}) and the Cauchy-Schwarz inequality for $\mathcal{E}$ it is immediate that this norm is equivalent to the graph norm of $\mathcal{L}$. The parameter $\lambda$ will always remain fixed, we therefore suppress it from notation; suitable choices will be addressed later.

We also make use of the variational definition for $\mathcal{L}$. Let $(\mathcal{D}(\mathcal{E}))^\ast$ be the topological dual of $\mathcal{D}(\mathcal{E})$. For any $f\in\mathcal{D}(\mathcal{E})$ we can define $\mathcal{L}f$ as a member of $(\mathcal{D}(\mathcal{E}))^\ast$ by 
\begin{equation}\label{E:vardef}
\mathcal{L}f(g):=-\mathcal{E}(f,g),\quad g\in \mathcal{D}(\mathcal{E}).
\end{equation}
We then observe by a simple application of the Riesz representation theorem and the density of $\mathcal{D}(\mathcal{E})$ in $L^2(M)$ that
\begin{equation}\label{E:domaininfo}
\mathcal{D}(\mathcal{L})=\left\lbrace f\in\mathcal{D}(\mathcal{E}): \mathcal{L}f\in L^2(M)\right\rbrace.
\end{equation}

We assume that $(\mathcal{E},\mathcal{D}(\mathcal{E}))$ is a Dirichlet form on $L^2(M)$, \cite[Section 3.2]{FOT94}, and that it admits a carr\'e du champ, in other words, that there is a bilinear nonnegative definite map $\Gamma$ from $\mathcal{D}(\mathcal{E})\cap L^\infty(M)\times\mathcal{D}(\mathcal{E})\cap L^\infty(M)$ into $ L^1(M)$ such that 
\[\frac12\left\lbrace\mathcal{E}(fh,g)+\mathcal{E}(f,gh)-\mathcal{E}(fg,h)\right\rbrace=\int_M h\:\Gamma(f,g)\:d\mu,\quad f,g,h\in\mathcal{D}(\mathcal{E})\cap L^\infty(M),\]
\cite[Chapter I, Definition 4.1.2 and Theorem 4.2.1]{BH91}; recall that by the Markov property the space $
\mathcal{D}(\mathcal{E})\cap L^\infty(M)$ is an algebra, \cite[Chapter I, Corollary 3.3.2]{BH91}. Using truncations and approximation one can naturally extend $\Gamma$ to a bilinear map from $\mathcal{D}(\mathcal{E})\times\mathcal{D}(\mathcal{E})$ into $L^1(M)$.

We further assume that  $(\mathcal{E},\mathcal{D}(\mathcal{E}))$ is regular and strongly local, \cite[Section 3.2]{FOT94}. One then refers to its generator $(\mathcal{L},\mathcal{D}(\mathcal{L}))$ as \emph{symmetric diffusion operator}. Strong locality implies that if $f\in\mathcal{D}(\mathcal{L})$ is constant on an open set $U\subset M$ then $\mathcal{L}f=0$ $\mu$-a.e. on $U$. It also implies that if $f\in\mathcal{D}(\mathcal{E})\cap C_c(M)$ and its support $\supp f$ is contained in an open set $U$, then $\Gamma(f,f)=0$ $\mu$-a.e. on $U^c$. In particular, we have $\mathcal{E}(f,g)=\int_M\Gamma(f,g)\:d\mu$ for all $f,g\in \mathcal{D}(\mathcal{E})\cap C_c(M)$. Moreover, by the Markov property and strong locality, $\mathcal{D}(\mathcal{E})$ is stable under taking compositions $F(f_1,...f_n)$ of elements $f_i\in\mathcal{D}(\mathcal{E})$ with functions $F\in C^1(\mathbb{R}^n)$ satisfying $F(0)=0$ and having uniformly bounded first derivatives. The carr\'e obeys the chain rule
\begin{equation}\label{E:chain}
\Gamma(F(f_1, ..., f_n),g)=\sum_{i=1}^n \frac{\partial F}{\partial x_i}(f_1,...,f_n)\Gamma(f_i,g)\quad \text{$\mu$-a.e.}
\end{equation}
for any $f_1,...,f_n,g\in \mathcal{D}(\mathcal{E})$ and $F$ as stated, see \cite[Chapter I, Proposition 3.3.1 and Corollary 6.1.3]{BH91} or \cite[Theorem 3.2.2]{FOT94}. For $n=1$ this remains true for Lipschitz $F$ with $F(0)=0$, see \cite[Chapter I, Corollary 7.1.2]{BH91}.

Let $(P_t)_{t>0}$ be the unique symmetric Markov semigroup generated by $(\mathcal{L},\mathcal{D}(\mathcal{L}))$, \cite{BH91, Da89, FOT94}, also referred to as \emph{symmetric diffusion semigroup}. 
The restriction of $(P_t)_{t>0}$ to $L^1(M)\cap L^\infty(M)$ extends to a contraction semigroup $(P^{(p)}_t)_{t>0}$ on each $L^p(M)$, $1\leq p\leq +\infty$, strongly continuous for $1\leq p<+\infty$, \cite[Theorem 1.4.1]{Da89}. Clearly $P_t^{(2)}=P_t$.
For any $1\leq p<+\infty$ the generator $(\mathcal{L}^{(p)},\mathcal{D}(\mathcal{L}^{(p)}))$ of $(P^{(p)}_t)_{t>0}$
 on $L^p(M)$ is the smallest closed extension of the restriction of $\mathcal{L}$ to the a priori domain
 \begin{equation}\label{E:aprioridomain}
\mathcal{D}_p:= \left\lbrace f\in\mathcal{D}(\mathcal{L})\cap L^p(M): \mathcal{L}f\in L^p(M)\right\rbrace .
\end{equation} 
Clearly $\mathcal{L}^{(2)}=\mathcal{L}$. Also for $p\neq 2$ we endow each $\mathcal{D}(\mathcal{L}^{(p)})$ with the norm 
\begin{equation}\label{E:DLpnorm}
\left\|f\right\|_{\mathcal{D}(\mathcal{L}^{(p)})}:=\left\|(\lambda-\mathcal{L}^{(p)})f\right\|_{L^p(M)}.
\end{equation}

The space $\mathcal{D}(\mathcal{L}^{(1)})\cap L^\infty(M)$ is an algebra, and we have 
\begin{equation}\label{E:Gamma}
\Gamma(f,g)=\frac12 \left\lbrace  \mathcal{L}^{(1)}(fg)-f\mathcal{L}^{(1)}g-g\mathcal{L}^{(1)}f \right\rbrace
\end{equation}
for any $f,g\in \mathcal{D}(\mathcal{L}^{(1)})\cap L^\infty(M)$, seen as an $L^1(M)$-identity, \cite[Chapter I, Theorem 4.2.1]{BH91}. 

If $\mathcal{A}$ is a vector space of real-valued Borel functions on $M$ and $U\subset M$ is an open set, then we write $\mathcal{A}(U)$ to denote the subspace of $\mathcal{A}$ consisting of functions with support contained in $U$. Clearly $\mathcal{A}(M)=\mathcal{A}$.  We write $\mathcal{A}_c$ and $\mathcal{A}_c(U)$ for the subspaces of $\mathcal{A}$ respectively $\mathcal{A}(U)$ consisting of compactly supported functions.

We write $L_+^p(M)$ for the cone of nonnegative elements in $L^p(M)$. By $L^0(M)$ we denote the space of $\mu$-equivalence classes of Borel functions on $M$ and by $\mathcal{B}(M)$ (respectively $\mathcal{B}_b(M)$) the space of Borel functions (respectively bounded Borel functions) on $M$. If $\mathcal{S}$ is a vector space of $\mu$-classes of functions on $M$, we write $f\in \mathcal{B}(M)\cap \mathcal{S}$ (respectively $f\in \mathcal{B}_b(M)\cap \mathcal{S}$) to say that the $\mu$-class of $f$ is in $\mathcal{S}$. Set inclusions and other statements involving functions and classes are silently understood in a similar manner. We use the shortcut notation $\Gamma(f):=\Gamma(f,f)$, similarly for other symmetric bilinear quantities.

\section{Capacities based on spaces of functions}\label{S:caps}

Let $\mathcal{A}$ be a vector space of real-valued Borel functions on $M$. Given a compact set $K\subset M$ we write $\omega_K^{\mathcal{A}}$ for set of functions $u \in \mathcal{A}$ such that $u=1$ on an open neighborhood of $K$. 

Suppose that $1<p<\infty$ and $\mathcal{A}\subset \mathcal{D}(\mathcal{L}^{(p)})$. For any compact set $K\subset M$ we define the \emph{$(2,p)$-capacity $\ccpct_{2,p}^{\mathcal{A}}(K)$ of $K$ with respect to $\mathcal{A}$} by 
\begin{equation}\label{E:cap}
\ccpct_{2,p}^\mathcal{A}(K) := \inf\big\lbrace \|u \|_{\mathcal{D}(\mathcal{L}^{(p)})}^p: u \in \omega_K^{\mathcal{A}}\big\rbrace
\end{equation}
with $\ccpct_{2,p}^\mathcal{A}(K):=+\infty$ if $\omega_K^{\mathcal{A}}=\emptyset$. For general sets $E\subset M$ we then set
\begin{equation}\label{E:innerreg}
\ccpct_{2,p}^{\mathcal{A}}(E)=\sup\left\lbrace \ccpct_{2,p}^{\mathcal{A}}(K): K\subset E, \text{ $K\subset M$ compact}\right\rbrace.
\end{equation}
Early references on definition (\ref{E:cap}) on Euclidean spaces are \cite{HarveyPolking73, Maz'ya63, Maz'ya70, Maz'ya73, Maz'yabook}, detailed comments can be found in \cite[Section 2.9]{AH96}. In comparison with \cite{HarveyPolking73} identity 
\eqref{E:cap} means that we use \cite[Theorem 2.1]{HarveyPolking73} as a definition; the identity in \cite[Definition 1.1]{HarveyPolking73} then follows as a result, see Proposition \ref{P:propertiesadd}. Only Proposition \ref{P:properties} (i) and condition \eqref{A:bump} below will be used in the following sections. Proposition \ref{P:properties} (ii) and Proposition \ref{P:propertiesadd} are stated to make clear that \eqref{E:cap} and \eqref{E:innerreg} are in line with \cite{HarveyPolking73}. 

\begin{proposition}\label{P:properties} Suppose that $1<p<\infty$ and $\mathcal{A}\subset \mathcal{D}(\mathcal{L}^{(p)})$.
\begin{enumerate}
\item[(i)] If $E_1\subset E_2\subset M$ then $\ccpct_{2,p}^\mathcal{A}(E_1)\leq \ccpct_{2,p}^\mathcal{A}(E_2)$.
\item[(ii)] For any $K\subset M$ compact we have $\ccpct_{2,p}^{\mathcal{A}}(K)=\inf\left\lbrace \ccpct_{2,p}^{\mathcal{A}}(G): K\subset G, \text{ $G$ \upshape{open}}\right\rbrace$.
\end{enumerate}
\end{proposition}

\begin{proof}
For compact sets (i) is obvious from (\ref{E:cap}), and by (\ref{E:innerreg}) it extends to general sets. Statement (ii) follows as \cite[Proposition 2.2.3]{AH96}: Suppose $K$ is compact and $\varepsilon>0$. By (i) we may assume that $\ccpct_{2,p}^\mathcal{A}(K)<+\infty$. Then there is some $u\in \mathcal{A}$ such that $\left\|u\right\|_{\mathcal{D}(\mathcal{L}^{(p)})}^p<\ccpct_{2,p}^\mathcal{A}(K)+\varepsilon$ and $u=1$ on an open neighborhood $U$ of $K$. Let $G$ be a relatively compact open neighborhood of $K$ such that $\overline{G}\subset U$. Clearly $u\in\omega_{\overline{G}}^\mathcal{A}$, so that $\ccpct_{2,p}^\mathcal{A}(G)\leq \ccpct_{2,p}^\mathcal{A}(\overline{G})\leq  \left\|u\right\|_{\mathcal{D}(\mathcal{L}^{(p)})}^p$.
Thus, we obtain $\ccpct_{2,p}^\mathcal{A}(K)\leq \ccpct_{2,p}^\mathcal{A}(G)\leq \ccpct_{2,p}^\mathcal{A}(K)+\varepsilon$. 
\end{proof}

We say that condition \eqref{A:bump} is satisfied if 
\begin{multline}\label{A:bump}\tag{B}
\hspace{50pt} \text{For any compact $K\subset M$ the set $\omega_K^{\mathcal{A}}$ is nonempty.}
\end{multline}
In the manifold case with $\mathcal{A}=C_c^\infty(M)$ it is implied by the existence of smooth bump functions. Condition \eqref{A:bump} will be used frequently throughout the later sections, in the present section it is used only for Proposition \ref{P:propertiesadd} below.

Suppose that $1< p<+\infty$ is fixed and $\mathcal{A}\subset \mathcal{D}(\mathcal{L}^{(p)})$. Given $f\in L^q(M)$, $\frac{1}{p}+\frac{1}{q}=1$, we define $(\lambda-\mathcal{L}^{(q)})f$ as a linear functional on $\mathcal{A}$ by 
\[(\lambda-\mathcal{L}^{(q)})f(g):=\int_M f (\lambda-\mathcal{L}^{(p)})g\:d\mu,\quad g\in\mathcal{A},\]
and, mimicking classical definitions in the theory of Schwarz distributions, define the \emph{support $\supp_\mathcal{A} (\lambda-\mathcal{L}^{(q)})f$ of $(\lambda-\mathcal{L}^{(q)})f$ with respect to $\mathcal{A}$} as the set of all $x\in M$ with the property that for any open neighborhood $U_x$ of $x$ there is some $g\in \mathcal{A}(U_x)$ such that $(\lambda-\mathcal{L}^{(q)}f)(g)\neq 0$. The set $\supp_\mathcal{A} (\lambda-\mathcal{L}^{(q)})f$ is seen to be closed. If \eqref{A:bump} is satisfied, $f\in L^q(M)$ and $\supp_\mathcal{A} (\lambda-\mathcal{L}^{(q)})f$ is compact, then we can define $((\lambda-\mathcal{L}^{(q)})f)(\mathbf{1}):=((\lambda-\mathcal{L}^{(q)})f)(g)$ with an arbitrary function $g\in \mathcal{A}$ satisfying $g= 1$ on an open neighborhood of $\supp_\mathcal{A} (\lambda-\mathcal{L}^{(q)})f$. The following observation reconnects to \cite[Definition 1.1]{HarveyPolking73}.

\begin{proposition}\label{P:propertiesadd}
Suppose that condition \eqref{A:bump} is satisfied, $1<p<\infty$ and $\mathcal{A}\subset \mathcal{D}(\mathcal{L}^{(p)})$.
Then 
\begin{multline}\label{E:claimedidentity}
\ccpct_{2,p}^\mathcal{A}(E)^{1/p}=\sup\big\lbrace |((\lambda-\mathcal{L}^{(q)})f)(\mathbf{1})|:\ \text{$f\in L^q(M)$, $\left\|f\right\|_{L^q(M)}\leq 1$,}\\
\text{$\supp_\mathcal{A} (\lambda-\mathcal{L}^{(q)})f$ compact and contained in $E$}\big\rbrace
\end{multline}
for any $E\subset M$, where $\frac{1}{p}+\frac{1}{q}=1$.
\end{proposition}

\begin{proof}
We can follow \cite[Theorem 2.1]{HarveyPolking73}: If $c(E)$ denotes the right hand side of (\ref{E:claimedidentity}), then 
\begin{equation}\label{E:compactssuffice}
c(E)=\sup\left\lbrace c(K):\quad \text{$K\subset E$, $K$ \upshape{compact}}\right\rbrace, \quad E\subset M.
\end{equation}
Since obviously $c(E)\geq c(K)$ for any compact $K\subset E$, the inequality $\geq$ in \eqref{E:compactssuffice} is clear. On the other hand, for any $\varepsilon>0$ we can find $f$ as in \eqref{E:claimedidentity} such that $c(E)\leq |((\lambda-\mathcal{L}^{(q)})f)(\mathbf{1})|+\varepsilon$. Since $K:=\supp_\mathcal{A} (\lambda-\mathcal{L}^{(q)})f$ itself is compact, the preceding is bounded by $c(K)+\varepsilon$ and in particular, by the supremum in \eqref{E:compactssuffice} plus $\varepsilon$. Letting $\varepsilon$ go to zero gives $\leq$ in \eqref{E:compactssuffice}. Consequently it suffices to verify (\ref{E:claimedidentity}) for compact sets $K$ in place of $E$. 

Let $K$ be a compact set. If $f\in L^q(M)$, $\left\|f\right\|_{L^q(M)}\leq 1$ and $\supp_\mathcal{A} (\lambda-\mathcal{L}^{(q)})f\subset K$, then for any $g\in \omega_K^\mathcal{A}$ we have 
\[|\left\langle (\lambda-\mathcal{L}^{(q)})f,\mathbf{1}\right\rangle|=|\left\langle (\lambda-\mathcal{L}^{(q)})f,g\right\rangle|=|\left\langle f, (\lambda-\mathcal{L}^{(p)})g\right\rangle|\leq \left\|(\lambda-\mathcal{L}^{(p)})g\right\|_{L^p(M)}\]
by the H\"older inequality, and therefore $c(K)\leq \ccpct_{2,p}^\mathcal{A}(K)^{1/p}$. Now suppose $h\in \omega^\mathcal{A}_K$. By the Hahn-Banach theorem, there is some $f\in L^q(M)$ with $\left\|f\right\|_{L^q(M)}\leq 1$ such that 
\begin{equation}\label{E:annihilate}
\left\langle f,(\lambda-\mathcal{L}^{(p)})\varphi\right\rangle=0,\quad \varphi \in \mathcal{A}(M\setminus K),
\end{equation}
and
\begin{equation}\label{E:distance}
\left\langle f,(\lambda-\mathcal{L}^{(p)})h \right\rangle=\inf\big\lbrace\left\|(\lambda-\mathcal{L}^{(p)})(h-\varphi)\right\|_{L^p(M)}:\quad \varphi \in \mathcal{A}(M\setminus K)\big\rbrace =\ccpct_{2,p}^\mathcal{A}(K)^{1/p},
\end{equation}
see \cite[Corollary 3 of Theorem III.6]{RS80v1}. From (\ref{E:annihilate}) it follows that $\supp_\mathcal{A} (\lambda-\mathcal{L}^{(q)})f\subset K$, and with (\ref{E:distance}) we arrive at $\ccpct_{2,p}^\mathcal{A}(K)^{1/p}=\left\langle (\lambda-\mathcal{L}^{(q)})f,\mathbf{1}\right\rangle\leq c(K)$.
\end{proof}

\begin{remark} 
The capacity $\ccpct_{2,p}^\mathcal{A}$ is not expected to be a Choquet capacity, see \cite[p.184]{HarveyPolking73}.
\end{remark}

\section{$L^p$-uniqueness and removable sets}\label{S:esa}

Let $(\mathcal{L}_0,\mathcal{A}_0)$ be a linear operator on $L^p(M)$. We call it \emph{$L^p$-unique} if its domain $\mathcal{A}_0$ is dense in $L^p(M)$ and there is at most one strongly continuous semigroup on $L^p(M)$ whose generator extends  $(\mathcal{L}_0,\mathcal{A}_0)$. See \cite[Chapter I, Definition 1.3]{Eberle99}. If $p=2$ and $(\mathcal{L}_0,\mathcal{A}_0)$ is symmetric and semibounded, then it is $L^2$-unique if and only if $\mathcal{A}_0$ is dense in $L^2(M)$ and $(\mathcal{L}_0,\mathcal{A}_0)$ is essentially self-adjoint, \cite[Chapter I, Corollary 1.2]{Eberle99}. 

In the sequel we assume that the $L^p$-uniqueness holds for the restrictions $\mathcal{L}^{(p)}|_\mathcal{A}$ of the globally defined operators $\mathcal{L}^{(p)}$ from Section \ref{S:Setup} to a given space of real-valued Borel functions $\mathcal{A}$ and investigate whether the removal of a small closed subset $\Sigma$ of $M$ leads to a loss of $L^p$-uniqueness or not. To prepare the discussion we formulate structural conditions on $\mathcal{A}$.

The first condition guarantees certain boundedness and multiplication properties:
\begin{multline}\label{A:basic}\tag{$\mathrm{L^\infty}$}
\text{The space $\mathcal{A}$ is a subalgebra of $\mathcal{B}_b(M)\cap L^1(M)$, contained in $\mathcal{D}(\mathcal{L}^{(1)})$,}\\
\text{and such that $\Gamma(f)\in L^\infty(M)$ and $\mathcal{L}f\in L^\infty(M)$ for all $f\in \mathcal{A}$.}
\end{multline}

\begin{remark} If \eqref{A:basic} holds and $\mathcal{A}\subset \mathcal{D}(\mathcal{L})$, then $\mathcal{A}\subset \mathcal{D}_p$ for all $1\leq  p<+\infty$.
\end{remark}

Given $1<p<+\infty$ we consider the following condition.
\begin{equation}\label{A:basicp}\tag{$\mathrm{C_p}$}
\text{The space $\mathcal{A}$ is contained in $\mathcal{D}(\mathcal{L}^{(p)})$ and the operator $\mathcal{L}^{p}|_\mathcal{A}$ is $L^p$-unique.}
\end{equation}
Recall that a subspace of the domain of a closed operator is said to be a \emph{core} if the closure of the restriction of the operator to this subspace agrees with the initially given closed operator. It is well-known that \eqref{A:basicp} is equivalent to saying that $\mathcal{A}$ is a core for $(\mathcal{L}^p, \mathcal{D}(\mathcal{L}^p))$, \cite[Chapter I, Appendix A, Theorem 1.2]{Eberle99}. If \eqref{A:basicp} holds, then the closure of 
$\mathcal{L}^{p}|_\mathcal{A}$ is $(\mathcal{L}^{(p)},\mathcal{D}(\mathcal{L}^{(p)}))$; in the special case $p=2$ it follows that $(\mathcal{L},\mathcal{D}(\mathcal{L}))$ with $\mathcal{D}(\mathcal{L})=\mathcal{D}(\mathcal{E})\cap \mathcal{D}((\mathcal{L}|_{\mathcal{A}})^\ast)$ is the unique self-adjoint extension of $\mathcal{L}|_\mathcal{A}$.

\begin{remark}\label{R:compactsuppsdense} If \eqref{A:basicp} holds and $\mathcal{A}_c$ is dense in $\mathcal{A}$ with respect to $\big\|\cdot\big\|_{\mathcal{D}(\mathcal{L}^{(p)})}$, then also $\mathcal{L}^{p}|_{\mathcal{A}_c}$ is $L^p$-unique with closure $(\mathcal{L}^{(p)},\mathcal{D}(\mathcal{L}^{(p)}))$.
\end{remark}

We formulate yet another condition needed for $p\neq 2$; for $p=2$ it is always satisfied.
\begin{equation}\label{A:Riesztrafo}\tag{$\mathrm{\Gamma_p}$}
\text{There is a constant $c(p)>0$ such that $\left\|\Gamma(f)^{1/2}\right\|_{L^p(M)}\leq c(p)\:\left\|f\right\|_{\mathcal{D}(\mathcal{L}^{(p)})}$ for all $f\in\mathcal{A}$.}
\end{equation}

Now suppose that $\Sigma\subset M$ is a closed set. We write $\mathring{M}:= M \setminus \Sigma$. For elliptic operators $\mathcal{L}$ on Euclidean spaces it is well-known that the $L^p$-uniqueness  of $\mathcal{L}|_{C_c^\infty(\mathring{M})}$ can be characterized in terms of the $(2,p)$-capacity of $\Sigma$, see for instance \cite{HarveyPolking73, Maz'ya75, Maz'yaKhavin} and the references listed in \cite[Section 2.9]{AH96}. The following theorem is a general version of this fact, applicable to manifolds and metric measure spaces. 

\begin{theorem}\label{T:esa} Let $1<p<+\infty$ and assume that condition \eqref{A:basicp} holds. 
\begin{enumerate}
\item[(i)] Suppose that also \eqref{A:bump} holds. If $\Sigma\subset M$ is closed, $\mu(\Sigma)=0$ and $\mathcal{L}|_{\mathcal{A}(\mathring{M})}$ is $L^p$-unique, then we have $\ccpct_{p,2}^\mathcal{A}(\Sigma)=0$. 
\item[(ii)] Suppose that also \eqref{A:basic} and \eqref{A:Riesztrafo} hold. If $\Sigma\subset M$ is compact and $\ccpct_{p,2}^\mathcal{A}(\Sigma)=0$, then we have $\mu(\Sigma)=0$ and $\mathcal{L}|_{\mathcal{A}(\mathring{M})}$ is $L^p$-unique with closure $(\mathcal{L}^{(p)}, \mathcal{D}(\mathcal{L}^{(p)}))$. If $\mathcal{A}_c$ is dense in $\mathcal{A}$ with respect to $\|\cdot\|_{\mathcal{D}(\mathcal{L}^{(p)})}$, then 
the conclusion remains true for general closed $\Sigma\subset M$; in this case also $\mathcal{L}|_{\mathcal{A}_c(\mathring{M})}$ is $L^p$-unique with closure $(\mathcal{L}^{(p)}, \mathcal{D}(\mathcal{L}^{(p)}))$.
\end{enumerate}
\end{theorem}

\begin{proof}
We verify (i). Since the operator $\mathcal{L}|_{\mathcal{A}(\mathring{M})}$ is densely defined and closable in $L^p(M)$,  and its smallest closed extension coincides with $(\mathcal{L}^{(p)},\mathcal{D}(\mathcal{L}^{(p)}))$, the adjoint $L^\ast$ of $L:=(\mathcal{L}|_{\mathcal{A}(\mathring{M})})^\ast$ equals $\mathcal{L}^{(p)}$, \cite[Chapter III, Theorems 5.28 and 5.29]{Kato80}. Let $(\Sigma_i)_{i\geq 1}$ be a sequence of compact sets $\Sigma_i\subset \Sigma$ such that
\begin{equation}\label{E:fromwithin}
\ccpct_{2,p}^\mathcal{A}(\Sigma)=\sup_i \ccpct_{2,p}^\mathcal{A}(\Sigma_i),
\end{equation}
by (\ref{E:innerreg}) such a sequence exists. For fixed $i$ let $f_i \in \mathcal{A}$ be such that $f_i = 1$ on a neighbourhood of $\Sigma_i$, by condition \eqref{A:bump} such $f$ exists. Since $\mathcal{A}\subset \mathcal{D}(\mathcal{L}^{(p)})$ we can find $g\in L^p(M)$ such that 
$\left\langle Lh,f_i\right\rangle=\left\langle h, g\right\rangle$ for all $h\in \mathcal{D}(L)$. Because $\mu(\Sigma_i)=0$ we also have $\left\langle Lh,f_i|_{\mathring{M}}\right\rangle=\left\langle h, g\right\rangle$ for all $h\in \mathcal{D}(L)$, in other words, $f_i|_{\mathring{M}}$ is an element of $D(L^\ast)=\mathcal{D}(\mathcal{L}^{(p)})$. Accordingly there exists a sequence $(f_{i,n})_{n \ge 1} \subset \mathcal{A}(\mathring{M})$ such that $\lim_n\left\|f_i - f_{i,n}\right\|_{\mathcal{D}(\mathcal{L}^{(p)})}= 0$. The functions $e_{i,n} = f_i-f_{i,n}$ are elements of $\omega_{\Sigma_i}^\mathcal{A}$ and consequently $\ccpct_{2,p}^\mathcal{A}(\Sigma_i) \le \lim_n \|e_{i,n}\|_{\mathcal{D}(\mathcal{L}^{(p)})}^p= 0$. By \eqref{E:fromwithin} this implies $\ccpct_{2,p}^\mathcal{A}(\Sigma)=0$.

To see (ii), suppose that $\ccpct_{2,p}^\mathcal{A}(\Sigma)=0$. Then $\mu(\Sigma)=0$, as follows straightforward from the definitions, and therefore $L^p(M)=L^p(\mathring{M})$. Denote by $\mathcal{D}(\mathcal{L}^{\mathring{M}})$ the closure of $\mathcal{L}|_{\mathcal{A}(\mathring{M})}$ in $L^p(M)$.
It suffices to prove $\mathcal{D}(\mathcal{L}^{\mathring{M}})= \mathcal{D}(\mathcal{L}^{(p)})$. Since $\mathcal{A}(\mathring{M})\subset \mathcal{A}$ we have $\mathcal{D}(\mathcal{L}^{\mathring{M}})\subset \mathcal{D}(\mathcal{L}^{(p)})$, and it remains to show that
\begin{equation}\label{E:tobechecked}
\mathcal{D}(\mathcal{L}^{\mathring{M}})\supset \mathcal{D}(\mathcal{L}^{(p)}).
\end{equation}
For any $f \in \mathcal{D}(\mathcal{L}^{(p)})$ let $(f_n)_{n\geq 1} \subset \mathcal{A}$ such that $f_n \to f$ in $\mathcal{D}(\mathcal{L}^{(p)})$, if $\Sigma$ is noncompact, then we may assume the $f_n$ have compact support. For any $n$ the set $K_n:=\Sigma\cap \supp f_n$ is compact and satisfies $\ccpct_{2,p}^\mathcal{A}(K_n)\leq \ccpct_{2,p}^\mathcal{A}(\Sigma)=0$. Accordingly we can find a sequence $(e_{n,l})_{l \ge 1} \subset \omega_{K_n}^\mathcal{A}$ such that $e_{n,l} \to 0$ in $\mathcal{D}(\mathcal{L}^{(p)})$ as $l \to \infty$.
Set $f_{n,l}=(1-e_{n,l})f_n$. Then $f_{n,l}\in  \mathcal{A}$ by \eqref{A:basic}, and since $(1-e_{n,l})=0$ on a neighborhood of $K_n$ and $f_n=0$ on $M\setminus \supp f_n$, it follows that $f_{n,l}\in \mathcal{A}(\mathring{M})$. We have 
\[\| f_{n,l} - f_n \|_{L^p(M)} = \|e_{n,l}f_n\|_{L^p(M)} \le  \|e_{n,l}\|_{L^p(M)}  \|f_n\|_{L^\infty(M)} ,\]
what goes to zero as $l\to\infty$. Moreover, using (\ref{E:Gamma}) we see that 
\begin{align*}
\|&\mathcal{L} f_{n,l} - \mathcal{L} f_n  \|_{L^p(M)} \notag\\
& =  \| (\mathcal{L}^{(1)}( e_{n,l} f_n)\|_{L^p(M)}\notag\\
& \le \| (\mathcal{L}^{(1)} e_{n,l}) f_n\|_{L^p(M)} + 2\left(\int_M |\Gamma(e_{n,l}, f_n)|^p d\mu\right)^{1/p} + \| e_{n,l} \mathcal{L}^{(1)} f_n \|_{L^p(M)} \\
&\le \| \mathcal{L} e_{n,l} \|_{L^p(M)}\|f_n\|_{L^\infty(M)}  + 2  \| \Gamma(e_{n,l})^{1/2} \|_{L^p(M)} \|\Gamma( f_n)\|_{L^\infty(M)}^{1/p}+  \| e_{n,l} \|_{L^p(M)} \|\mathcal{L}f_n \|_{L^\infty(M)},
\end{align*}
what converges to zero as $l\to\infty$ by \eqref{A:basic}; for $p\neq 2$ we also use condition \eqref{A:Riesztrafo} on the second summand in the last line. Hence the functions $f_{n,l} \in \mathcal{A}(\mathring{M})$ approximate $f$ in $\mathcal{D}(\mathcal{L}^{(p)})$, what shows that $f \in D(\mathcal{L}^{\mathring{M}})$, and consequently (\ref{E:tobechecked}) holds. 
\end{proof}

The density of $\mathcal{A}_c$ in $\mathcal{A}$ with respect to the graph norm follows if there is a suitable 
approximation of the identity. We say that condition \eqref{A:approxid} holds if
\begin{align}\label{A:approxid}\tag{A}
&\text{There is a sequence $(h_n)_{n\geq 1}\subset \mathcal{A}_c$ such that $0\leq h_n\leq 1$, $h_n\uparrow 1$ as $n\to \infty$,}\\
&\hspace{30pt}\sup_n\big\|\Gamma(h_n)^{1/2}\big\|_{L^\infty(M)}<+\infty\quad \text{ and }\quad \sup_n\big\| \mathcal{L}h_n\big\|_{L^\infty(M)}<+\infty.\notag
\end{align}
We record the following observation for later use.

\begin{lemma}\label{L:compactsupps}
Assume that \eqref{A:basic} holds and $\mathcal{A}\subset \mathcal{D}(\mathcal{L})$. Let $1<p<+\infty$ and assume further that \eqref{A:Riesztrafo} holds and that $\mathcal{A}$ is dense in $L^q(M)$, $\frac{1}{p}+\frac{1}{q}=1$. Then \eqref{A:approxid} implies the density of $\mathcal{A}_c$ in $\mathcal{A}$ with respect to $\|\cdot\|_{\mathcal{D}(\mathcal{L}^{(p)})}$.
\end{lemma}

\begin{proof}
Let $g\in\mathcal{A}$. Then $h_ng\in \mathcal{A}_c$, and using \eqref{E:Gamma} we see that 
\begin{align}
\left\|(\lambda-\mathcal{L})(h_ng)\right\|_{L^p(M)}&\leq \lambda\left\|h_n\right\|_{L^\infty(M)}\left\|g\right\|_{L^p(M)}+\left\|\mathcal{L}(h_ng)\right\|_{L^p(M)}\notag\\
&\leq \lambda\left\|g\right\|_{L^p(M)}+\left\|\mathcal{L}h_n\right\|_{L^\infty(M)}\left\|g\right\|_{L^p(M)}\notag\\
&\hspace{50pt}+2\big\|\Gamma(h_n)^{1/2}\big\|_{L^\infty(M)}\left\|\Gamma(g)\right\|_{L^p(M)}+\left\|h_n\right\|_{L^\infty(M)}\left\|\mathcal{L}g\right\|_{L^p(M)}\notag\\
&\leq c\:\left\|g\right\|_{\mathcal{D}(\mathcal{L}^{(p)})}\notag
\end{align} 
with a constant $c>0$ independent of $g$ and $n$. In particular,
\[\sup_n \left\|(\lambda-\mathcal{L})(h_ng-g)\right\|_{L^p(M)}<+\infty,\]
so that by reflexivity and Banach-Alaoglu we can find a sequence $(n_k)_k$ and a function $g_0\in L^p(M)$ such that 
\[\lim_k \left\langle (\lambda-\mathcal{L})(h_{n_k}g-g)-g_0,f\right\rangle=0,\quad f\in L^q(M).\]
By Mazur's lemma we may assume that 
\[\lim_k \big\|\frac{1}{k}\sum_{j=1}^{n_k}(\lambda-\mathcal{L})(h_{n_k}g-g)-g_0\big\|_{L^p(M)}=0,\]
otherwise pass to a subsequence. On the other hand
\[\lim_k \left\langle (\lambda-\mathcal{L})(h_{n_k}g-g),f\right\rangle=\lim_k \left\langle (h_{n_k}g-g),(\lambda-\mathcal{L})f\right\rangle=0,\quad f\in \mathcal{A},\]
by the symmetry of $\mathcal{L}|_{\mathcal{A}}$ and dominated convergence, so that by the density of $\mathcal{A}$ in $L^q(M)$ we have $g_0=0$. Setting $g_k:=\frac{1}{k}\sum_{j=1}^{n_k}h_{n_k}g$ we obtain a sequence $(g_k)_k\subset \mathcal{A}_c$ such that $\lim_k g_k=g$ in $\mathcal{D}(\mathcal{L}^{(p)})$.
\end{proof}

\section{Capacities via resolvents, and a first comparison}

We recall another well-known definition of capacities and record a simple first comparison result for the two types of capacities. 

Recall that $(P_t)_{t>0}$ denotes the symmetric Markov semigroup generated by $(\mathcal{L},\mathcal{D}(\mathcal{L}))$. For any $\lambda >0$ we write $G_\lambda$ to denote the associated \emph{$\lambda$-resolvent operator}, defined by 
\begin{equation}\label{E:resolventop}
G_\lambda f:=\int_0^\infty e^{-\lambda t} P_tfdt
\end{equation}
for $f\in L^2(M)$. For any $1\leq p\leq +\infty$ the restriction of $G_\lambda$ to $L^1 (M)\cap L^\infty(M)$ extends to a bounded linear operator $G_\lambda^{(p)}:L^p(M)\to L^p(M)$, and for all $f\in L^p(M)$ an analog of (\ref{E:resolventop}) holds with  $G_\lambda^{(p)}$ and $(P_t^{(p)})_{t>0}$ in place of $G_\lambda$ and $(P_t)_{t>0}$. For any $1\leq p< +\infty$ we have $G_\lambda^{(p)}=(\lambda-\mathcal{L}^{(p)})^{-1}$. Since in the sequel the meaning will be clear from the context, we suppress the superscript $(p)$ from notation.

We say that $(P_t)_{t>0}$ is a \emph{strong Feller semigroup} if for any $t>0$ and any $f\in L^\infty(M)$ we have $P_t f\in C_b(M)$, where $C_b(M)$ is the space of continuous bounded functions on $M$. See e.g. \cite[Section V.2]{BliedtnerHansen86}. In the following we assume that $(P_t)_{t>0}$ is strong Feller. Then 
\begin{equation}\label{E:Markovkernel}
P_t(x,A):=P_t\mathbf{1}_A(x),\quad  t>0,\quad x\in M,\quad A\subset M\ \text{Borel},
\end{equation}
defines a (sub-)Markovian kernel $(P_t(x,dy))_{t>0}$, and we have $P_tf(x)=\int_M f(y)P_t(x,dy)$ for all $t>0$, $x\in M$ $f\in L^\infty(M)$. Clearly then also $G_\lambda f\in C_b(M)$ for all $f\in L^\infty(M)$. For any $f\in L_+^0(M)$ we can define $G_\lambda f$ as an element of $L_+^0(M)$ by (\ref{E:resolventop}) and taking limits of increasing sequences. The following is immediate.
\begin{proposition}\label{P:strongFeller}
For any $f\in L^0_+(M)$ and $\lambda >0$ the function $G_\lambda f$ is lower semicontinuous on $M$.
\end{proposition}

Standard definitions yield a second type of capacities, now based on resolvent operators associated with the symmetric Markov semigroup. See \cite{Deny65, Hirsch95} for $(r,2)$-capacities and \cite{FukushimaKaneko} for general $(r,p)$-capacities. By Proposition \ref{P:strongFeller} we can proceed similarly as in \cite[Section 2.3]{AH96}. For our purposes it is convenient to use the $\lambda$-resolvent operators $G_\lambda$ for $\lambda>0$ as in Section \ref{S:Setup}. Since different choices of $\lambda$ lead to comparable values for the capacities and do not change our results, we suppress $\lambda$ from notation. 

For a set $E\subset M$ let
\begin{equation}\label{E:cpctdef}
\cpct_{2,p}(E) := \inf\left\lbrace \|f \|_{L^p(M)}^p:\  \text{$f\in L^p_+(M)$ with $G_\lambda f(x)\geq 1$ for all $x\in E$}\right\rbrace, 
\end{equation}
with $\cpct_{2,p}(E) :=+\infty$ if no such $f$ exists. Proceeding as in \cite[Section 2.3]{AH96}, we can observe the following basic properties. 

\begin{proposition}\label{P:properties2} Let $1<p<\infty$ and assume that $(P_t)_{t>0}$ is strong Feller.
\begin{enumerate}
\item[(i)] If $E_1\subset E_2\subset M$ then $\cpct_{2,p}(E_1)\leq \cpct_{2,p}(E_2)$.
\item[(ii)] For any $E\subset M$ we have 
\[\cpct_{2,p}(E)=\inf\left\lbrace \cpct_{2,p}(G): E\subset G, \text{ $G\subset M$ open}\right\rbrace.\]
\item[(iii)] If $E_i\subset M$, $i=1,2,...$ and $E=\bigcup_{i=1}^\infty E_i$, then $\cpct_{2,p}(E)\leq \sum_{i=1}^\infty\cpct_{2,p}(E_i)$. 
\item[(iv)] The capacity $\cpct_{2,p}$ is a Choquet capacity. In particular, for any $E\subset M$ we have
\[\cpct_{2,p}(E)=\sup\left\lbrace \cpct_{2,p}(K): K\subset E, \text{ $K\subset M$ compact}\right\rbrace.\]
\end{enumerate}
\end{proposition}

\begin{proof}
We can follow the same arguments as used in \cite[Propositions 2.3.4, 2.3.5 and 2.3.12]{AH96}: Statement (i) is immediate, (ii) and (iii) can be seen as in \cite[Propositions 2.3.5 and 2.3.6]{AH96}. Since $\cpct_{2,p}(\emptyset)=0$ and we already know (i) and (ii), a proof of (iv) is achieved if we can verify that for any increasing sequence $(E_i)_{i\geq 1}$ of subsets $E_i\subset M$ we have $\cpct_{2,p}(\bigcup_{i=1}^\infty E_i)=\lim_{i\to \infty} \cpct_{2,p}(E_i)$, see \cite[Theorem 2.3.11 and the comments following it]{AH96}. This can be shown as in \cite[Proposition 2.3.12]{AH96}: One inequality is trivial by monotonicity. For the other we may assume that $\lim_{i\to \infty} \cpct_{2,p}(E_i)$ is finite. Then uniform convexity implies that the sequence $(f^{E_i})_{i\geq 1}$ of capacitary functions $f^{E_i}$ for the sets $E_i$ converges in $L^p(M)$ to a limit $f\geq 0$ with $\left\|f\right\|_{L^p(M)}^p=\lim_{i\to \infty} \cpct_{2,p}(E_i)$, cf. \cite[Corollary 1.3.3, Theorem 2.3.10]{AH96}. Using closure properties in $L^p(M)$, \cite[Proposition 2.3.9]{AH96}, together with (iii), one can then show that $G_\lambda f\geq 1$ $\cpct_{2,p}$-quasi everywhere on $E$ and conclude that $\left\|f\right\|_{L^p(M)}^p\geq \cpct_{2,p}(E)$. 
\end{proof}

A first comparison of the capacities $\ccpct_{2,p}^{\mathcal{A}}$ and $\cpct_{2,p}$ is now straightforward.
\begin{corollary}\label{C:firstcomparison}
Let $\mathcal{A}$ be a vector space of real-valued functions satisfying condition \eqref{A:bump}. Suppose that $1<p<\infty$, $\mathcal{A}\subset \mathcal{D}(\mathcal{L}^{(p)})$ and $(P_t)_{t>0}$ is strong Feller. Then for any set $E\subset M$ we have
\begin{equation}\label{E:firstcomparison}
\cpct_{2,p}(E)\leq \ccpct_{2,p}^\mathcal{A}(E).
\end{equation}
\end{corollary}

\begin{proof}
By (\ref{E:innerreg}) and Proposition \ref{P:properties2} (iv) it suffices to verify the respective inequality for compact sets $K\subset M$. Let $K$ be compact, we may assume that $\ccpct_{2,p}^\mathcal{A}(K)<+\infty$. Let $\varepsilon >0$. By (\ref{E:cap}) can find $u\in\omega_K^\mathcal{A}$ such that $\left\|u\right\|_{\mathcal{D}(\mathcal{L})}^p\leq \ccpct_{2,p}^\mathcal{A}+\varepsilon$. Since $\omega_K^\mathcal{A}\subset \mathcal{D}(\mathcal{L}^{(p)})$ we have $u=G_\lambda f$ with some $f\in L^p(M)$. If now $U\subset M$ is an open neighborhood of $K$ such that $u=1$ on $U$, then 
\[\cpct_{2,p}(K)\leq \cpct_{2,p}(U)\leq \left\|f\right\|_{L^p(M)}^p=\left\|u\right\|_{\mathcal{D}(\mathcal{L}^{(p)})}^p \leq \ccpct_{2,p}^\mathcal{A}+\varepsilon.\]
\end{proof}

\section{Truncations of potentials, and a second comparison}\label{S:second}

An inequality opposite to (\ref{E:firstcomparison}) is less trivial. To prove it, we first establish a norm estimate for trunctations of potentials. 

We say that $(P_t)_{t>0}$ \emph{satisfies condition \eqref{E:log}} if there are constants $c_1>0$, $c_2>0$ and $1\leq \alpha<2$ such that  for any nonnegative $f\in C_c(M)$ and any $t>0$ we have 
\begin{equation}\label{E:log}\tag{$\mathrm{LG}$}
\sqrt{\Gamma(P_t f)}\leq \frac{c_1\:e^{c_2t}}{\sqrt{t}}\:P_{\alpha t} f\quad \text{$\mu$-a.e. on $M$.}
\end{equation}
Condition \eqref{E:log} can be verified for large classes of manifolds and metric measure spaces, see the comments at the end of this section and the examples in Sections \ref{S:manifolds}, \ref{S:subRiemann}, and \ref{S:RCD}.  

The following theorem is a generalization of well-known truncation estimates for potentials on Euclidean spaces, \cite{A76, AH96, AP73, Hedberg72, Maz'ya72}. We provide a proof in Section \ref{S:truncation}.

\begin{theorem}\label{T:Mazja}
Assume that $(P_t)_{t>0}$ is a strong Feller semigroup satisfying \eqref{E:log}.
Let $F\in C^2(\mathbb{R}_+)$ be a function such that 
\begin{equation}\label{E:Mazja}
\sup_{t>0}|t^{i-1}F^{(i)}(t)|\leq L,\quad i=0,1,2,
\end{equation}
with a constant $L>0$. Then for any $1< p< +\infty$, any $\lambda>\frac{2}{2-\alpha}\:c_2$  and any nonnegative $f\in L^p(M)$ we have $F\circ G_\lambda f\in\mathcal{D}(\mathcal{L}^{(p)})$ and 
\begin{equation}\label{E:truncation}
\left\|F\circ G_\lambda f\right\|_{\mathcal{D}(\mathcal{L}^{(p)})}\leq c_3\left\|f\right\|_{L^p(M)}
\end{equation}
with a constant $c_3>0$ depending only on $c_1$, $c_2$, $L$, $\lambda$, $\alpha$ and $p$. For all $\lambda>\frac{2}{2-\alpha}\:c_2$ and all nonnegative $f\in C_c(M)$  we have $\mathcal{L} (F\circ G_\lambda f)\in L^\infty(M)$ and $\left\|\mathcal{L} (F\circ G_\lambda f)\right\|_{L^\infty(M)}\leq c_3\left\|f\right\|_{L^\infty(M)}$ with a constant $c_3>0$ depending only on $c_1$, $c_2$, $L$, $\lambda$ and $\alpha$.
\end{theorem}

\begin{remark}\label{R:moregeneral}
If \eqref{E:log} is assumed for all nonnegative $f\in L^1(M)\cap L^\infty(M)$, then also the stated results in the case $p=+\infty$ hold for all such $f$. 
\end{remark}

Theorem \ref{T:Mazja} allows to establish an inequality opposite to (\ref{E:firstcomparison}), provided that $\mathcal{A}$ is rich enough to contain suitable truncations of potentials. To an increasing function $F\in C^2(\mathbb{R})$ with $0\leq F\leq 1$, $F(t)=0$ for all $t\leq t_0$ with some fixed $0<t_0<1$ and $F(t)=1$ for all $t\geq 1$ we refer as a \emph{$C^2$-truncation}. Any $C^2$-truncation satisfies (\ref{E:Mazja}). Consider the following condition on $\mathcal{A}$:
\begin{multline}\label{A:functionsT}\tag{$\mathrm{F}$}
\hspace{30pt}\text{There is a $C^2$-truncation $F$ such that $F\circ G_\lambda f \in\mathcal{A}$}\\
\text{for any nonnegative $f\in L^1(M)\cap L^\infty(M)$.}
\end{multline}

The next corollary is similar to the less straightforward part of \cite[Proposition 2.3.13 and Corollary 3.3.4]{AH96}; it follows by analogous arguments as used there.

\begin{corollary}\label{C:secondcomparison} Assume that $(P_t)_{t>0}$ is a strong Feller semigroup satisfying \eqref{E:log} and let $\mathcal{A}$ be a vector space of real-valued Borel functions satisfying \eqref{A:bump} and \eqref{A:functionsT}. Suppose further that $1<p<\infty$ and $\mathcal{A}\subset \mathcal{D}(\mathcal{L}^{(p)})$. Then for all $E\subset M$ we have 
\[\ccpct_{2,p}^{\mathcal{A}}(E)\leq c_3^p\cpct_{2,p}(E),\]
where $c_3>0$ is as in \eqref{E:truncation}.
\end{corollary}

\begin{proof} As in the proof of Corollary \ref{C:firstcomparison} we may assume that $E=K$ is compact. Given $\varepsilon>0$ let $f\in L^p(M)$ be nonnegative with $G_\lambda f>1$ on $K$ and such that $\left\|f\right\|_{L^p(M)}^p<\cpct_{2,p}(K)+\varepsilon$. 
Fix $x_0\in M$ and let $f_n(x):=\mathbf{1}_{B(x_0,n)}(x)(f(x)\wedge n)$. Then all $G_\lambda f_n$ are continuous, $G_\lambda f_n\leq G_\lambda f_{n+1}$ for all $n$ and $G_\lambda f=\sup_n G_\lambda f_n$. By the Dini-Cartan lemma, \cite[Lemma 2.2.9]{Helms}, we can find $n$ so that $G_\lambda f_n > 1$ on $K$. Clearly also $\left\|f_n\right\|_{L^p(M)}^p\leq \cpct_{2,p}(K)+\varepsilon$. By condition \eqref{A:functionsT} we can find a $C^2$-truncation $F$ such that $F\circ G_\lambda f_n = 1$ on a neighborhood of $K$ and $F \circ G_\lambda f_n \in\mathcal{A}$. By \eqref{E:cap} and \eqref{E:truncation}, $\ccpct_{2,p}^\mathcal{A}(K)\leq \left\|F\circ G_\lambda f_n\right\|_{\mathcal{D}(\mathcal{L}^{(p)})}^p\leq c_3^p \left\|f_n\right\|_{L^p(M)}^p\leq c_3^p(\cpct_{2,p}(K)+\varepsilon)$.
\end{proof}

If the semigroup admits a heat kernel, then it can be used to state a condition that implies (\ref{E:log}).
We call a real valued function $(t,x,y)\mapsto p_t(x,y)$ on $(0,+\infty)\times M\times M$ a \emph{heat kernel} for $(P_t)_{t>0}$ if 
for any $t>0$ it is jointly measurable in $(x,y)$, we have 
\[P_t f(x)=\int_M p_t(x,y)f(y)\mu(dy),\quad t>0,\quad x\in M,\quad f\in L^2(M),\]
$p_t(x,y)=p_t(y,x)$ for all $t>0$ and $x,y\in M$ and 
\begin{equation}\label{E:CK}
p_{t+s}(x,y)=\int_M p_t(x,z)p_s(z,y)\mu(dz),\quad s,t>0,\quad x,y\in M.
\end{equation}

\begin{remark}\label{R:ultra} By (\ref{E:CK}) we have $p_t(x,\cdot)\in L^2(M)$ for all $t>0$ and $x\in M$. Since $p_t(x,\cdot)=P_{t/2}p_{t/2}(x,\cdot)$, semigroup regularization implies that $p_t(x,\cdot)\in \mathcal{D}(\mathcal{E})$, see for example \cite[Lemma 1.3.3]{FOT94} or \cite[Theorem 1.4.2]{Da89}. 
\end{remark}

Suppose that the semigroup $(P_t)_{t>0}$ is strong Feller. It is said to be \emph{absolutely continuous} if for any $t>0$ and $x\in M$ the Borel measure $P_t(x,dy)$ defined in (\ref{E:Markovkernel}) is absolutely continuous with respect to $\mu$; in this case it has a density $p_{t,x}\in L^1(M)$ with respect to $\mu$ which can be regularized to give a heat kernel:

\begin{proposition}\label{P:abscont}
Suppose that $(P_t)_{t>0}$ is strong Feller and absolutely continuous. Then 
$p_t(x,y):=\int_M p_{t/2,x}(z)p_{t/2,y}(z)\mu(dz)$
defines a unique heat kernel for $(P_t)_{t>0}$. For any fixed $t>0$ and $y\in M$ the function $p_t(\cdot,y)$ is lower semicontinuous.
\end{proposition}

\begin{proof}
The first statement is shown in \cite[Theorem 2]{Yan88}. Since $p_t(\cdot,y)=\sup_{N\geq 1}P_{t/2}(p_{t/2,y}\wedge N)$,
the strong Feller property implies the stated lower semicontinuity.
\end{proof}

In the presence of a heat kernel \eqref{E:log} follows from a \emph{(relaxed) logarithmic gradient estimate}. 
We say that a heat kernel $p_t(x,y)$ \emph{satisfies condition \eqref{E:loghk}} if there are constants $c_1>0$, $c_2>0$ and $1\leq \alpha<2$ such that 
\begin{equation}\label{E:loghk}\tag{$\mathrm{LG'}$}
\frac{\sqrt{\Gamma(p_t(x,\cdot))}(y)}{p_{\alpha t}(x,y)}\leq \frac{c_1\:e^{c_2t}}{\sqrt{t}}
\end{equation}
for all $t>0$ and $\mu$-a.a. $x,y\in M$,

\begin{proposition}\label{P:hktosg}
If $(P_t)_{t>0}$ admits a heat kernel $p_t(x,y)$ satisfying \eqref{E:loghk}, then $(P_t)_{t>0}$ satisfies \eqref{E:log}. 
\end{proposition}

\begin{proof}
Let $f\in C_c(M)$ be a nonnegative. The element $P_tf$ of $\mathcal{D}(\mathcal{E})$ is the integral of the $\mathcal{D}(\mathcal{E})$-valued function $y\mapsto p_t(\cdot,y)$ with respect to the measure  $f(y)\mu(dy)$. For any nonnegative $\varphi \in C_c(M)$ the map
\begin{equation}\label{E:seminorm}
g\mapsto \left\|g\right\|_\varphi:=\int_M \sqrt{\Gamma(g)}(x)\varphi(x)\mu(dx)
\end{equation}
defines a seminorm on $\mathcal{D}(\mathcal{E})$. By Cauchy-Schwarz and contractivity in $L^2(M)$ it satisfies $\left\|g\right\|_{\varphi}\leq \left\|\varphi\right\|_{L^2(M)}\mathcal{E}(g)^{1/2}$, and in particular, it is continuous on $\mathcal{D}(\mathcal{E})$. The corresponding triangle inequality for $\mathcal{D}(\mathcal{E})$-valued integrals gives
\[\left\|P_tf\right\|_\varphi=\big\|\int_Mp_t(\cdot,y)f(y)\mu(dy)\big\|_\varphi\leq \int_M \left\|p_t(\cdot,y)\right\|_\varphi f(y)\mu(dy).\]
Since this is true for any such $\varphi$, Fubini and (\ref{E:loghk}) yield
\[\sqrt{\Gamma(P_tf)}(x)\leq \int_M\sqrt{\Gamma(p_t(x,\cdot)}(y)f(y)\mu(dy)\leq  \frac{c_1\:e^{c_2t}}{\sqrt{t}}\int_M p_{\alpha t}(x,y)f(y)\mu(dy)=\frac{c_1\:e^{c_2t}}{\sqrt{t}} P_{\alpha t}f(x)\]
for $\mu$-a.e. $x\in M$.
\end{proof}

\section{Proof of the truncation estimate}\label{S:truncation}

The main tool to prove Theorem \ref{T:Mazja} is the following logarithmic type $L^p$-estimate for potentials of nonnegative functions.

\begin{proposition}\label{P:integral}
Assume that $(P_t)_{t> 0}$ is strong Feller and satisfies \eqref{E:log}. Then for any 
$\lambda>\frac{2}{2-\alpha}\:c_2$, any $1<p\leq +\infty$ and any nonnegative $f\in C_c(M)$ the function $u=G_\lambda f$ satisfies $\mathbf{1}_{\{u>0\}}\frac{\Gamma(u)}{u}\in L^p(M)$ and 
\begin{equation}\label{E:pmoment}
\Big\|\mathbf{1}_{\{u>0\}}\frac{\Gamma(u)}{u}\Big\|_{L^p(M)}\leq c_4\left\|f\right\|_{L^p(M)}
\end{equation}
with a constant $c_4>0$ depending only on $c_1$, $c_2$, $\alpha$, $\lambda$ and $p$.
\end{proposition}
 
The proof of Proposition \ref{P:integral} uses the $L^p$-boundedness of the semigroup maximal function in the sense of Rota and Stein, \cite[Corollary 2]{St61}, \cite[Chapter III, Section 3, p.73, Maximal Theorem]{St70}. For our purposes the following form of this result is suitable: For any $1<p\leq +\infty$ there exists a constant $c(p)>0$ such that
\begin{equation}\label{E:maximal}
\big\|\sup_{t>0}P_tf\big\|_{L^p(M)}\leq c(p)\left\|f\right\|_{L^p(M)}, \quad f\in L^1(M)\cap L^\infty(M).
\end{equation}
The case $p=+\infty$ is immediate with $c(+\infty)=1$ since each $P_tf$ is continuous and $\mu$ has full support. 

A second tool to prove Proposition \ref{P:integral} is the following pointwise multiplicative estimate based on (\ref{E:log}). Since the $\mu$-null sets on which the estimate \eqref{E:log} does not need to hold may depend on $t$, we use an additional regularization by the strong Feller semigroup $(P_t)_{t>0}$.

\begin{lemma}\label{L:multest}
Assume that $(P_t)_{t> 0}$ has the strong Feller property and satisfies \eqref{E:log}. Then for any $\lambda>\frac{2}{2-\alpha}\:c_2$ there is a constant $c_5>0$ depending only on $c_1$, $c_2$, $\alpha$ and $\lambda$ such that for all nonnegative $f\in C_c(M)$, all $s,t>0$ and all $x\in M$ we have 
\begin{equation}\label{E:multest}
\int_0^\infty e^{-\lambda t} P_s(\sqrt{\Gamma(P_t f)})(x)dt\leq c_5\:\left(P_s G_\lambda f(x)\right)^{1/2}\Big(\sup_{t>0} P_tf(x)\Big)^{1/2}.
\end{equation}
\end{lemma}

\begin{proof}
By the strong Feller property and \eqref{E:log} the functions $P_s(\sqrt{\Gamma(P_t f)})$ and $P_{s+\alpha t} f$ are continuous for any $s,t>0$, so that by the positivity of $P_s$ we have 
\begin{equation}\label{E:lastline}
P_s(\sqrt{\Gamma(P_t f)})(x)\leq \frac{c_1 e^{c_2t}}{\sqrt{t}}\:P_{s+\alpha t}f(x)
\end{equation}
for all $x\in M$. We may now assume that $x$ is such that $\sup_{t>0} P_tf(x)>0$, otherwise the right hand side of (\ref{E:lastline}) is zero for any $s,t>0$ and (\ref{E:multest}) is trivial. For arbitrary fixed $\delta>0$ the sum of integrals
\begin{equation}\label{E:Hedberg}
\int_0^\infty e^{(c_2-\lambda)t} t^{-1/2} P_{s+\alpha t}f(x)dt= \int_0^\delta e^{(c_2-\lambda) t}t^{-1/2} P_{s+\alpha t}f(x)dt+ \int_\delta^\infty e^{(c_2-\lambda) t}t^{-1/2} P_{s+\alpha t}f(x)dt
\end{equation}
dominates the left hand side of (\ref{E:multest}). By the hypothesis on $\lambda$ we can find $\alpha<\beta<2$ such that 
$\lambda>\frac{\beta'}{\beta-\alpha}\:c_2$,
where $\frac{1}{\beta}+\frac{1}{\beta'}=1$. This implies that $\lambda(1-\frac{\alpha}{\beta})>c_2$ and $\lambda(1-\frac{\alpha}{\beta'})>c_2$. By H\"older's inequality the first summand on the right hand side of (\ref{E:Hedberg}) is bounded by 
\begin{multline}
\left(\int_0^\delta e^{c_2\beta t-\lambda(1-\frac{\alpha}{\beta'})\beta t}t^{-\beta/2}P_{s+\alpha t}f(x)dt\right)^{1/\beta}\left(\int_0^\delta e^{-\lambda \alpha t}P_{s+\alpha t}f(x)dt\right)^{1/\beta'}\notag\\
\leq \frac{\delta^{1/\beta-1/2}}{(1-\frac{\beta}{2})^{1/\beta}\alpha^{1/\beta'}}\:\left(\sup_{t>0}P_tf(x)\right)^{1/\beta}\left(\int_0^\infty e^{-\lambda \tau} P_{s+\tau}f(x)d\tau\right)^{1/\beta'}.\notag
\end{multline} 

For the second summand in (\ref{E:Hedberg}) we can swap the roles of $\beta$ and $\beta'$ and use the fact that $\beta'>2$ to deduce the analogous bound
\[\frac{\delta^{1/\beta'-1/2}}{(\frac{\beta'}{2}-1)^{1/\beta'}\alpha^{1/\beta}}\:\left(\sup_{t>0}P_tf(x)\right)^{1/\beta'}\left( \int_0^\infty e^{-\lambda \tau} P_{s+\tau}f(x)d\tau  \right)^{1/\beta}. \]
Now the choice
\begin{equation}\label{E:delta}
\delta=\frac{\int_0^\infty e^{-\lambda \tau} P_{s+\tau}f(x)d\tau}{\sup_{t>0}P_tf(x)}
\end{equation}
yields the claimed inequality.
\end{proof}

We prove Proposition \ref{P:integral}. 

\begin{proof}
Let $\varepsilon>0$. Since $\Gamma(u)\in L^1(M)$ we have $\lim_{s\to 0}P_s(\sqrt{\Gamma(u)})=\sqrt{\Gamma(u)}$ in $\mu$-measure, so that Fatou's lemma yields
\begin{equation}\label{E:appFatou}
\int_{\{u>\varepsilon\}}\frac{\Gamma(u)^p}{u^p}\varphi\:d\mu\leq \liminf_{s\to 0}\int_{\{u>\varepsilon\}}\frac{(P_s\sqrt{\Gamma(u)})^{2p}}{u^p}\varphi\:d\mu
\end{equation}
for any $1\leq p<+\infty$ and any nonnegative $\varphi\in L^\infty(M)$. 

The integral $\int_0^\infty e^{-\lambda t} P_tf\:dt$ converges in $\mathcal{D}(\mathcal{E})$. For any nonnegative $\psi\in L^1(M)\cap L^\infty(M)$ and $s>0$ we consider the seminorm $g\mapsto \|g\|_{P_s\psi}=\int_M \sqrt{\Gamma(g)}P_s\psi\:d\mu$ on $\mathcal{D}(\mathcal{E})$ as defined in (\ref{E:seminorm}). By the symmetry of the semigroup, the triangle inequality for integrals of $\mathcal{D}(\mathcal{E})$-valued functions and Fubini, we have 
\begin{multline}
\int_M \psi \:P_s\sqrt{\Gamma(G_\lambda f)}\:d\mu=\left\|G_\lambda f\right\|_{P_s\psi}=\big\|\int_0^\infty e^{-\lambda t}P_tf\:dt\big\|_{P_s\psi}\leq \int_0^\infty e^{-\lambda t} \left\|P_tf\right\|_{P_s\psi}dt\notag\\
= \int_M \psi \int_0^\infty e^{-\lambda t} P_s\sqrt{\Gamma(P_t f)}\:dt\:d\mu.\notag
\end{multline}
Since $\psi$ was arbitrary, it follows that 
\[P_s\sqrt{\Gamma(G_\lambda f)}\leq \int_0^\infty e^{-\lambda t} P_s\sqrt{\Gamma(P_t f)}\:dt\]
$\mu$-a.e. on $M$. Therefore we have
\begin{align}\label{E:inequalities}
\int_{\{u>\varepsilon\}} \frac{(P_s\sqrt{\Gamma(u)})^{2p}}{u^{p}}\varphi\: d\mu &\leq \int_{\{u>\varepsilon\}}\frac{\varphi}{u^p}\left(\int_0^\infty e^{-\lambda t} P_s\sqrt{\Gamma(P_t f)}dt \right)^{2p}\:d\mu\notag\\
&\leq c_5 \int_{\left\lbrace u>\varepsilon\right\rbrace} \frac{\varphi}{u^p} (P_s u)^p \Big(\sup_{t>0} P_tf\Big)^p\:d\mu, 
\end{align}
the second inequality follows using Lemma \ref{L:multest}. 

Clearly we have $\lim_{s\to 0} (P_su)^p=u^p$ in $\mu$-measure. In the case that $1<p<+\infty$ it suffices to consider the choice $\varphi\equiv 1$, for which the integrand of the last integral in (\ref{E:inequalities}) admits the majorant
$\varepsilon^{-p}\left\|u\right\|_{L^\infty(M)}^p\big(\sup_{t>0} P_tf\big)^p$. By (\ref{E:maximal}) this majorant is integrable, and the dominated convergence theorem gives
\begin{equation}\label{E:domconv}
\lim_{s\to 0} \int_{\left\lbrace u>\varepsilon\right\rbrace} \frac{1}{u^p} (P_s u)^p \Big(\sup_{t>0} P_tf\Big)^p\:d\mu=  \int_{\left\lbrace u>\varepsilon\right\rbrace}  \Big(\sup_{t>0} P_tf\Big)^p\:d\mu.
\end{equation}
Combining (\ref{E:appFatou}), (\ref{E:inequalities}) and (\ref{E:domconv}) and using (\ref{E:maximal}) we obtain
\[\int_{\{u>\varepsilon\}}\frac{\Gamma(u)^p}{u^p}\:d\mu\leq c_5 c(p) \left\|f\right\|_{L^p(M)}^p,\]
and letting $\varepsilon$ go to zero we arrive at (\ref{E:pmoment}). Similarly, we have 
\[\lim_{s\to 0} \int_{\left\lbrace u>\varepsilon\right\rbrace} \frac{\varphi}{u} (P_s u) \Big(\sup_{t>0} P_tf\Big)\:d\mu=  \int_{\left\lbrace u>\varepsilon\right\rbrace} \varphi \Big(\sup_{t>0} P_tf\Big)\:d\mu \]
for any nonnegative $\varphi\in L^\infty(M)\cap L^1(M)$, note that in this case $\varepsilon^{-1}\left\|u\right\|_{L^\infty(M)}\left\|f\right\|_{L^\infty(M)}\:\varphi$ provides an integrable majorant. This gives 
\[\int_{\{u>\varepsilon\}}\frac{\Gamma(u)}{u}\varphi\:d\mu\leq c_5 \int_{\{u>\varepsilon\}}\varphi \Big(\sup_{t>0} P_tf\Big)\:d\mu\leq c_5 \left\|f\right\|_{L^\infty(M)}\left\|\varphi\right\|_{L^1(M)}\]
for any such $\varphi$. Using the standard decomposition and approximation and letting $\varepsilon\to 0$, we find that
\[\left|  \int_{\{u>0\}}\frac{\Gamma(u)}{u}\varphi\:d\mu \right|\leq c_5 \left\|f\right\|_{L^\infty(M)}\left\|\varphi\right\|_{L^1(M)}\]
for arbitrary $\varphi\in L^1(M)$. Consequently $\mathbf{1}_{\{u>0\}}\frac{\Gamma(u)}{u}$ is an element of $L^\infty(M)$ and its $L^\infty(M)$-norm is bounded by $c_5 \left\|f\right\|_{L^\infty(M)}$.
\end{proof}

To prove Theorem \ref{T:Mazja} we combine Proposition \ref{P:integral} with the chain rule. The standard chain rule for generators says that if $F\in C^2(\mathbb{R})$ is such that $F(0)=0$, $F'(0)=0$ and $F''$ is bounded, then for any $u\in\mathcal{D}(\mathcal{L})$ we have 
\begin{equation}\label{E:usualchainrule}
F(u)\in \mathcal{D}(\mathcal{L}^{(1)})\quad \text{ and }\quad
\mathcal{L}^{(1)}F(u)=F'(u)\mathcal{L}u+F''(u)\Gamma(u);
\end{equation}
see \cite[Chapter I, Corollary 6.1.4]{BH91}. 
The following variant is in line with Maz'ya's original proof of the quantitative  estimate (\ref{E:truncation}), it allows $F''$ to be singular at zero. See \cite{A76} and \cite{Maz'ya72}.

\begin{lemma}\label{L:chainforL}
Assume that $(P_t)_{t> 0}$ is a strong Feller semigroup satisfying \eqref{E:log}. Let $F$ be as in Theorem \ref{T:Mazja}, $\lambda>\frac{2}{2-\alpha}\:c_2$ and let $u=G_\lambda f$ with $f\in C_c(M)$. Then $F(u)\in\mathcal{D}(\mathcal{E})$ and 
\begin{equation}\label{E:chainforL}
\mathcal{L}F(u)(v)=\int_M (\mathcal{L}u)F'(u) vd\mu+\int_{\{u>0\}}F''(u)\Gamma(u)vd\mu, \quad v\in\mathcal{D}(\mathcal{E})\cap C_c(M).
\end{equation}
\end{lemma}

\begin{proof}
Clearly $u$ is in $\mathcal{D}(\mathcal{E})\cap C_b(M)$, hence also $F(u)$ is in this space by (\ref{E:Mazja}) and the Markov property. Let $\varepsilon>0$. Set $F_0(t):=F(t)-F'(0)t$. Since $F_0'(0)=0$ and $F_0''=F''$ is bounded by $L/\varepsilon$ on the range of $u\vee \varepsilon$, the Markov property also implies that $F_0'(u\vee \varepsilon)\in \mathcal{D}(\mathcal{E})\cap C_b(M)$, and on $\{u>\varepsilon\}$ this function equals $F_0'(u)$ $\mu$-a.e. For any $v\in \mathcal{D}(\mathcal{E})$ we have $\Gamma(u \vee\varepsilon, v)=\Gamma(u,v)$ and $\Gamma(F_0(u\vee \varepsilon),v)=\Gamma(F_0(u),v)$ $\mu$-a.e. on  $\{u>\varepsilon\}$ as can be seen from (\ref{E:chain}). Making further use of (\ref{E:chain}) we therefore observe that  
\begin{align}\label{E:claimforintegrals}
\int_{\{u>\varepsilon\}}&\Gamma(F_0(u),v)d\mu\notag\\
&=\int_{\{u>\varepsilon\}} F_0'(u\vee \varepsilon) \Gamma(u,v)d\mu\notag\\
&=\int_{\{u> \varepsilon\}}\Gamma(u,F_0'(u\vee\varepsilon)v)d\mu-\int_{\{u>\varepsilon\}}v\Gamma(u,F_0'(u\vee\varepsilon))d\mu\notag\\
&=-\int_M(\mathcal{L}u)F_0'(u\vee\varepsilon)vd\mu-\int_{\{u\leq \varepsilon\}}\Gamma(u,F_0'(u\vee\varepsilon)v)d\mu
-\int_{\{u>\varepsilon\}}v F_0''(u)\Gamma(u,u)d\mu\notag
\end{align}
for any $v\in \mathcal{D}(\mathcal{E})\cap C_c(M)$. Since $(F_0'(u\vee\varepsilon)-F_0'(\varepsilon))v$ has support in $\{u>\varepsilon\}$, we have $\Gamma(u,F_0'(u\wedge\varepsilon)v)=F_0'(\varepsilon)\Gamma(u,v)$ $\mu$-a.e. on $\{u\leq \varepsilon\}$ so that the second integral on the right-hand side is seen to equal $F_0'(\varepsilon)\int_{\{u\leq \varepsilon\}}\Gamma(u,v)d\mu$. Taking limits as $\varepsilon\to 0$, using the continuity of $F_0'$ and Proposition \ref{P:integral} we obtain
\begin{multline}
\mathcal{L}F(u)(v)=-\mathcal{E}(F(u),v)=-\int_{\{u>0\}}\Gamma(F(u),v))d\mu-\int_{\{u=0\}}\Gamma(F(u),v))d\mu\notag\\
=\int_M \mathcal{L}u\:F'(u)vd\mu+\int_{\{u>0\}}F''(u)\Gamma(u,u)vd\mu-\int_{\{u=0\}}\Gamma(F(u),v)d\mu\notag
\end{multline}
for any $v\in \mathcal{D}(\mathcal{E})\cap C_c(M)$. The last integral in the last line is zero, because it is bounded in modulus by  $L\left(\int_{\{u=0\}}\Gamma(u)d\mu\right)^{1/2}\mathcal{E}(v)^{1/2}$ 
and $\int_{\{u=0\}}\Gamma(u)\:d\mu=0$, as follows for instance from \cite[Chapter I, Theorem 5.2.3 and Theorem 7.1.1 and its proof]{BH91}. This proves (\ref{E:chainforL}) for $v\in \mathcal{D}(\mathcal{E})\cap C_c(M)$. Since by Cauchy-Schwarz, (\ref{E:Mazja}) and Proposition \ref{P:integral} we have $|\mathcal{L}F(u)(v)|\leq \mathcal{E}(F(u))^{1/2}\mathcal{E}(v)^{1/2}$,
\begin{equation}\label{E:inter1}
\big| \int_M \mathcal{L}u\:F'(u)vd\mu\big|\leq L\:\left\|\mathcal{L}u\right\|_{L^2(M)}\left\|v\right\|_{L^2(M)}
\end{equation}
and 
\begin{equation}\label{E:inter2}
\big|\int_{\{u>0\}}F''(u)\Gamma(u)\:v\:d\mu\big|\leq c_4\:L\:\left\|f\right\|_{L^2(M)}\left\|v\right\|_{L^2(M)},
\end{equation}
identity (\ref{E:chainforL}) is seen to hold for all $v\in \mathcal{D}(\mathcal{E})$.
\end{proof}

We can now prove Theorem \ref{T:Mazja}. 
\begin{proof}
As before let $u=G_\lambda f$. Suppose first that $f$ is a nonnegative element of $C_c(M)$. Clearly $F(u)\in L^p(M)$ for any $1<p\leq +\infty$. Since $\mathcal{D}(\mathcal{E})\cap C_c(M)$ is dense in $L^2(M)$ and $u\in\mathcal{D}(\mathcal{L})$, the chain rule (\ref{E:chainforL}), (\ref{E:inter1}) and (\ref{E:inter2}) imply that $\mathcal{L}F(u)\in L^2(M)$. Therefore $F(u)\in\mathcal{D}(\mathcal{L})$ by (\ref{E:domaininfo}). A similar argument
yields $\mathcal{L}F(u)\in L^p(M)$, $1<p\leq +\infty$. For all $1<p<+\infty$ we therefore have $F(u)\in\mathcal{D}_p\subset \mathcal{D}(\mathcal{L}^{(p)})$ by (\ref{E:aprioridomain}). To see (\ref{E:truncation}) for $1<p<+\infty$, note that combining \eqref{E:chainforL} and Proposition \ref{P:integral} gives
\begin{align}\label{E:finalestimate}
\left\|(\lambda-\mathcal{L})F(u)\right\|_{L^p(M)}&\leq \lambda\:L\:\left\|u\right\|_{L^p(M)}+\left\|\mathcal{L}F(u)\right\|_{L^p(M)}\\
&\leq \lambda\:L\:\left\|u\right\|_{L^p(M)}+L\:\left\|\mathcal{L}u\right\|_{L^p(M)}+L\:\Big\|\mathbf{1}_{\{u>0\}}\frac{\Gamma(u)}{u}\Big\|\notag\\
&\leq (3+c_4)L\left\|f\right\|_{L^p(M)},\notag
\end{align}
which is (\ref{E:truncation}). The estimate for $p=+\infty$ follows similarly. 

For $1<p<+\infty$ we can now extend (\ref{E:truncation}) similarly as in \cite[Theorem 3.3.3]{AH96}. Suppose that $f\in L^p_+(M)$ and let $(f_n)_n$ be a sequence of nonnegative functions $f_n\in C_c(M)$ such that $\lim_n f_n=f$ in $L^p(M)$. Then also 
\[\lim_n\left\|F\circ G_\lambda f_n-F\circ G_\lambda f\right\|_{L^p(M)}\leq L\:\lim_n\left\|G_\lambda f_n-G_\lambda f\right\|_{L^p(M)}=0\]
by (\ref{E:Mazja}), the mean value theorem and the boundedness of $G_\lambda$ on $L^p(M)$. Set 
\[g_n:=(\lambda-\mathcal{L}^{(p)})F\circ G_\lambda f_n.\] 
By (\ref{E:truncation}) we have $\left\|g_n\right\|_{L^p(M)}\leq c_4\:\sup_n \left\|f_n\right\|_{L^p(M)}\leq c_4\:\left\|f\right\|_{L^p(M)}$ for all $n$, hence we may assume that $(g_n)_n$ converges to some $g$ weakly in $L^p(M)$. As a consequence, 
\[\left\|G_\lambda g\right\|_{\mathcal{D}(\mathcal{L}^{(p)})}=\left\|g\right\|_{L^p(M)}\leq c_4\:\left\|f\right\|_{L^p(M)}.\] Since by weak convergence also $\lim_n F\circ G_\lambda f_n=\lim_n G_\lambda g_n=G_\lambda g$ weakly in $L^p(M)$, we must have $F\circ G_\lambda f= G_\lambda g$.
\end{proof}

\section{From finite Hausdorff measure to zero capacity}\label{S:Hausdorfftocap}

Sufficient conditions for a set to have zero $(2,p)$-capacity can be stated in terms of its Hausdorff measure respectively dimension. For Euclidean spaces these results are standard, \cite[Chapter 5]{AH96}, we provide adapted versions for metric measure spaces. 

The following fact is well-known, \cite[Lemma 4.2.4]{FOT94}, and immediate from Proposition \ref{P:abscont}.

\begin{proposition}\label{P:Gabscont}
Suppose that $(P_t)_{t>0}$ is strong Feller and absolutely continuous and that $\lambda>0$. Then $g_\lambda(x,y):=\int_0^\infty e^{-\lambda t}p_t(x,y)\:dt$ defines a function $g_\lambda:M\times M\to [0,+\infty]$ that is symmetric and jointly measurable in $(x,y)$, and we have 
\begin{equation}\label{E:Gabscont}
G_\lambda f(x)=\int_M g_\lambda(x,y)f(y)\mu(dy),\quad x\in M,
\end{equation}
for any $f\in L^0_+(M)$. For any $y\in M$ the function $g_\lambda(\cdot, y)$ is lower-semicontinuous on $M$. 
\end{proposition}

Let $\mathcal{M}_+(M)$ denote the cone of nonnegative Radon measures on $M$. In the absolutely continuous case (\ref{E:Gabscont}) can be generalized by setting
\begin{equation}\label{E:measurepot}
G_\lambda \nu(x):=\int_Mg_\lambda(x,y)\nu(dy),\quad x\in M,
\end{equation}
for any $\nu\in \mathcal{M}_+(M)$.  For any $x\in M$ the map $\mu\mapsto  G_\lambda \nu(x)$
is lower semicontinuous on $\mathcal{M}_+(M)$ w.r.t. weak convergence of measures. This follows using monotone convergence since by Proposition \ref{P:Gabscont} the function $g_\lambda(\cdot, y)$ can be approximated pointwise by an increasing sequence of nonnegative continuous compactly supported functions on $M$, \cite[Proposition 2.3.2 (b)]{AH96}. 

The following is a variant of a well-known dual representation of capacities. As in the preceding sections we keep $\lambda>0$ fixed.

\begin{corollary}\label{C:dualrep}
Let $1<p<+\infty$, assume that $(P_t)_{t>0}$ is strong Feller and absolutely continuous. Then we have 
\[\cpct_{2,p}(E)^{1/p}=\sup\left\lbrace \nu(E):\ \nu\in\mathcal{M}_+(M),\ \supp\nu\subset E,\ \left\|G_\lambda \nu\right\|_{L^q(M)}\leq 1\right\rbrace\]
for any $E\subset M$ Borel, where $\frac{1}{p}+\frac{1}{q}=1$.
\end{corollary}

\begin{proof}
The result follows from an application of the minimax theorem, \cite[Theorem 2.4.1]{AH96}, to the bilinear map $(\nu,f)\mapsto \int_M G_\lambda f\:d\nu=\int_M G_\lambda\nu\: f\:d\mu$,
where $\mu$ ranges over all Radon probability measures on $M$ an $f$ over the closed unit ball in $L^p(M)$, see \cite[Theorem 2.5.1 and Corollary 2.5.2]{AH96} or Corollary \ref{C:dualreploc} below for details.
\end{proof}

Let $h:[0,+\infty)\to [0,+\infty)$ be a non-decreasing and right-continuous function, strictly positive on $(0,+\infty)$ and having the doubling property $h(2r)\leq c\:h(r)$, $r>0$, where $c>1$ is a fixed constant. We call such $h$ a \emph{Hausdorff function}. 

Let $E\subset M$. For any $\delta>0$ let $\mathcal{H}^h_\delta(E)$ be the infimum over all sums $\sum_{i=1}^\infty h(\diam(E_i))$, where $E_i\subset M$ are sets with $\diam(E_i)<\delta$ and such that $E\subset \bigcup_i E_i$. This quantity is decreasing in $\delta$, and its limit
\[\mathcal{H}^h(E):=\lim_{\delta\to 0}\mathcal{H}_\delta^h(E)\] 
is called the \emph{Hausdorff measure of $E$ with Hausdorff function $h$}. See for instance \cite[p. 132]{AH96}, \cite{Howroyd1995} or \cite[Section 4.9]{Mattila}. For $s\geq 0$ and $h(r)=r^s$ we obtain the \emph{$s$-dimensional Hausdorff measure}, for which we use the traditional notation $\mathcal{H}^s$. Recall that the \emph{Hausdorff dimension} of $E$ is defined as the unique nonnegative real number $\dim_H E$ at which $s\mapsto \mathcal{H}^s(E)$ jumps from infinity to zero.

If $g_\lambda(x,y)$ admits adequate asymptotics, then the finiteness of a suitable Hausdorff measure of a set $\Sigma$ implies that $\cpct_{2,p}(\Sigma)$ is zero. Given $1<p<+\infty$ we consider the Hausdorff function $h_p$ defined by $h_p(0):=0$ and 
\begin{equation}\label{E:hp}
h_p(r):=\Big(1+\log_+ \frac{1}{r}\Big)^{1-p}, \quad r>0,
\end{equation}
where $\log_+$ denotes the nonnegative part of $\log$. 

\begin{lemma}\label{L:resolventlower}
Let $(P_t)_{t>0}$ be an absolutely continuous strong Feller semigroup and $\Sigma\subset M$ a closed set. Let $1<p<+\infty<$, $d>0$ and assume that 
\begin{equation}\label{E:lowerd}
\liminf_{r\to 0} \frac{\mu(B(x,r))}{r^d}>0
\end{equation}
and that for any $x\in \Sigma$ we have
\begin{equation}\label{E:resolventlower}
\lim_{r\to 0} \inf_{y,z\in B(x,r)} \varrho(y,z)^{d-2}g_\lambda(y,z)>0.
\end{equation}
If $d>2p$ and $\mathcal{H}^{d-2p}(\Sigma)<+\infty$, then $\cpct_{2,p}(\Sigma)=0$. If $d=2p$ and $\mathcal{H}^{h_p}(\Sigma)<+\infty$, then we also have $\cpct_{2,p}(\Sigma)=0$.
\end{lemma}

\begin{remark}\label{R:locallysuffices}\mbox{}
\begin{enumerate}
\item[(i)] Clearly $\dim_H\Sigma<d-2p$ implies $\mathcal{H}^{d-2p}(\Sigma)=0$. 
\item[(ii)] By Proposition \ref{P:properties2} (i) and (iii) the conclusion of Lemma \ref{L:resolventlower} remains true if $(\Sigma_i)_{i\geq 1}$ is a sequence of closed sets $\Sigma_i$ such that $\Sigma\subset \bigcup_i \Sigma_i$ and instead of  $\mathcal{H}^{d-2p}(\Sigma)<+\infty$ (resp. $\mathcal{H}^{h_p}(\Sigma)<+\infty$) we have $\mathcal{H}^{d-2p}(\Sigma_i)<+\infty$  (resp. $\mathcal{H}^{h_p}(\Sigma_i)<+\infty$) for all $i$. 
\end{enumerate}
\end{remark}

Given $1<p<+\infty$ and $\nu\in \mathcal{M}_+(M)$ we consider the \emph{Maz'ya-Khavin type nonlinear $(2,p)$-potential of $\nu$} on $(M,\mu)$, defined by 
\begin{equation}\label{E:MazjaKhavin}
V^\nu_{2,p}(x)=\int_M g_\lambda(x,y)\left(\int_M g_\lambda(z,y)\nu(dz)\right)^{q-1}\mu(dy),\quad x\in M,
\end{equation}
where $\frac{1}{p}+\frac{1}{q}=1$. See \cite[formula (2.1)]{Maz'yaKhavin} or \cite[Definition 2.5.4]{AH96}, and see \cite{Cascante} for related definitions.  By Fubini's theorem we have
\begin{equation}\label{E:nonlinFubini}
\int_M  V^\nu_{2,p}\ d\nu=\left\|G_\lambda\nu\right\|_{L^q(M)}^q.
\end{equation}

Lemma \ref{L:resolventlower} follows by versions of well-known arguments, cf. \cite[Theorem 8.7]{Mattila}.
\begin{proof}
Suppose that $d>2p$ and $\mathcal{H}^{d-2p}(\Sigma)<+\infty$ but $\cpct_{2,p}(\Sigma)>0$. Then Corollary \ref{C:dualrep} guarantees the existence of some $\nu \in \mathcal{M}_+(M)$ with $\supp \nu \subset \Sigma$ and such that $\int_M V^\nu_{2,p}\ d\nu\leq 1$ by (\ref{E:nonlinFubini}).
Accordingly there is a Borel set $\Sigma_0\subset \Sigma$ with $0<\nu(\Sigma_0)<+\infty$ such that $V_{2,p}^\nu(x)<+\infty$ for all $x\in \Sigma_0$. For any such $x$ conditions (\ref{E:lowerd}) and (\ref{E:resolventlower}) guarantee the existence of 
$r_x>0$ and $c_x>0$ such that $\mu(B(x,r))\geq c_x\:r^d$ and $g_\lambda(y,z)>c_x\:\varrho(y,z)^{2-d}$ for all $0<r<r_x$ and $y,z\in B(x,r)$. Consequently
\begin{align}
V_{2,p}^\nu(x)&\geq \int_{B(x,r)} g_\lambda(x,y)\left(\int_{B(y,2r)} g_\lambda(y,z)\nu(dz)\right)^{q-1}\mu(dy)\notag\\
&\geq c_x^q\:\int_{B(x,r)} \varrho(x,y)^{2-d}\left(\int_{B(y,2r)} \varrho^{2-d}(y,z)\nu(dz)\right)^{q-1}\mu(dy)\notag\\
&\geq  c\:r^{(2-d)q} \int_{B(x,r)}\nu(B(y,2r))^{q-1}\mu(dy)\notag\\
&\geq c\:r^{(d-2p)(1-q)}\nu(B(x,r))^{q-1},\notag
\end{align}
note that $B(y,2r)\supset B(x,r)$ for any $y\in B(x,r)$ and that $2p(q-1)=2q$. Since the integral in the first line above goes to zero as $r\to 0$, we obtain
\[\limsup_{r\to 0} \frac{\nu(B(x,r))}{r^{d-2p}}=0,\quad x\in \Sigma_0,\]
and by Egorov's theorem there is a Borel set $\Sigma_1\subset \Sigma_0$ with $\nu(\Sigma_1)\geq \frac12\nu(\Sigma_0)$ and such that for any $\varepsilon>0$ we can find $r_\varepsilon>0$ that guarantees
\[\nu(B(x,r))\leq \varepsilon\:r^{d-2p},\quad x\in \Sigma_1,\ 0<r<r_\varepsilon.\] 
Suppose $\varepsilon>0$. Let $A_1, A_2, ...$ be Borel sets with $r_i:=\diam (A_i)<r_\varepsilon$ and $A_i\cap \Sigma_1\neq \emptyset$ for all $i$  and such that $ \Sigma_1\subset \bigcup_i A_i$ and $\sum_i \diam(A_i)^{d-2p}\leq \mathcal{H}^{d-2p}(\Sigma_1)+1$. For any $i$ let $x_i$ be a point in $A_i\cap \Sigma_1$. Then 
\[\frac12 \nu(\Sigma_1)\leq \sum_i \nu(B(x_i,r_i)))\leq \varepsilon\:\sum_i r_i^{d-2p}\leq \varepsilon\:(\mathcal{H}^{d-2p}(\Sigma)+1).\]
Since $\varepsilon>0$  was arbitrary, this would imply $\mathcal{H}^{d-2p}(\Sigma)=+\infty$, a contradiction. In the case that $d=2p$ we obtain 
\[V_{2,p}^\nu(x)\geq  c\:\int_{B(x,r)}\varrho(x,y)^{-2p}\nu(B(y,2r))^{q-1}\mu(dy)
\geq c\:(-\log r)\nu(B(x,r))^{q-1}\]
for any $x\in \Sigma_0$ and small enough $r$. Similarly as before these inequalities imply that $\limsup_{r\to 0} \nu(B(x,r))\:h_p(r)^{-1}=0$, and the result follows by analogous arguments.
\end{proof}

\section{From zero capacity to zero Hausdorff measure}\label{S:captoHausdorff}

We also provide a version of the opposite implication that zero capacity implies zero Hausdorff measure of sufficiently large dimension.

\begin{lemma}\label{L:resolventupper} 
Let $(M,\varrho)$ be complete, $(P_t)_{t>0}$ strong Feller and absolutely continuous, 
\begin{equation}\label{E:resolventupper}
g_\lambda(x,y)\leq c_6\:\varrho(x,y)^{2-d}\quad \quad  x,y\in M,
\end{equation}
and
\begin{equation}\label{E:upperd}
\mu(B(x,r))\leq c_7\:r^d,\quad x\in M,\quad r>0,
\end{equation}
with positive constants $c_6$ and $c_7$. Let $\Sigma\subset M$ be a closed set.
\begin{enumerate}
\item[(i)] If $d>4$ and $\varepsilon>0$, then $\cpct_{2,2}(\Sigma)=0$ implies $\mathcal{H}^{d-2p+\varepsilon}(\Sigma)=0$. If $d=4$, $M$ is bounded and $h$ is a Hausdorff function satisfying
\begin{equation}\label{E:condh}
\int_0^1 (-\log r)dh(r)<+\infty,
\end{equation}
then $\cpct_{2,2}(\Sigma)=0$ implies $\mathcal{H}^h(\Sigma)=0$.
\item[(ii)] Suppose that $1<p<+\infty$, $d>2p$ and $\varepsilon>0$. For $p\neq 2$ assume in addition that
\begin{equation}\label{E:lowerdadd}
\mu(B(x,r))\geq \frac{1}{c_7}\:r^d,\quad x\in M,\quad r>0.
\end{equation}
Then $\cpct_{2,p}(\Sigma)=0$ implies $\mathcal{H}^{d-2p+\varepsilon}(\Sigma)=0$. 
\end{enumerate}
\end{lemma}

Our proof of Lemma \ref{L:resolventupper} (ii)  for $p\neq 2$ employs a result from \cite{PerezWheeden}, and we assume (\ref{E:lowerdadd}) to ensure its applicability. Condition (\ref{E:lowerdadd}) implies that $\mu(M)=+\infty$. 

\begin{remark}\label{R:exampleh}\mbox{}
\begin{enumerate}
\item[(i)] Clearly $\mathcal{H}^{d-2p+\varepsilon}(\Sigma)=0$ implies that $\dim_H \Sigma \leq d-2p$.
\item[(ii)] For any $\varepsilon>0$ the functions $h(r)=r^\varepsilon$ and 
$h(r)=(1+\log_+\frac{1}{r})^{-1-\varepsilon}$ satisfy \eqref{E:condh}. 
\end{enumerate}
\end{remark}

Lemma \ref{L:resolventupper} (i) can be shown using (\ref{E:resolventupper}), (\ref{E:upperd}) and standard arguments, \cite[Theorem 8.9]{Mattila}. 

\begin{proof}[Proof of Lemma \ref{L:resolventupper} (i).]
We first consider the case $d>4$. By (\ref{E:resolventupper}) and (\ref{E:upperd}) we have
\begin{equation}\label{E:convolute2}
\left\|G_\lambda \nu\right\|_{L^2(M)}^2\leq c\int_M\int_M \varrho(x,z)^{4-d}\nu(dz)\nu(dx)
\end{equation}
for any $\nu\in \mathcal{M}_+(M)$ with $c>0$ dependig only on $c_6$ and $c_7$. This can be seen from (\ref{E:upperd}) and elementary estimates, see \cite[Proposition 4.12]{Aubin82}: With $r:=\frac12\varrho(x,z)$ one finds that 
\begin{multline}\label{E:Giraud}
\int_M  \varrho(x,y)^{2-d}\varrho(y,z)^{2-d} \mu(dy)
\leq c\:r^{2-d}\int_{B(x,r)}\varrho(x,y)^{2-d}\mu(dy)\\
+c\:r^{2-d}\int_{B(z,3r)\setminus B(x,r)}\varrho(y,z)^{2-d}\mu(dy)
+c\int_{M\setminus B(z,3r)}\varrho(y,z)^{4-2d}\mu(dy),
\end{multline}
which does not exceed $c\:\varrho(x,z)^{4-d}$, and (\ref{E:convolute2}) follows by Fubini.

Now suppose that $\varepsilon>0$ and $\mathcal{H}^{d-4+\varepsilon}(\Sigma)>0$. By Frostman's lemma for complete separable metric spaces, see \cite[Theorem 8.8 and comments on p. 117]{Mattila}, there is some $\nu \in\mathcal{M}_+(M)$ with $0<\nu(\Sigma)<+\infty$, $\supp \nu$ compact and contained in $\Sigma$, and there is a constant $c>0$ such that $\nu(B(x,r))\leq c\:r^{d-4+\varepsilon}$, for all $x\in \Sigma$ and $r>0$. This, together with (\ref{E:convolute2}), implies $\left\|G_\lambda\nu\right\|_{L^2(M)}<+\infty$ and therefore $\cpct_{2,2}(\Sigma)^{1/2}\geq \nu(\Sigma)\:\left\|G_\lambda\nu\right\|_{L^2(M)}^{-1}>0$  by Corollary \ref{C:dualrep}, what contradicts $\cpct_{2,2}(\Sigma)=0$.

In the case that $d=4$ and $M$ is bounded, \eqref{E:Giraud} gives 
\[\int_M \varrho(x,y)^{-2}\varrho(y,z)^{-2} \mu(dy)\leq c\big(1+\log_+\frac{1}{\varrho(x,z)}\big),\quad x,y\in M.\]
If we would have $\mathcal{H}^h(\Sigma)>0$ with $h$ as indicated, Frostman's lemma would guarantee the existence of some $\nu\in 
\mathcal{M}_+(M)$ with value and support as before and $\nu(B(x,r))\leq c\:h(r)$ for all $x\in \Sigma$ and $r>0$. See for instance \cite[Theorem 5.1.12]{AH96}, the statement remains valid under the present hypotheses. The preceding would produce the same contradiction because
\[\left\|G_\lambda\nu\right\|_{L^2(M)}^2\leq c\:\int_M\int_M\big(1+\log_+\frac{1}{\varrho(x,z)}\big)\nu(dz)\nu(dx)<+\infty,\]
note that for any fixed $x\in M$ the inner integral is bounded by 
\begin{multline}
\int_{B(x,1)}\big(1+\log_+\frac{1}{\varrho(x,z)}\big)\nu(dz)+\int_{M\setminus B(x,1)}\big(1+\log_+\frac{1}{\varrho(x,z)}\big)\nu(dz)\notag\\
\leq c\:\int_0^1 (-\log r)\:dh(r)+2\mu(M).
\end{multline}
\end{proof}

For any $s\geq 0$ and $\nu\in \mathcal{M}_+(M)$ let 
\begin{equation}\label{E:Riesz}
U^s\nu(x):=\int_M\varrho(x,y)^{-s}\nu(dy),\quad x\in M,
\end{equation}
denote the \emph{Riesz type potential of $\nu$ of order $s$ on $M$}. The right hand side of \eqref{E:convolute2} equals 
the square of the $L^2$-norm of $U^{2-d}\nu$. Lemma \ref{L:resolventupper} (ii) follows similarly as (i) if the $L^p$-norm of $U^{2-d}\nu$ can be controlled suitably. Given $1<p<+\infty$, $s\geq 0$ and a nonnegative Radon measure $\nu$ on $M$, we define the \emph{Wolff type nonlinear $p$-potential of $\nu$ of order $s$ on $(M,\mu)$}  by 
\begin{equation}\label{E:Wolffpot}
\dot{W}_{s,p}^\nu(x):=\sum_{j=-\infty}^\infty 2^{jsq}\int_{B(x,2^{-j})}\left(\nu(B(y,2^{-j}))\right)^{q-1}\mu(dy),\quad x\in M,
\end{equation}
where $\frac{1}{p}+\frac{1}{q}=1$, see \cite[Definitions 4.5.1 and 4.5.3]{AH96} and \cite{HedbergWolff} for the classical case. We have the following version of \emph{Wolff's inequality}, \cite[Theorem 1]{HedbergWolff}. 

\begin{proposition}\label{P:Wolff}
Let $1<p<+\infty$, assume that (\ref{E:upperd}) and (\ref{E:lowerdadd}) hold and that $0<s<d$. Then there is a constant $c>0$ such that $\int_M (U^s \nu)^q\:d\mu \leq c\:\int_M \dot{W}_{s,p}^\nu\:d\nu$ for any $\nu \in \mathcal{M}_+(M)$.
\end{proposition}

A short proof of Proposition \ref{P:Wolff} is given at the end of this section, it relies on a metric measure space version of the Muckenhoupt-Wheeden inequality, \cite{MuckenhouptWheeden}, shown in \cite[Corollary 2.2]{PerezWheeden}. The definition (\ref{E:Wolffpot}) of $\dot{W}_{s,p}^\nu$ as potentials of homogeneous type is chosen for an easy fit with the Riesz type potentials (\ref{E:Riesz}) and maximal functions used in \cite[Corollary 2.2]{PerezWheeden}. 

Lemma \ref{L:resolventupper} (ii) now follows as in the Euclidean case, \cite[Theorem 5.1.13]{AH96}.

\begin{proof}[Proof of Lemma \ref{L:resolventupper} (ii).]
Suppose that $\mathcal{H}^{d-2p+\varepsilon}(\Sigma)>0$. Then, again by Frostman, there is some $\nu \in\mathcal{M}_+(M)$ with $0<\nu(\Sigma)<+\infty$, $\supp \nu$ compact and contained in $\Sigma$, and such that $\nu(B(x,r))\leq c\:r^{d-2p+\varepsilon}$ for all $x\in \Sigma$ and $r>0$. This implies 
\[(\nu(B(x,2^{-j})))^{q-1}\leq c\:2^{-j(d-2)q}\: 2^{jd}\: 2^{-j\varepsilon q/p}.\]
By (\ref{E:upperd}) and since $\nu(B(x,2^{-j}))\leq \nu(\Sigma)$ for $j< 0$, we see that for any $x\in \Sigma$ we have
\[\dot{W}^\nu_{d-2,p}(x)\leq \sum_{j=0}^\infty 2^{-j\varepsilon q/p}+\nu(\Sigma)^{q/p}\:\sum_{j=-\infty}^{-1} 2^{j(d-2p)q/p}=:S<+\infty.\]
By (\ref{E:resolventupper}) we have $G_\lambda\nu\leq c\:U^{d-2}\nu$, so that $\left\|G_\lambda \nu\right\|_{L^q(M)}\leq c\:S^{1/q}\nu(\Sigma)^{1/q}$ by Proposition \ref{P:Wolff} and the preceding. Using Corollary \ref{C:dualrep} we obtain the contradicting fact that
\[\cpct_{2,p}(\Sigma)^{1/p}\geq \frac{\nu(\Sigma)}{\left\|G_\lambda\nu\right\|_{L^q(M)}}\geq \frac{\nu(\Sigma)^{1/p}}{c\:S^{1/q}}>0.\]
\end{proof}

Proposition \ref{P:Wolff} can be proved following the method in \cite[Corollary 3.6.3 and Theorem 4.5.2]{AH96}. Given $s>0$ and $\nu\in \mathcal{M}_+(M)$, let
\begin{equation}\label{E:maxfct}
\mathcal{M}_s\nu(x):=\sup_{r>0} r^{-s} \nu(B(x,r)),\quad x\in M,
\end{equation}
denote the \emph{maximal function of $\nu$ of order $s$}.

\begin{proof}
From (\ref{E:upperd}) and (\ref{E:lowerdadd}) it is immediate that $\mu$ is a doubling measure (cf. \eqref{E:doubling} below) and 
\[\left(\frac{r}{R}\right)^{-s}\frac{\mu(B(x,r))}{\mu(B(x,R))}\leq c\left(\frac{r}{R}\right)^{d-s},\quad x\in M,\ 0<r<R.\]
We may therefore apply \cite[Corollary 2.2]{PerezWheeden},which shows that with a constant $c>0$ we have
\begin{equation}\label{E:PW}
\left\|U^{s}\nu\right\|_{L^q(M)}^q\leq c\left\|\mathcal{M}_{s}\nu\right\|_{L^q(M)}^q
\end{equation}
for all $\nu\in \mathcal{M}_+(M)$. 

For any $r>0$ we can find $j\in\mathbb{Z}$ such that $2^{-j-1}<r\leq 2^{-j}$, and if $\nu \in \mathcal{M}_+(M)$ and $x\in M$ then
$r^{-s}\nu(B(x,r))\leq 2^{-s} 2^{js}\nu(B(x,2^{-j}))$,
and taking suprema, we see that
\[\mathcal{M}_s\nu(x)\leq c\:\left\|\left( 2^{js} \nu(B(x,2^{-j}))\right)_{j\in\mathbb{Z}}\right\|_{\ell^\infty}\leq c_\mu\:\left\|\left( 2^{js} \nu(B(x,2^{-j}))\right)_{j\in\mathbb{Z}}\right\|_{\ell^q}  \]
for any $x\in M$. Taking $q$-th moments, using (\ref{E:upperd}) and Fubini, we obtain
\begin{align}
\left\|\mathcal{M}_s\nu\right\|_{L^q(M)}^q&\leq c\:\int_M \sum_{j=-\infty}^\infty 2^{jsq}\left(\nu(B(y,2^{-j}))\right)^q\mu(dy)\notag\\
&=c\: \sum_{j=-\infty}^\infty 2^{jsq} \int_M \left(\nu(B(y,2^{-j}))\right)^{q-1} \int_M \mathbf{1}_{B(y,2^{-j})}(x) \nu(dx)\mu(dy)\notag\\
&=\int_M\left(\sum_{j=-\infty}^\infty 2^{jsq}\int_{B(x,2^{-j})}\left(\nu(B(y,2^{-j}))\right)^{q-1}\mu(dy)\right)\nu(dx)\notag\\
&=\int_M \dot{W}_{s,p}^\nu d\nu.\notag
\end{align}
Combination with (\ref{E:PW}) now yields the result.
\end{proof}

\section{Riemannian manifolds}\label{S:manifolds}

We provide results on essential self-adjointness and $L^p$-uniqueness on Riemannian manifolds after the removal of a set $\Sigma$.

\subsection{Essential self-adjointness and capacities}\label{SS:weighted}

Let $(M,\mathbf{g},\mu)$ be a weighted manifold in the sense of \cite[Definition 3.17]{Grigoryan2009}, that is,
$M=(M,\mathbf{g})$ is a complete Riemannian manifold and $\mu$ a Borel measure which has a smooth and positive density with respect to the Riemannian volume. Riemannian manifolds appear as the special case that the density is identically one, so that $\mu$ is the Riemannian volume. The distance $\varrho$ as in Section \ref{S:Setup} is the geodesic distance on $M$. We point out that $M$ is assumed to be second countable, \cite[Definition 3.2]{Grigoryan2009}. Throughout the entire section we silently assume that $M$ is connected.

We consider the Dirichlet integral on $M$, defined by
\[\mathcal{E}(f,g)=\int_M\left\langle \nabla f, \nabla g\right\rangle_{TM}\:d\mu,\]
where $f,g$ are elements of $\mathcal{D}(\mathcal{E})=W_0^1(M)$, defined as the closure of the space $C_c^\infty(M)$ of smooth compactly supported functions on $M$ in the space $W^1(M)$ of all $u\in L^2(M)$ with  $|\nabla u|\in L^2(M)$ and endowed with the Hilbert norm determined by 
\[\left\|u\right\|_{W^1(M)}^2=\left\|u\right\|_{L^2(M)}^2+\left\|\nabla u\right\|_{L^2(M)}^2.\] 
Clearly $\Gamma(f,g)(x)=\left\langle \nabla f(x), \nabla g(x)\right\rangle_{T_xM}$, $x\in M$. The operator $(\mathcal{L},\mathcal{D}(\mathcal{L}))$ is the Dirichlet Laplacian on $M$, i.e. the Friedrichs extension of the classical Laplace operator $\De_\mu|_{C_c^\infty(M)}$, \cite[Section 3.6]{Grigoryan2009}. If $M$ is complete, then this self-adjoint extension is unique, see \cite[Theorem 11.5]{Grigoryan2009} or \cite[Theorem 2.4]{Strichartz83}. The domain $\mathcal{D}(\mathcal{L})$ of $\mathcal{L}$ is 
\begin{equation}\label{E:domofL}
W_0^2(M)=\left\lbrace u\in W_0^1(M): \Delta_\mu u\in L^2(M)\right\rbrace, 
\end{equation}
\cite[Theorem 4.6]{Grigoryan2009}. With the choice $\mathcal{A}(U)=C_c^\infty(U)$ for any open $U\subset M$ conditions \eqref{A:bump}, \eqref{A:basic} and $(C_2)$ are satisfied. In the sequel $\Sigma$ will always denote a subset of $M$ and $\mathring{M}:=M\setminus \Sigma$ its complement. 

The following characterization of the critical size of $\Sigma$ in terms of the capacities $\ccpct_{2,2}$ is immediate from Theorem \ref{T:esa}.

\begin{theorem}\label{T:mfdccpct} Let $M$ be a complete weighted manifold and $\Sigma\subset M$ a closed subset. Then we have $\ccpct_{2,2}^{C_c^\infty(M)}(\Sigma)=0$ if and only if $\mu(\Sigma)=0$ and $\Delta_\mu|_{C_c^\infty(\mathring{M})}$ is essentially self-adjoint.
\end{theorem}

Additional conditions allow a result in terms of the capacities $\cpct_{2,2}$. The measure $\mu$ is said to satisfy the \emph{doubling condition} \eqref{E:doubling} if there is a constant $c_D>1$ such that 
\begin{equation}\label{E:doubling}\tag{D}
\mu(B(x,2r))\leq c_D\mu(B(x,r)),\quad x\in M,\quad r>0.
\end{equation}
It is well-known that the heat semigroup $(P_t)_{t>0}$ on $M$ is absolutely continuous and strong Feller, see for example
\cite[Theorems 7.13 and 7.15]{Grigoryan2009}. Its heat kernel $p_t(x,y)$ is said to satisfy the \emph{gradient upper estimate} \eqref{E:gradest} if there is some $c>0$ such that 
\begin{equation}\label{E:gradest}\tag{G}
|\nabla_y p_t(x,\cdot)|(y)\leq \frac{c}{\sqrt{t}\mu(B(x,\sqrt{t}))}
\end{equation}
for all $x,y\in M$ and $t>0$. By \cite[Proposition 2.1]{CS10}, \cite[Theorem 1.1]{Grigoryan1997} and \cite{LY86}, \eqref{E:doubling} and \eqref{E:gradest} imply \emph{Li-Yau type heat kernel estimates}, \cite{Grigoryan1991, LY86, Saloff-Coste1991}, that is,
\begin{equation}\label{E:LY}\tag{LY}
\frac{1}{c_8 \mu(B(y,\sqrt{t})}\:\exp\Big(-c_9\:\frac{\varrho(x,y)^2}{t}\Big)\leq p_t(x,y)\leq \frac{c_8}{\mu(B(y,\sqrt{t})}\:\exp\Big(-\frac{\varrho(x,y)^2}{c_9 t}\Big)
\end{equation}
for all $x,y\in M$ and $t>0$ with universal constants $c_8>1$ and $c_9>1$; see \cite[p. 687]{CS10}.

\begin{theorem}\label{T:mfdcpct} Let $M$ be a complete weighted manifold and $\Sigma\subset M$ a closed subset.
\begin{enumerate}
\item[(i)]  If $\mu(\Sigma)=0$ and $\Delta_\mu|_{C_c^\infty(\mathring{M})}$ is essentially self-adjoint, then $\cpct_{2,2}(\Sigma)=0$.
\item[(ii)] Suppose that \eqref{E:doubling} and \eqref{E:gradest} hold. If $\cpct_{2,2}(\Sigma)=0$, then $\mu(\Sigma)=0$ and $\Delta_\mu|_{C_c^\infty(\mathring{M})}$ is essentially self-adjoint.
\end{enumerate}
\end{theorem}

\begin{proof}
Statement (i) follows from Corollary \ref{C:firstcomparison} and Theorem \ref{T:mfdccpct}. Under \eqref{E:doubling} condition \eqref{E:gradest} is equivalent to the estimate
\begin{equation}\label{E:gradestfull}
|\nabla_y p_t(x,\cdot)|(y)\leq \frac{c}{\sqrt{t}\mu(B(x,\sqrt{t}))}\:\exp\Big(-\frac{\varrho(x,y)^2}{c't}\Big),
\end{equation}
valid for all $x,y\in M$ and $t>0$; here $c$ and $c'$ are universal positive constants. This follows using the upper estimate in \eqref{E:LY} together with \cite[Theorem 1.1]{Dungey06}. Combining the lower estimate in (\ref{E:LY}) with \eqref{E:gradestfull} and using \eqref{E:doubling}, we obtain \eqref{E:loghk} with $\alpha=1$. Using the hypothesis and Corollary \ref{C:secondcomparison} we see that $\ccpct_{2,2}^{\mathcal{A}}(\Sigma)=0$ with $\mathcal{A}=\mathcal{D}(\mathcal{L})\cap C_b^\infty(M)$. 

We claim that this implies $\ccpct_{2,2}^{C_c^\infty(M)}(\Sigma)=0$. If so, then (ii) follows by Theorem \ref{T:mfdccpct}. By (\ref{E:innerreg}) it suffices to prove the claim for compact $\Sigma$. Let $\varphi\in \omega_\Sigma^{C_c^\infty(M)}$ and given $\varepsilon>0$, choose $v\in \omega_\Sigma^\mathcal{A}$ such that $\left\|v\right\|_{\mathcal{D}(\mathcal{L})}^2\leq \varepsilon$. Then $\varphi v\in \omega_\Sigma^{C_c^\infty(M)}$ and 
\begin{align}\label{E:passover} 
\left\|\varphi v\right\|_{L^2(M)}+\left\|\mathcal{L}(\varphi v)\right\|_{L^2(M)}&=\left\|\varphi\right\|_{L^\infty(M)}\left\|v\right\|_{L^2(M)}+2\big\|\Gamma(\varphi)^{1/2}\big\|_{L^\infty(M)} \mathcal{E}(v)^{1/2}\\
&\qquad + \left\|\varphi\right\|_{L^\infty(M)}\left\|\mathcal{L}v\right\|_{L^2(M)}+\left\|\Delta_\mu \varphi\right\|_{L^\infty(M)}\left\|v\right\|_{L^2(M)}\notag\\
&\leq c_\varphi\Big(\left\|v\right\|_{L^2(M)}+\left\|\mathcal{L}v\right\|_{L^2(M)}\Big)\notag
\end{align}
with 
\[c_\varphi:=2\big(\left\|\varphi\right\|_{L^\infty(M)}+\big\|\Gamma(\varphi)^{1/2}\big\|_{L^\infty(M)}+\left\|\Delta_\mu \varphi\right\|_{L^\infty(M)}\big);\]
we have used (\ref{E:Gamma}). Consequently $\left\|\varphi v\right\|_{\mathcal{D}(\mathcal{L})}^2\leq c\:\left\|v\right\|_{\mathcal{D}(\mathcal{L})}^2$ and therefore also
$\ccpct_{2,2}^{C_c^\infty(M)}(\Sigma)\leq c\:\varepsilon$ with $c>0$ independent of $\varepsilon$, what proves the claim.
\end{proof}

\begin{remark} Another proof of \eqref{E:gradestfull} could be formulated using \cite[Theorem 4.9]{CS08} together with \eqref{E:doubling} and \eqref{E:LY}. 
\end{remark}

\subsection{$L^p$-uniqueness and capacities}

Suppose that $M$ is a complete Riemannian manifold and $1<p<+\infty$. Clearly also condition \eqref{A:basicp} is satisfied for $\mathcal{A}(U)=C_c^\infty(U)$, see \cite[Theorem 3.5]{Strichartz83}. For compact $M$ also the validity of condition (\ref{A:Riesztrafo}) is well-known, see \cite{Seeley67} and \cite[Section 6]{Strichartz83}. A sufficient condition for the validity of (\ref{A:Riesztrafo}) on a general (possibly non-compact) manifold $M$ is the \emph{$L^p$-gradient bound} 
\begin{equation}\label{E:gradestp}\tag{$\mathrm{G_p}$}
\left\||\nabla P_t f|\right\|_{L^p(M)} \leq \frac{c(p)}{\sqrt{t}}\:\left\|f\right\|_{L^p(M)},\quad f\in C_c^\infty(M),
\end{equation}
for the heat semigroup $(P_t)_{t>0}$; the sufficiency follows from \cite[Proposition 3.6]{CS10} and its proof. Further sufficient conditions could be formulated using the results in \cite{ACDH04} and \cite{CD99}. The following is then immediate from Theorem \ref{T:esa}.

\begin{theorem}\label{T:Lpuni}
Let $M$ be a complete Riemannian manifold, $\Sigma\subset M$ closed, and $1<p<+\infty$.
\begin{enumerate}
\item[(i)] If $\mu(\Sigma)=0$ and $\Delta_\mu|_{C_c^\infty(\mathring{M})}$ is $L^p$-unique, then $\ccpct_{2,p}^{C_c^\infty(M)}(\Sigma)=0$.
\item[(ii)] Suppose that $M$ is compact or \eqref{E:gradestp} holds. If $\ccpct_{2,p}^{C_c^\infty(M)}(\Sigma)=0$, then $\mu(\Sigma)=0$ and $\Delta_\mu|_{C_c^\infty(\mathring{M})}$ is $L^p$-unique on $M$. 
\end{enumerate}
\end{theorem}

It is straightforward to see that \eqref{E:log} implies \eqref{E:gradestp}, therefore also \eqref{E:loghk} implies it and the following can be seen similarly as Theorem \ref{T:mfdcpct}.

\begin{theorem}\label{T:mfdcpctp} Let $M$ be a complete Riemannian manifold, $\Sigma\subset M$ closed, and $1<p<+\infty$.
\begin{enumerate}
\item[(i)] If $\mu(\Sigma)=0$ and $\Delta_\mu|_{C_c^\infty(\mathring{M})}$ is $L^p$-unique, then $\cpct_{2,p}(\Sigma)=0$.
\item[(ii)] Suppose that \eqref{E:doubling} and \eqref{E:gradest} hold. If $\Sigma$ is closed and $\cpct_{2,p}(\Sigma)=0$, then $\Delta_\mu|_{C_c^\infty(\mathring{M})}$ is $L^p$-unique.
\end{enumerate}
\end{theorem}

\subsection{Localized arguments for compact sets}
For compact $\Sigma$ we can rely on localized estimates to arrive at a variant of Theorem \ref{T:mfdcpct} (ii) that does not need \eqref{E:doubling} or \eqref{E:gradest}.  We establish it using truncated Laplace transforms of the semigroup and related capacities. Let $B\subset M$ be a domain (a connected open subset) with smooth boundary $\partial B$ and let $p_t^B(x,y)$ denote the Neumann heat kernel on $B$. Given $0<T\leq+\infty$ and $\lambda>0$ let 
\[g_\lambda^{B,T}(x,y):=\int_0^T e^{-\lambda t}p_t^B(x,y)\:dt,\quad x,y\in B,\]
and consider 
\begin{equation}\label{E:truncres}
G_\lambda^{B,T} f(x):=\int_B g_\lambda^{B,T}(x,y)f(y)\:\mu(dy)
\end{equation}
for $f\in L^2(B)$. Obviously $G_\lambda^B:=G^{B,\infty}_\lambda$ is the $\lambda$-resolvent operator for the Neumann Laplacian $(\mathcal{L}^{B},\mathcal{D}(\mathcal{L}^{B}))$ on $B$. Also (\ref{E:truncres}) induces bounded linear operators $G^{B,T}_\lambda$ on $L^p(B)$, $1\leq p\leq +\infty$. For any $1\leq p<+\infty$ and $f\in L^p(B)$ we have $G^{B,T}_\lambda f\in \mathcal{D}(\mathcal{L}^{B,(p)})$, where $\mathcal{L}^{B,(p)}$ is the $L^p$-realization of the Neumann Laplacian on $B$,
and $(\lambda-\mathcal{L}^{B,(p)})G^{B,T}_\lambda f=f-e^{-\lambda T}P^B_T f$, where $(P_t^B)_{t>0}$ denotes the Neumann heat semigroup on $B$.

Let $\Omega\subset B$ be an open set. For compact $K\subset \Omega$ we define
\begin{multline}
\cpct_{2,p}^{B,T}(K,\Omega):=\inf\Big\lbrace \left\|f\right\|_{L^p(B)}^p: f\in L_+^p(B)\ \text{with $f=0$ $\mu$-a.e. on $B\setminus \Omega$}\notag\\
\text{and}\quad G^{B,T}_\lambda f(x)\geq 1\ \text{for all $x\in K$}\Big\rbrace.
\end{multline}

We provide variants of Theorem \ref{T:Mazja} and Corollary \ref{C:secondcomparison} in terms of the operators $G^{B,T}$ and the capacities $\cpct_{2,p}^{B,T}(\cdot,\Omega)$. The boundary $\partial B\subset M$ of $B$ is said to be \emph{infinitesimally convex}, \cite{Bishop}, if its second fundamental form is nonnegative definite at all of its points.

\begin{theorem}\label{T:Mazjaloc}
Let $M$ be a complete Riemannian manifold and $B\subset M$ a smooth bounded domain such that $\partial B$ is infinitesimally convex. Let $\Omega$ be open and such that $\overline{\Omega}\subset B$ and write $T:=\dist(\overline{\Omega},\partial B)^2$. Suppose that $F\in C^2(\mathbb{R}_+)$ is a function satisfying (\ref{E:Mazja}). Then for any $1< p< +\infty$, $\lambda>\frac{1}{T}$ and nonnegative $f\in L^p(\Omega)$ we have $F\circ G_\lambda^{B,T} f\in\mathcal{D}(\mathcal{L}^{B,(p)})$ and 
\begin{equation}\label{E:truncationloc}
\big\|F\circ G_\lambda^{B,T} f\big\|_{\mathcal{D}(\mathcal{L}^{B,(p)})}\leq c_3\left\|f\right\|_{L^p(B)}.
\end{equation}
with a constant $c_3>0$ independent of $f$. 
\end{theorem}

To prove Theorem \ref{T:Mazjaloc} we use a variant of Lemma \ref{L:multest}.
\begin{lemma}\label{L:multestloc}
Let $f\in C_c(\Omega)$ and suppose all other hypotheses of Theorem \ref{T:Mazjaloc} are in force. Then there is a constant $c_6>0$, independent of $f$, such that
\begin{equation}\label{E:multestloc}
\int_0^T e^{-\lambda t}\sqrt{\Gamma(P_t^B f)}(x)dt\leq c_5\:\Big(G_\lambda^{B,T}f(x)\Big)^{1/2}\Big(\sup_{t>0}P_t^Bf(x)\Big)^{1/2}
\end{equation}
for all $x\in B$.
\end{lemma}

Since $M$ is complete $B$ is relatively compact and can therefore be covered by finitely many
relatively compact coordinate charts. Consequently we can find $c>1$ such that 
\begin{equation}\label{E:localEuclid}
\frac{1}{c}\:r^d\leq \mu(B(y,r))\leq c\:r^d,\quad y\in \Omega,\quad 0<r<\sqrt{T},
\end{equation}
where $d$ is the dimension of $M$. There are constants $c_8>1$ and $c_9>1$ such that 
\begin{equation}\label{E:LYloc}
\frac{1}{c_8}\:t^{-d/2}\:\exp\Big(-c_9\:\frac{\varrho(x,y)^2}{t}\Big)\leq p_t^B(x,y)\leq c_8\:t^{-d/2}\:\exp\Big(-\frac{\varrho(x,y)^2}{c_9 t}\Big)
\end{equation}
for all $0<t<T$, $y\in \Omega$ and $x\in B$. The upper bound in (\ref{E:LYloc}) is due to \cite[Theorem 3.2]{LY86}, applied to $\overline{B}$ and combined with \eqref{E:localEuclid}. The lower bound in (\ref{E:LYloc}) is a consequence of \cite[Theorem 2.3]{LY86} together with standard arguments, \cite{Da89}, involving the conservativeness of the Neumann heat semigroup on $\overline{B}$ and \eqref{E:localEuclid}. Since for any $0<t<T$ and any $y\in \Omega$ we have $\sqrt{t}<\dist(y,\partial B)$, \cite[Theorem 2.1]{Zhang06} yields the logarithmic gradient estimate
\begin{equation}\label{E:loggradestZh}
\frac{|\nabla_y p_t^B(x,\cdot)|(y)}{p_t^B(x,y)}=|\nabla_y\log p_t^B(x,\cdot)|(y)\leq \frac{c_{B,T}}{\sqrt{t}}\Big(1+\frac{\varrho(x,y)}{\sqrt{t}}\Big)
\end{equation}
for all $x\in B$, $y\in \Omega$ and $0<t<T$; here $c_{B,T}>0$ is a constant depending on $B$ and $T$. See also \cite{Hsu99, Sheu91} and the references cited there. We use these ingredients to prove Lemma \ref{L:multestloc}.

\begin{proof}
Set $\chi_t(x,y):=1+\frac{\varrho(x,y)}{\sqrt{t}}$. By \eqref{E:loggradestZh} we have 
\[\int_0^T e^{-\lambda t}\int_B|\nabla_y p_t^B(x,\cdot)|(y)f(y)\mu(dy)dt\leq c_{B,T}\int_0^T e^{-\lambda t}t^{-1/2}\int_B\chi_t(x,y)p_t^B(x,y)f(y)\mu(dy)dt,\]
and for any $0<\delta<T$ we can split the integral on the right hand side similarly as in (\ref{E:Hedberg}) into two integrals over $(0,\delta)\times B$ and $[\delta,T)\times B$, respectively. H\"olders inequality with $1<\beta<2$ and $\frac{1}{\beta}+\frac{1}{\beta'}=1$ and with respect to $\mu(dy)dt$ gives
\begin{multline}\label{E:deltatarget}
\int_0^\delta e^{-\lambda t}t^{-1/2}\int_B\chi_t(x,y)p_t^B(x,y)f(y)\mu(dy)dt\\
\leq \Big(\int_0^\delta\int_B e^{-\lambda t}t^{-\beta/2}\chi_t(x,y)p_t^B(x,y)f(y)\mu(dy)dt\Big)^{1/\beta}\Big(\int_0^\delta\int_B e^{-\lambda t}p_t^B(x,y)f(y)\mu(dy)dt\Big)^{1/\beta'}.
\end{multline}
Using \eqref{E:LYloc} it follows that for any $\gamma>1$ and $\frac{1}{\gamma}+\frac{1}{\gamma'}=1$ we have 
\begin{align}
\chi_t(x,y)p_t^B(x,y)&\leq c_8\:t^{-d/2}\:\chi_t(x,y)\:\exp\Big(-c_9\:\frac{\varrho(x,y)^2}{c_9^2 \gamma t}\Big)\:\exp\Big(-\frac{\varrho(x,y)^2}{c_9 \gamma' t}\Big)\notag\\
&\leq \frac{c}{c_8}\:(c_9^2\gamma t)^{-d/2}\:\exp\Big(-c_9\:\frac{\varrho(x,y)^2}{c_9^2 \gamma t}\Big)\notag\\
&\leq c\:p_{c_9^2\gamma t}^B(x,y)\notag
\end{align}
with a constant $c>0$ independent of $x$, $y$ and $t$. For the second inequality note that the function $s\mapsto (1+s)\exp\big(-\frac{s^2}{c_9\gamma'}\big)$ is bounded on $[0,+\infty)$. Consequently the right hand side of (\ref{E:deltatarget}) is bounded by 
\[\Big(c\int_0^\delta e^{-\lambda t}t^{-\beta/2}P_{c_9^2\gamma t}^Bf(x)dt\Big)^{1/\beta}\Big(G_\lambda^{B,T}f(x)\Big)^{1/\beta'}\leq c'\:\delta^{1/\beta' -1/2}\Big(\sup_{t>0} P_t^B f(x)\Big)^{1/\beta}\Big(G_\lambda^{B,T}f(x)\Big)^{1/\beta'}. \]
We can proceed similarly for the integral over $[\delta,T)$. Choosing $\delta=G_\lambda^{B,T}f(x)/\sup_{t>0} P_t^Bf(x)$ and using the estimates in the proof of Proposition \ref{P:hktosg} we arrive at (\ref{E:multestloc}).
\end{proof}

Now Theorem \ref{T:Mazjaloc} follows quickly.
\begin{proof}
For $f\in C_c(\Omega)$ Proposition \ref{P:integral} and Lemma \ref{L:chainforL} hold with $G^{B,T}_\lambda f$ in place of $G_\lambda f$. As in the proof of Theorem \ref{T:Mazja} we obtain $\mathcal{L}^{B}F(u)\in L^p(M)$ for $u=G^{B,T}_\lambda f$ and arbitrary nonnegative $f\in C_c(\Omega)$ if $1<p\leq +\infty$, and for $1<p<+\infty$ also $F(u)\in \mathcal{D}(\mathcal{L}^{B,(p)})$. Similarly as in (\ref{E:finalestimate}) we see that such $f$ satisfy $\left\|(\lambda-\mathcal{L}^{B})F(u)\right\|_{L^p(B)}\leq (3+e^{-\lambda T}+c_5)L\left\|f\right\|_{L^p(B)}$. The extension to $f\in L^p(\Omega)$ follows as in the proof of Theorem \ref{T:Mazja}.
\end{proof}

As a consequence, we obtain a variant of Corollary \ref{C:secondcomparison} respectively Theorem \ref{T:mfdcpct} (ii). Here and in later occurrences of $G_\lambda^{B,T}$ or $\cpct_{2,p}^{B,T}(\cdot,\Omega)$ we agree to always use $\lambda>\frac{1}{T}$.

\begin{corollary}\label{C:simplify}
Let $M$ be a complete Riemannian manifold and $B\subset M$ a smooth bounded domain such that $\partial B$ is infinitesimally convex. Let $\Omega$ be open and such that $\overline{\Omega}\subset B$ and write $T:=\dist(\overline{\Omega},\partial B)^2$.
If  $\Sigma\subset \Omega$ is closed and $\cpct_{2,2}^{B,T}(\Sigma,\Omega)=0$, then $\ccpct_{2,2}^{C_c^\infty(M)}(\Sigma)=0$. In this case $\Delta_\mu|_{C_c^\infty(\mathring{M})}$ is essentially self-adjoint.
\end{corollary}

\begin{proof}
Given $\varepsilon>0$ let $f\in L^2(B)$ be nonnegative, zero $\mu$-a.e. on $B\setminus \Omega$, with $G^{B,T}_\lambda f>1$ on $\Sigma$ and such that $\left\|f\right\|_{L^2(B)}^2<\varepsilon$. Let $(f_n)_{n\geq 1}\subset C_c(\Omega)$ an increasing sequence of nonnegative functions approximating $f$ pointwise from below. As in the proof of Corollary \ref{C:secondcomparison} we can find $n$ such that $G_\lambda^{B,T} f_n > 1$ on $\Sigma$ and $\left\|f_n\right\|_{L^2(B)}^2< \varepsilon$. If $F$ is a $C^2$-truncation, then the function $v:=F \circ G_\lambda^{B,T} f_n$ is in $\mathcal{D}(\mathcal{L}^{B})\cap C^\infty(B)$ and satisfies
$v = 1$ on a neighborhood of $\Sigma$. By \eqref{E:truncationloc} we have $\left\|v\right\|_{\mathcal{D}(\mathcal{L}^{B})}^2\leq c_3^2 \left\|f_n\right\|_{L^2(B)}^2$. Now choose $\varphi\in \omega_\Sigma^{C_c^\infty(\Omega)}$. Then $\varphi v\in \omega_\Sigma^{C_c^\infty(\Omega)}$, and proceeding as in \eqref{E:passover} we obtain $\left\|\varphi v\right\|_{\mathcal{D}(\mathcal{L}^{B})}^2\leq c\:\left\|v\right\|_{\mathcal{D}(\mathcal{L}^{B})}^2$, what implies $\ccpct_{2,2}^{C_c^\infty(M)}(\Sigma)<c\:\varepsilon$ with $c>0$ independent of $\varepsilon$.
\end{proof}

For later use we provide variants of Lemmas \ref{L:resolventlower} and \ref{L:resolventupper}. Recall \eqref{E:hp}.

\begin{lemma}\label{L:resolventlowerloc}
Let $1<p<+\infty$, let $M$ be a complete Riemannian manifold of dimension $d$ and $B\subset M$ a smooth bounded domain such that $\partial B$ is infinitesimally convex. Let $\Omega$ be open and such that $\overline{\Omega}\subset B$ and write $T:=\dist(\overline{\Omega},\partial B)^2$. Suppose that $\Sigma\subset \Omega$ is closed. If $d>2p$ and $\mathcal{H}^{d-2p}(\Sigma)<+\infty$, then $\cpct_{2,p}^{B,T}(\Sigma,\Omega)=0$. If $d=2p$ and $\mathcal{H}^{h_p}(\Sigma)<+\infty$, then we also have $\cpct_{2,p}^{B,T}(\Sigma,\Omega)=0$.
\end{lemma}

\begin{proof}
Integrating the lower estimate in \eqref{E:LYloc} gives $g_\lambda^{B,T}(y,z)\geq c\:\varrho(y,z)^{2-d}$ for all $y,z\in \Omega$. One can now follow the proof of Lemma \ref{L:resolventlower} using \eqref{E:localEuclid}, the obvious analogs of (\ref{E:MazjaKhavin}) and (\ref{E:nonlinFubini}) and the following variant of Corollary \ref{C:dualrep}. 
\end{proof}

\begin{corollary}\label{C:dualreploc}
Let $M$, $B$, $\Omega$ and $T$ be as in Lemma \ref{L:resolventlowerloc}. Then for any  $1<p<+\infty$ and any compact set $K\subset \Omega$ we have
\[\cpct_{2,p}^{B,T}(K,\Omega)^{1/p}=\sup\big\lbrace \nu(K):\ \nu\in\mathcal{M}_+(K),\ \big\|G_\lambda^{B,T} \nu\big\|_{L^q(\Omega)}\leq 1\big\rbrace,\]
where $\frac{1}{p}+\frac{1}{q}=1$.
\end{corollary}

The proof is similar as in \cite[Theorem 2.5.1]{AH96}, we sketch it for convenience.
\begin{proof}
Let $Y$ be the set of all $f\in L^p_+(B)$ with $f=0$ $\mu$-a.e. on $B\setminus \Omega$  with $\left\|f\right\|_{L^p(\Omega)}\leq 1$ and let $\mathcal{P}(K)$ be the set of all Borel probability measures on $K$. We have $\sup_{f\in Y}\int_B G_\lambda^{B,T}\nu f\:d\mu=\big\|G_\lambda^{B,T}\nu\big\|_{L^q(\Omega)}$ and consequently 
\[\min_{\nu \in \mathcal{P}(K)}\sup_{f\in Y}\int_B G_\lambda^{B,T}\nu f\:d\mu= \min_{\nu\in\mathcal{M}_+(K)}\frac{\big\|G_\lambda^{B,T}\nu\big\|_{L^q(\Omega)}}{\nu(K)}.\]
On the other hand, $\min_{\nu\in \mathcal{P}(K)}\int_B G_\lambda^{B,T}\nu\:f\:d\mu=\min_{x\in K} G_\lambda^{B,T} f(x)$. The set $\mathcal{P}(K)$ is convex, and it is compact with respect to weak convergence. Also $Y$ is convex. For any fixed $f\in Y$ the map $\nu\mapsto 
\int_M G_\lambda^{B,T}\nu f\:d\mu$ is lower semicontinuous on $\mathcal{P}(K)$. Therefore the minimax theorem, \cite[Theorem 2.4.1]{AH96}, allows to conclude that
\begin{align}
\min_{\nu\in\mathcal{M}_+(K)}&\frac{\big\|G_\lambda^{B,T}\nu\big\|_{L^q(\Omega)}}{\nu(K)}
=\sup\Big\lbrace \frac{\min_{x\in K} G_\lambda^{B,T} f(x)}{\left\|f\right\|_{L^p(\Omega)}}:\ f\in L^p(B),\ f=0\ \text{$\mu$-a.e. on $B\setminus \Omega$}\Big\rbrace\notag\\
&=\sup\big\lbrace \left\|f\right\|_{L^p(\Omega)}^{-1}: \quad f\in L^p(B),\ f=0\ \text{$\mu$-a.e. on $B\setminus \Omega$ and $G^{B,T}_\lambda f\geq 1$ on $K$}\big\rbrace\notag\\
&=\cpct_{2,p}^{B,T}(K,\Omega)^{-1/p}.\notag
\end{align}
\end{proof}

We write $\cpct_{2,2}^B(\cdot,\Omega):=\cpct_{2,2}^{B,\infty}(\cdot,\Omega)$.

\begin{lemma}\label{L:resolventupperloc}
Let $M$ be a complete Riemannian manifold of dimension $d$ and $B\subset M$ a smooth bounded domain such that $\partial B$ is infinitesimally convex. Let $\Omega$ be open and such that $\overline{\Omega}\subset B$ and write $T:=\dist(\overline{\Omega},\partial B)^2$. There is some $\lambda_{B,T}>0$ such that for all $\lambda>\lambda_{B,T}$ and any closed set $\Sigma\subset \Omega$ we have the following: If $d>4$, $\cpct_{2,2}^{B}(\Sigma,\Omega)=0$ and $\varepsilon>0$, then $\mathcal{H}^{d-4+\varepsilon}(\Sigma)=0$. If $d=4$, $\cpct_{2,2}^{B}(\Sigma,\Omega)=0$ and $h$ is a Hausdorff function satisfying \eqref{E:condh}, then $\mathcal{H}^h(\Sigma)=0$.
\end{lemma}

\begin{proof}
We first claim that the density $g_\lambda^B:=g_\lambda^{B,\infty}$ of $G_\Lambda^B$ satisfies
\begin{equation}\label{E:resupperauto}
g_\lambda^B(x,y)\leq c\:\varrho(x,y)^{2-d}\quad x\in B,\quad y\in \Omega
\end{equation}
with a universal constant $c>0$. By \cite[Theorem 3.2]{LY86} we have 
\[p_t^B(x,y)\leq \frac{c_8\:e^{c_{10} t}}{\mu(B(y,\sqrt{t}))}\exp\big(-\frac{\varrho(x,y)^2}{c_9\:t}\big)\]
for all $x,y\in B$ and $t>0$ with universal positive constants $c_8$, $c_9$ and $c_{10}$. Using \eqref{E:localEuclid} and choosing $\lambda>\lambda_{B,T}:=\frac{1}{T}\vee c_{10}$, we obtain
\begin{align}
g_\lambda^B(x,y)&\leq c_8\:\int_0^Te^{-(\lambda-c_{10})t}t^{-d/2}e^{-\varrho(x,y)^2/c_9t}dt+\frac{c_8}{\mu(B)}\int_T^\infty e^{-(\lambda-c_{10})t}e^{-\varrho(x,y)^2/c_9t}dt\notag\\
&\leq c\:\varrho(x,y)^{2-d}+c'\notag
\end{align}
for all $x\in B$ and $y\in \Omega$, where $c$ and $c'$ are positive constants not depending on $x$ or $y$. Since $B$ is bounded, \eqref{E:resupperauto} follows by readjusting constants.

Now we can proceed as in Lemma \ref{L:resolventupper}: If $d>4$ and we would have $\mathcal{H}^{d-4+\varepsilon}(\Sigma)>0$, then Frostman's lemma would produce nonzero finite $\nu\in \mathcal{M}_+(M)$ with support inside $\Sigma$ such that $\nu(B(x,r))\leq c\:r^{d-4+\varepsilon}$, $x\in \Sigma$, $r>0$. Similarly as before we could use \eqref{E:localEuclid} to see that 
\[\big\|G_\lambda^{B}\nu\big\|_{L^2(\Omega)}^2\leq c\:\int_\Omega\int_\Omega \varrho(x,z)^{4-d}\nu(dx)\nu(dz)<+\infty,\] contradicting $\cpct_{2,2}^{B}(\Sigma,\Omega)=0$. The argument for the limit case $d=4$ is analogous.
\end{proof}

\subsection{Essential self-adjointness and Hausdorff measures}

For the case $p=2$ we provide a characterization for essential self-adjointness in terms of the Hausdorff measure and dimension of $\Sigma$. 

\begin{theorem}\label{T:Hausdorff}\mbox{}
Let $M$ be a complete Riemannian manifold of dimension $d\geq 4$ and $\Sigma\subset M$ a closed subset.
\begin{enumerate}
\item[(i)] If $\Sigma$ is of zero measure and such that $\Delta_\mu|_{C_c^\infty(\mathring{M})}$ is essentially self-adjoint, then $\mathcal{H}^{d-4+\varepsilon}(\Sigma)=0$ for all $\varepsilon>0$ (if $d>4$) respectively $\mathcal{H}^h(\Sigma)=0$ for any Hausdorff function $h$ satisfying \eqref{E:condh} (if $d=4$).
\item[(ii)] If $\mathcal{H}^{d-4}(\Sigma)<+\infty$ (if $d>4$) respectively $\mathcal{H}^{h_2}(\Sigma)<+\infty$ (if $d=4$), then $\Delta_\mu|_{C_c^\infty(\mathring{M})}$ is essentially self-adjoint. 
\end{enumerate}
\end{theorem}

\begin{proof}
If $\Sigma$ is as in (i), then $\ccpct_{2,2}^{C_c^\infty(M)}(\Sigma)=0$ by Theorem \ref{T:esa}, hence $\ccpct_{2,2}^{C_c^\infty(M)}(\Sigma_0)=0$ for any compact subset $\Sigma_0\subset \Sigma$. If we could prove that $\mathcal{H}^{d-4+\varepsilon}(\Sigma_0)=0$ for any such $\Sigma_0$, the inner regularity of $\mathcal{H}^{d-4+\varepsilon}$ on $\Sigma$ would imply the result of (i). Similarly, if $\Sigma$ is as in (ii), then by \eqref{E:innerreg} we can find an increasing sequence $(\Sigma_i)_{i\geq 1}$ of 
compact sets $\Sigma_i\subset \Sigma$ with $\ccpct_{2,2}^{C_c^\infty(M)}(\Sigma)=\lim_i \ccpct_{2,2}^{C_c^\infty(M)}(\Sigma_i)$. By hypothesis we have $\mathcal{H}^{d-4}(\Sigma_i)<+\infty$ for any $i$, and if we could conclude that $\ccpct_{2,2}^{C_c^\infty(M)}(\Sigma_i)=0$, this would give the result of (ii). We may therefore assume that $\Sigma$ is compact.

By compactness we can find finitely many (geodesically) convex balls $B_j:=B(x_j,2r_j)$, $j=1,...,k$, with $x_j\in \Sigma$ and $r_j>0$ such that $\Sigma\subset \bigcup_{j=1}^k B(x_j,r_j)$. For each $j$ set $\Sigma_j:=\Sigma\cap \overline{B(x_j,r_j)}$, $\Omega_j:=B(x_j,\frac{3}{2}r_j)$ and $T_j:=r_j^2/4$. Choose $\lambda>\max_j \lambda_{B_j,T_j}$ with notation as in Lemma \ref{L:resolventupperloc}. Clearly convex balls are infinitesimally convex.

To see (i) note that $\ccpct_{2,2}^{C_c^\infty(M)}(\Sigma_j)=0$ for all $j$ the stated hypothesis, Theorem \ref{T:mfdccpct} and monotonicity. Consequently for any fixed $j$ and any $\delta>0$ we can find $v\in C_c^\infty(M)$ with $\left\|(\lambda-\Delta_\mu)v\right\|_{L^2(M)}^2<\delta$.
Now let $\varphi\in \omega_{\Sigma_j}^{C_c^\infty(\Omega_j)}$. Proceeding as in \eqref{E:passover} we find that $\left\|(\lambda-\Delta_\mu)(\varphi v)\right\|_{L^2(M)}^2<c\:\delta$ with a constant $c$ independent of $\delta$ and $v$. On the other hand $\varphi v\in \mathcal{D}(\mathcal{L}^{B_j})$, so that $\varphi v=G_\lambda^{B_j} f$ with $f\in L^2(B)$, and by locality $f$ vanishes outside $\Omega_j$. Consequently $\cpct_{2,2}^{B_j}(\Sigma_j,\Omega_j)<\big\|f\big\|_{L^2(\Omega)}^2=\left\|(\lambda-\Delta_\mu)(\varphi v)\right\|_{L^2(M)}^2$. Since $\delta$ was arbitrary, this gives $\cpct_{2,2}^{B_j}(\Sigma_j,\Omega_j)=0$, and by Lemma \ref{L:resolventupperloc} also $\mathcal{H}^{d-4+\varepsilon}(\Sigma_j)=0$. Since this works for all $j$, we have 
 $\mathcal{H}^{d-4+\varepsilon}(\Sigma)=0$ by subadditivity, and this proves (i) for the case $d>4$. The case $d=4$ is similar.

To see (ii) for $d>4$, note that since $\mathcal{H}^{d-4}(\Sigma_j)\leq \mathcal{H}^{d-4}(\Sigma)<+\infty$ for all $j$, Lemma \ref{L:resolventlowerloc} implies that $\cpct_{2,2}^{B_j,T_j}(\Sigma_j,\Omega_j)=0$ for all $j$. Corollary \ref{C:simplify}
yields $\ccpct_{2,2}^{C_c^\infty(M)}(\Sigma_j)=0$, and given $\delta>0$ we can find $v_j\in C_c^\infty(M)$ with $v_j=1$ on an open neighborhood $U_j\subset B_j$ of $\Sigma_j$ and $\left\|(\lambda-\Delta_\mu)v_j\right\|_{L^2(M)}^2<\delta$.
Let $\{\varphi_j\}_{j=1}^k\subset C_c^\infty(M)$ be a partition of unity subordinate to the finite open cover $\{B(x_j,r_j)\}_{j=1}^k$ of $\Sigma$. We may assume that $\sum_j \varphi_j=1$ on $U:=\bigcup_j U_j$, otherwise shrink the $U_j$. Similarly as in \eqref{E:passover} we see that 
\[\left\|(\lambda-\Delta_\mu)(\varphi_jv_j)\right\|_{L^2(M)}^2<c\:\delta\]
with a constant $c>0$ independent of $v_j$, $j$ and $\delta$. Clearly $\varphi_j v_j=\varphi_j$ on $U_j$, respectively, and consequently $v:=\sum_j \varphi_j v_j$ equals one on the neighborhood $U$ of $\Sigma$. Since
\[\ccpct_{2,2}^{C_c^\infty(M)}(\Sigma)^{1/2}\leq \big\|v\big\|_{\mathcal{D}(\mathcal{L})}=\big\|(\lambda-\Delta_\mu)v\big\|_{L^2(M)}\leq \sum_i \left\|(\lambda-\Delta_\mu)(\varphi_jv_j)\right\|_{L^2(M)}\leq k\sqrt{c}\sqrt{\delta}\]
with a constant $c$ independent of $\delta$, it follows that $\ccpct_{2,2}^{C_c^\infty(M)}(\Sigma)=0$, what by Theorem \ref{T:mfdccpct} entails the claim in (ii). The arguments for $d=4$ are similar.

\end{proof}

\begin{remark}\mbox{}
\begin{enumerate}
\item[(i)] For more specific $\Sigma$ there are established results for general complete Riemannian manifolds $M$: For $d\geq 4$ and one-point sets $\Sigma=\{x\}$ the essential self-adjointness of $\Delta_\mu|_{C_c^\infty(\mathring{M})}$ had been shown in \cite{CdV82}, along with counterexamples for $d=2,3$; see \cite[p. 161]{RS80v2} for the Euclidean case. For $d\geq 4$ any countable subset of $M$ has $\mathcal{H}^{d-4}$- resp. $\mathcal{H}^{h_2}$-measure zero, so Theorem \ref{T:Hausdorff} (ii) recovers and extends the positive result for one-point sets.
\item[(ii)] For the case that $\Sigma$ is a closed submanifold it had been shown in  \cite[Theorem 3]{Masamune1999} that $\Delta_\mu|_{C_c^\infty(\mathring{M})}$ is essentially self-adjoint if any only if $\dim_H \Sigma <d-3$. Theorem \ref{T:Hausdorff} requires $d\geq 4$ but makes no further assumptions on the closed set $\Sigma$.
\end{enumerate}
\end{remark}

\begin{remark}\label{R:mayormaynot}
For the limit case $d=4$ one can construct uncountable sets of zero Hausdorff dimension whose removal may or may not destroy essential self-adjointness: Let $M=\mathbb{R}^4$ and let $\mu$ be the $4$-dimensional Lebesgue measure. For any $\varepsilon\geq 0$ one can construct a finite measure $\mu_\varepsilon\in \mathcal{M}_+(\mathbb{R}^4)$ with compact support $\Sigma_\varepsilon\subset \mathbb{R}^4$ such that
\[c^{-1}(-\log r)^{-1-\varepsilon}\leq \mu(B(x,r))\leq c\:(-\log r)^{-1-\varepsilon},\quad x\in \Sigma_\varepsilon,\quad 0<r<1,\]
with a constant $c>1$ (depending on $\varepsilon$). See \cite[Theorem 3.3 and Examples 3.8]{Bricchi04} and the references cited there. Let $h^{(\varepsilon)}(r):=(1+\log_+\frac{1}{r})^{-1-\varepsilon}$, note that $h^{(0)}=h_2$ with notation as in \eqref{E:hp}. Then $0<\mathcal{H}^{h^{(\varepsilon)}}(\Sigma_\varepsilon)<+\infty$, and all $\Sigma_\varepsilon$ are uncountable sets with $\dim_H \Sigma_\varepsilon=0$, \cite[Proposition 3.5]{Bricchi04}. By Theorem \ref{T:Hausdorff} (ii) the operator $\Delta_\mu|_{C_c^\infty(\mathbb{R}^4\setminus \Sigma_0)}$ is essentially self-adjoint on $L^2(\mathbb{R}^4)$, but by Remark \ref{R:exampleh} and Theorem \ref{T:Hausdorff} (i) none of the operators $\Delta_\mu|_{C_c^\infty(\mathbb{R}^4\setminus \Sigma_\varepsilon)}$, $\varepsilon>0$, is.
\end{remark}

\begin{remark}
Lemmas \ref{L:resolventlower} and \ref{L:resolventupper} together with Theorem \ref{T:mfdcpctp} also permit characterizations of $L^p$-uniqueness using Hausdorff measures. We leave this to the reader.
\end{remark}

\section{Sub-Riemannian manifolds}\label{S:subRiemann}

We provide similar results for Sub-Riemannian manifolds.

\subsection{Essential self-adjointness and capacities}
Let $M$ be a smooth (connected) manifold and let $V_1,...,V_m$ be linearly independent smooth vector fields. Set $D^1:=\lin \{V_1,...,V_m\}$ and $D^k:=D^{k-1}+[D^1,D^{k-1}]$, $k\geq 2$. The vector fields $V_1,...,V_m$ are said to satisfy the H\"ormander
condition if for any $x\in M$ there is some $k$ such that $T_xM=\{V_x: V\in D^k\}$, and if so, then $(M,\{V_1,...,V_m\})$ is said to be a sub-Riemannian manifold. We assume this is the case. The corresponding sub-Riemannian gradient 
is then defined by 
\[\nabla f:=\sum_{i=1}^m V_i(f)V_i,\quad f\in C_c^\infty(M).\]
Let $\mu$ be a smooth measure (that is, a Borel measure with a smooth density in every local chart) and let $\diverg_\mu$ be the divergence, defined as minus the formal adjoint of $\nabla$ with respect to $\mu$. The sub-Laplacian $\Delta_\mu$, \cite{GL16}, defined by
\[\Delta_\mu f:=\sum_{i=1}^m V_i^2 f+ (\diverg_\mu V_i)V_if,\quad f\in C_c^\infty(M),\]
is symmetric in $L^2(M,\mu)$. Let $\varrho$ be the natural Carnot-Caratheodory metric defined by $V_1,...,V_m$, \cite[Section 3.2]{ABB19}. If the metric space $(M,\varrho)$ is complete, then $\Delta_\mu|_{C_c^\infty(M)}$ is essentially self-adjoint, \cite[Section 12]{Strichartz86}; see also \cite[Theorem 1.5]{ABFP21}. Given a closed set $\Sigma\subset M$ we write again $\mathring{M}$ for its complement. We have the following special case of Theorem \ref{T:esa}.

\begin{theorem}\label{T:mfdccpctsub} Suppose that $(M,\varrho)$ is complete and $\Sigma\subset M$ is closed. Then $\ccpct_{2,2}^{C_c^\infty(M)}(\Sigma)=0$ if and only if $\mu(\Sigma)=0$ and $\Delta_\mu|_{C_c^\infty(\mathring{M})}$ is essentially self-adjoint.
\end{theorem}

\subsection{$L^p$-uniqueness on Carnot groups}

Let $G$ be a simply connected Lie group whose Lie algebra is the sum $g=\mathcal{V}_1\oplus...\oplus \mathcal{V}_N$ of nontrivial subspaces $\mathcal{V}_i$ such that $[\mathcal{V}_1,\mathcal{V}_i]=\mathcal{V}_{i+1}$, $i=1,...,N-1$, and $[\mathcal{V}_1,\mathcal{V}_N]=\{0\}$. Then $G$ is called a Carnot group of depth $N$, see \cite{BaudoinBonnefont16}. Let $V_1,...,V_m$ be a basis of $\mathcal{V}_1$, we may interpret the $V_i$ as left invariant vector fields on $G$; they satisfy H\"ormander's condition. We endow $G$ with the corresponding Carnot-Caratheodory metric $\varrho$, \cite[Section III.4]{VSCC92}, and the Haar measure $\mu$. Then there are positive constants $c$, $c'$ such that 
\begin{equation}\label{E:Carnotdreg}
c\:r^d\leq \mu(B(x,r))\leq c'\:r^d
\end{equation}
for all $x\in G$ and $r>0$, where $d=\sum_{i=1}^N i\:\dim\mathcal{V}_i$, \cite[IV.5.9 Remark]{VSCC92}. The sub-Laplacian defined by
\[\Delta_\mu f:=\sum_{i=1}^m V_i^2f,\quad f\in C_c^\infty(G),\]
is symmetric with respect to $\mu$ and essentially self-adjoint. We consider a closed subset $\Sigma\subset G$ and write $\mathring{G}:=G\setminus \Sigma$. Recall that in the symmetric and semibounded case $L^2$-uniqueness is equivalent to being densely defined and essentially self-adjoint.

\begin{theorem}\label{T:Carnot}
Let $G$ be a Carnot group, let $\Sigma\subset G$ be a closed set and Let $1<p<+\infty$.
\begin{enumerate}
\item[(i)] If $\mu(\Sigma)=0$ and $\Delta_\mu|_{C_c^\infty(\mathring{G})}$ is $L^p$-unique, then $\ccpct_{2,p}^{C_c^\infty(G)}(\Sigma)=0$ and $\cpct_{2,p}(\Sigma)=0$. If in addition $d\geq 2p$, then $\mathcal{H}^{d-2p+\varepsilon}(\Sigma)=0$ for any $\varepsilon>0$, and in the special case $p=2$ and $d=4$ also $\mathcal{H}^h(\Sigma)=0$ for all $h$ satisfying \eqref{E:condh}.
\item[(ii)] If $\cpct_{2,p}(\Sigma)=0$, then $\ccpct_{2,p}^{C_c^\infty(G)}(\Sigma)=0$ and $\Delta_\mu|_{C_c^\infty(\mathring{G})}$ is $L^p$-unique. This happens in particular if $d>2p$ and $\mathcal{H}^{d-2p}(\Sigma)<+\infty$ or $d=2p$ and $\mathcal{H}^{h_p}(\Sigma)<+\infty$.
\end{enumerate}
\end{theorem}

\begin{proof}
We have $\Gamma(f)=\sum_{i=1}^m (V_i f)^2$, $f\in C_c^\infty(G)$, and \cite[VIII.2.7 Theorem]{VSCC92} the gradient estimate $\Gamma(p_t(0,\cdot))(x)\leq c\:t^{-(d+1)/2}\:\exp\big(-\varrho(0,x)^2/c't\big)$ holds for all $x\in G$ and $t>0$ with universal positive constants $c$ and $c'$. This implies condition \eqref{A:Riesztrafo}. By \cite[VIII.2.9 Theorem]{VSCC92} condition \eqref{E:LY} holds. Together with (\ref{E:Carnotdreg}) and standard calculations, Theorem \ref{T:esa}, Corollaries \ref{C:firstcomparison} and \ref{C:secondcomparison} and Lemmas \ref{L:resolventlower} and \ref{L:resolventupper} now the result.
\end{proof}

\begin{remark}\label{R:Heisenberg}
For any $n\geq 1$ the Heisenberg group $\mathbb{H}_n$ is a Carnot group of depth $2$, its homogeneous dimension is $d=2n+2$, and this is also its Hausdorff dimension. The essential self-adjointness of the natural sub-Laplacian on $C_c^\infty(\mathbb{H}_1\setminus \{0\})$ had been shown in \cite[Theorem 1.7]{ABFP21}. The above Theorem \ref{T:Carnot} complements this result, it applies for any $n\geq 1$ and any closed set $\Sigma\subset \mathbb{H}_n$.
\end{remark}

\section{RCD*(K,N) spaces}\label{S:RCD}

This section contains uniqueness results for Laplacians on {\rm RCD}$^*(K,N)$ spaces a after the removal of a set $\Sigma$.

Let $(M,\varrho,\mu)$ be a complete separable geodesic metric space with a locally finite Borel regular measure $\mu$ having full support and satisfying 
$\mu(B(x_0,r)) \le c\:e^{c'r^2}$, $r>0$, for some point $x_0\in M$ and with positive constants $c$, $c'$ independent of $r$.

Let $\mathrm{Lip}(M)$ denote the space of Lipschitz functions on $M$, and given $f \in \mathrm{Lip}(M)$ and $x\in M$, consider its local Lipschitz constant $|\mathrm{D} f|(x):=\limsup_{y \to x} |f(y)-f(x)|/\varrho(y,x)$  at $x$. Given $f\in L^2(M)$, its {\it Cheeger energy}, \cite{AGS14, Cheeger99}, is defined as
\[\mathrm{Ch}(f)=\frac{1}{2}\inf\Bigl\{ \liminf_{n \to \infty} \int_M |\mathrm{D}  f_n|^2 d\mu: f_n \in \mathrm{Lip}(M)\cap L^2(M),\ \lim_n \int_M|f_n -f|^2 d\mu =0 \Bigr\}.\]
We write $W^{1}(M):=\{f \in L^2(M): \mathrm{Ch}(f)<\infty\}$. One can show that 
\[\mathrm{Ch}(f)=\frac{1}{2}\int_M|\nabla f|^2_w d\mu, \quad f \in W^{1}(M),\]
where $|\nabla f|_w$ is the (unique) minimal weak upper gradient of $f$, see \cite{AGS14, Shan00}. If $M$ is a Riemannian manifold, then 
\begin{equation}\label{E:gradientscoincide}
|\nabla u|_w=|\nabla u|\quad \text{$\mu$-a.e.} 
\end{equation}
and the space $W^{1}(M)$ has the same meaning as in Subsection \ref{SS:weighted}.

The space $(M,\varrho,\mu)$ is said to be \emph{infinitesimally Hilbertian} if $\mathrm{Ch}$ is a quadratic form. In this case $(\mathrm{Ch},W^{1}(M))$ is a local Dirichlet form, \cite{BH91}; we write $(\mathcal{L},\mathcal{D}(\mathcal{L}))$ for its generator, $(P_t)_{t>0}$ for the associated symmetric Markov semigroup and $\Gamma(f):=|\nabla f|_w^2$, $f\in W^{1}(M)$, for the carr\'e du champ.

Let $K\in\mathbb{R}$ and $1\leq N<+\infty$. Following \cite[Theorem 7]{EKS15}, the triple $(M,\varrho,\mu)$ may be called an \emph{{\rm RCD}$^*(K,N)$ space} if  
\begin{enumerate}
\item[(i)] each $f\in  W^{1}(M) $ with $\Gamma(f) \leq 1$ has a continuous version, and
\item[(ii)] for each $f\in W^{1}(M)$ and $t>0$ we have the \emph{Bakry--Ledoux gradient estimate}
\[
\Gamma(P_tf)+\frac{4Kt^2}{N(e^{2Kt}-1)}\:|\mathcal LP_tf|^2\leq e^{-2Kt}P_t(\Gamma(f))\quad \text{$\mu$-a.e.}\]
\end{enumerate}
with the convention that for $K=0$ the fraction $K/(e^{2Kt}-1)$ is replaced by its limit $1/(2t)$ for $K\to 0$.

Let $(M,\varrho,\mu)$ be an {\rm RCD}$^*(K,N)$-space. Then it is locally compact, and the Dirichlet form $(\mathrm{Ch},W^{1}(M))$ is regular and strongly local in the sense of \cite{FOT94}. The semigroup $(P_t)_{t>0}$ is strong Feller, \cite[Theorem 7.1]{AGMR15}, and moreover, for any $f \in L^\infty(M)$ and any $t>0$ the function $P_t f$ is Lipschitz, see \cite[Theorem 7.3]{AGMR15}. See also \cite[Theorems 6.1 and 6.8]{AGS14b}. The semigroup $(P_t)_{t>0}$ is absolutely continuous, we write $p_t(x,y)$ to denote its heat kernel. Suppose that $K\leq 0$ and $1\leq N<+\infty$. Then for any $\varepsilon>0$ there are positive constants $c_8$ and $c_{10}$ such that 
\begin{equation}\label{E:LYRCD}
\frac{1}{c_8\mu(B(y,\sqrt{t}))}\:\exp\Big(-\frac{\varrho(x,y)^2}{(4-\varepsilon)t}-c_{10}t\Big)\leq p_t(x,y)\leq \frac{c_8}{\mu(B(y,\sqrt{t}))}\:\exp\Big(-\frac{\varrho(x,y)^2}{(4+\varepsilon)t}+c_{10}t\Big)
\end{equation}
for all $t>0$ and $x,y\in M$; in the case $K=0$ the bounds remain true with $0$ in place of $c_{10}$. See \cite[Theorems 1.1 and 1.2]{JLZ16}. Moreover, for any $\varepsilon>0$ there are positive constants $c_{11}$ and $c_{12}$ such that 
\begin{equation}\label{E:gradestRCD}
\sqrt{\Gamma(p_t(x,\cdot))}(y)\leq \frac{c_{11}}{\sqrt{t}\mu(B(y,\sqrt{t}))}\:\exp\Big(-\frac{\varrho(x,y)^2}{(4+\varepsilon)t}+c_{12}t\Big)
\end{equation}
for all $t>0$ and $\mu$-a.a. $x,y\in M$; again we may replace $c_{12}$ by $0$ if $K=0$. See \cite[Corollaries 1.1 and 1.2]{JLZ16}. Moreover, $\mu$ satisfies the \emph{local doubling condition}: For any $R>0$ there is a constant $c_{D,R}>0$ such that
\begin{equation}\label{E:localdoubling}\tag{LD}
\mu(B(x,2r))\leq c_{D,R}\:\mu(B(x,r)),\quad x\in M,\quad 0<r\leq R,
\end{equation}
see \cite[Theorem 2.3]{Sturm06-1}. In the case that $K=0$ the global doubling condition \eqref{E:doubling} is known to hold. The preceding gives the following variant of \eqref{E:loghk}.
\begin{corollary}\label{C:loghkT}
Let $(M,\varrho,\mu)$ be an {\rm RCD}$^*(K,N)$-space with $K\leq 0$ and $1\leq N<+\infty$. For any sufficiently small $\varepsilon>0$  and any $0<T<+\infty$ we can find $1<\alpha<2$ and positive constants $c_1$ and $c_2$ such that 
\begin{equation}\label{E:loghkT}\tag{$\mathrm{LG_T'}$}
\frac{\sqrt{\Gamma(p_t(x,\cdot))}(y)}{p_{\alpha t}(x,y)}\leq \frac{c_1\:e^{c_2t}}{\sqrt{t}}
\end{equation}
for all $0<t\leq T$ and $\mu$-a.a. $x,y\in M$. If \eqref{E:doubling} holds, then this remains true for $T=+\infty$, i.e. \eqref{E:loghk} holds.
\end{corollary}

\begin{proof}
The right hand side in (\ref{E:gradestRCD}) rewrites 
\[\frac{c_{11}}{\sqrt{t}\mu(B(y,\sqrt{t}))}\:\exp\Big(-\frac{\varrho(x,y)^2}{(4-\varepsilon)\alpha t}-c_{10}at\Big)\:\exp\Big((c_{10}\alpha +c_{12})t\Big)\]
with $\alpha:=\frac{4+\varepsilon}{4-\varepsilon}$, and the claim follows using \eqref{E:localdoubling} and the left hand side in \eqref{E:LYRCD}.
\end{proof}

\begin{remark}\label{R:logTRCD} Proceeding as in Proposition \ref{P:hktosg} we see that \eqref{E:loghkT} implies a 'local' variant of \eqref{E:log}, valid for all $0<t<T$.
\end{remark}

Suppose that $1<\alpha<2$ and $T>0$ as in \eqref{E:localdoubling} are fixed and choose $\lambda>\lambda_0:=\frac{2}{2-\alpha}\:c_2\vee \frac{1}{T}$. We consider the truncated kernels
\[g_\lambda^T(x,y):=\int_0^T e^{-\lambda t}p_t(x,y)\:dt,\quad x,y\in M,\]
and write $G_\lambda^T$ for the corresponding truncated resolvent operators. Using Remark \ref{R:logTRCD} and following the proof of Theorem \ref{T:Mazja} but with integration restricted to $(0,T)$ (cf. Theorem \ref{T:Mazjaloc}), we obtain the following.

\begin{theorem}\label{T:MazjaRCD}
Let $(M,\varrho,\mu)$ be an {\rm RCD}$^*(K,N)$-space with $K\leq 0$ and $1\leq N<+\infty$. Let $F\in C^2(\mathbb{R}_+)$ be a function satisfying (\ref{E:Mazja}). Then for $1< p< +\infty$, $T$ and $\lambda$ as above and any nonnegative $f\in L^p(M)$ we have $F\circ G_\lambda^{T} f\in\mathcal{D}(\mathcal{L}^{(p)})$ and $\big\|F\circ G_\lambda^{T} f\big\|_{\mathcal{D}(\mathcal{L}^{(p)})}\leq c_3\left\|f\right\|_{L^p(M)}$
with $c_3>0$ independent of $f$. If \eqref{E:doubling} holds, then the statements remain valid with $G_\lambda$ in place of $G_\lambda^T$.
\end{theorem}

In the case of metric measure spaces the choice of a suitable operator core is less obvious than in the manifold case. We
consider the space
\begin{equation}\label{E:alternativeA1}
\mathcal{A}:=\big\lbrace f\in C_b(M)\cap \mathcal{D}(\mathcal{L})\cap\mathcal{D}(\mathcal{L}^{(1)}):\ \text{$\Gamma(f) \in L^\infty(M)$ and $\mathcal{L}f\in L^\infty(M)$}\big\rbrace.
\end{equation}

\begin{proposition}\label{P:Aisnice}
Let $(M,\varrho,\mu)$ be an {\rm RCD}$^*(K,N)$ space with $K\leq 0$ and $1\leq N<+\infty$. Then $\mathcal A$ satisfies conditions \eqref{A:basic} and \eqref{A:bump} and for any $1<p<+\infty$ also \eqref{A:basicp}. It satisfies condition \eqref{A:functionsT} with $G_\lambda^T$ in place of $G_\lambda$. If \eqref{E:doubling} holds, then also \eqref{A:Riesztrafo} is satisfied and \eqref{A:functionsT} holds in its original form.
\end{proposition}

\begin{proof}
Given $f,g\in\mathcal{A}$, we have $fg\in\mathcal{A}$: Clearly $fg\in C_b(M)$, and by the discussion in Section \ref{S:Setup} also in $\mathcal{D}(\mathcal{L}^{(1)})\cap L^\infty(M)$. Since both $\mathcal{L}^{(1)}f$ and $\mathcal{L}^{(1)}g$ are members of $L^\infty(X,\mu)$ and by the Cauchy-Schwarz inequality also $\Gamma(f,g)\in L^\infty(M)$, we have $\mathcal{L}^{(1)}(fg) \in L^1(M)\cap L^\infty(M)\subset L^2(M)$ by (\ref{E:Gamma}) and consequently $fg\in \mathcal{D}(\mathcal{L})$ by 
\eqref{E:domaininfo}. Clearly also $\Gamma(fg)\in L^\infty(M)$ by the product rule. This shows the initial claim and \eqref{A:basic}. 

Given $K\subset M$ compact, let $f\in C_c(M)$ be nonnegative and positive at some point. Integrating the lower bound in \eqref{E:LYRCD} we see that $g_\lambda^T(x,y)>0$ for all $x,y\in M$, and consequently $G_\lambda^Tf(x)>0$ for all $x\in M$. By continuity $G_\lambda^Tf$ then must be bounded away from zero on $K$, and we may assume that $G_\lambda^Tf>1$ on $K$. Let $F$ be a $C^2$-truncation and set $v:=F\circ G_\lambda^T f$; then $v=1$ on $K$. The chain rules \eqref{E:chain} and \eqref{E:usualchainrule} show that $\Gamma(v)\in L^\infty(M)$, $v\in \mathcal{D}(\mathcal{L}^{(1)})$, and together with \eqref{E:gradestRCD} and the estimates in the proof of Proposition \ref{P:hktosg} also $\mathcal{L}v\in L^\infty(M)$. Theorem \ref{T:MazjaRCD} gives $v\in \mathcal{D}(\mathcal{L})$, so that $v$ is seen to be in $\mathcal{A}$, what proves \eqref{A:bump} and \eqref{A:functionsT}.

We claim that $\mathcal{A}$ is dense in any $L_p(M)$. By resolvent approximation the set 
\[\mathcal{A}_0:=\{G_\lambda f:\ f\in C_c(M),\ \lambda>\lambda_0\}\] 
is dense in $L^p(M)$, so it suffices to show that $\mathcal{A}_0\subset \mathcal{A}$. Let $f\in C_c(M)$ and $\lambda>\lambda_0$, we may assume that $f$ is nonzero. Clearly $G_\lambda f\in C_b(M)\cap \mathcal{D}(\mathcal{L})\cap\mathcal{D}(\mathcal{L}^{(1)})$ and $\mathcal{L}G_\lambda f=\lambda G_\lambda f-f\in L^\infty(M)$. Similarly as in the proof of Proposition \ref{P:integral} we obtain 
\begin{equation}\label{E:fullbound}
\sqrt{\Gamma(G_\lambda f)}(x)\leq \int_0^\infty e^{-\lambda t}\sqrt{\Gamma(P_tf)}(x)dt\leq \int_0^\infty e^{-\lambda t}\int_M \sqrt{\Gamma(p_t(x,\cdot)}(y)|f(y)|\mu(dy)dt,
\end{equation}
and if $\supp f$ has diameter less than $\sqrt{T}$, we can use \eqref{E:gradestRCD} and \eqref{E:loghkT} to bound the preceding uniformly in $x$ by
\[c_1\left\|f\right\|_{L^\infty(M)}\int_0^Tt^{-1/2}e^{-(\lambda-c_2)t}dt+\frac{c_{11}}{\mu(\supp f)}\left\|f\right\|_{L^1(M)}\int_T^\infty t^{-1/2}e^{-(\lambda-c_{12})t}dt<+\infty.\]
A uniform bound for $\sqrt{\Gamma(G_\lambda f)}$ with general $f\in C_c(M)$ follows using a finite cover of $\supp f$ by open balls $B_1,...,B_k$ and applying the preceding to $f\mathbf{1}_{B_j}$, $j=1,...,k$ in place of $f$. This shows that $\Gamma(G_\lambda f)\in L^\infty(M)$ and consequently $G_\lambda f\in \mathcal{A}$. 

Now observe that for any $t>0$ we have $P_t(\mathcal{A})\subset \mathcal{A}$: Given $f\in \mathcal{A}$ it is clear that $P_t f\in C_b(M)\cap \mathcal{D}(\mathcal{L})\cap\mathcal{D}(\mathcal{L}^{(1)})$. We have $\mathcal{L}P_tf=P_t\mathcal{L}f\in L^\infty(M)$ by $L^\infty$-contractivity and $\Gamma(P_tf)\leq \Lip(P_tf)^2<+\infty$ by Lipschitz regularization; here $\Lip(\cdot)$ denotes the global Lipschitz constant. Now condition \eqref{A:basicp} follows from \cite[Chapter 1, Proposition 3.3]{EK86}.

Under \eqref{E:doubling} we have \eqref{E:loghk} by Corollary \ref{C:loghkT}, and combining with \eqref{E:fullbound} we obtain
\[\big\|\sqrt{\Gamma(G_\lambda f)}\big\|_{L^p(M)}\leq \int_0^\infty\big\|\sqrt{\Gamma(P_tf)}\big\|_{L^p(M)}dt\leq c_1\int_0^\infty e^{-(\lambda-c_2)t}t^{-1/2}dt\:\left\|f\right\|_{L^p(M)},\]
and this implies \eqref{A:Riesztrafo} by density.
\end{proof}

\subsection{$L^p$-uniqueness and capacities}

Using Theorem \ref{T:esa}, Corollary \ref{C:firstcomparison} and Theorem \ref{T:MazjaRCD} plus a slight variation of Corollary \ref{C:secondcomparison} we obtain a characterization of uniqueness in terms of capacities. We write $\cpct_{2,2}^T$ for the capacity defined as in \eqref{E:cpctdef} with $G_\lambda^T$ in place of $G_\lambda$ and as before, given a closed subset $\Sigma$, $\mathring{M}:=M\setminus \Sigma$. Recall condition \eqref{A:approxid}.

\begin{theorem}\label{T:RCDLpuni}
Let $(M,\varrho,\mu)$ be an {\rm RCD}$^*(K,N)$ space with $K\leq 0$ and $1\leq N<+\infty$. 
\begin{enumerate}
\item[(i)] If $\Sigma\subset M$ is a closed set, $\mu(\Sigma)=0$ and $\Delta_\mu|_{\mathcal{A}(\mathring{M})}$ is is essentially self-adjoint, then $\ccpct_{2,2}^{\mathcal{A}}(\Sigma)=0$ and $\cpct_{2,2}(\Sigma)=0$.
\item[(ii)] If $\Sigma\subset M$ is a compact set and $\ccpct_{2,2}^{\mathcal{A}}(\Sigma)=0$, then $\mu(\Sigma)=0$ and $\Delta_\mu|_{\mathcal{A}(\mathring{M})}$ is essentially self-adjoint. This happens in particular if $\cpct_{2,2}^T(\Sigma)=0$. If \eqref{A:approxid} holds, then the conclusion remains true for closed $\Sigma\subset M$; in this case we may also replace $\mathcal{A}(\mathring{M})$ by $\mathcal{A}_c(\mathring{M})$.
\end{enumerate}
If \eqref{E:doubling} holds, then $\cpct_{2,2}^T$ in (ii) may be replaced by $\cpct_{2,2}$ and 
analogous statements hold for $L^p$-uniqueness, $1<p<+\infty$.
\end{theorem}

\begin{remark}\label{R:trivialifcompact}
If $\diam(M)<+\infty$ then $M$ is compact, and both the density of $\mathcal{A}_c$ in $\mathcal{A}$ and condition \eqref{A:approxid} are trivially true.
\end{remark}

\subsection{Essential self-adjointness and Hausdorff measures}

In the special case that the volume $\mu$ is $N$-regular, one has the following metric characterization.

\begin{theorem}\label{T:HausdorffRCD}\mbox{}
Let $(M,\varrho,\mu)$ be an {\rm RCD}$^*(K,N)$ space with $K\leq 0$ and $4\leq N<+\infty$ and suppose that $\mu$ is $N$-regular, more precisely, that
\begin{equation}\label{E:globalNreg}
\frac{1}{c}\:r^N\leq \mu(B(x,r))\leq c\:r^N,\quad  x\in M,\quad 0<r<\diam (M),
\end{equation} 
where $c>1$ is a fixed constant. Let $\Sigma\subset M$ be a closed set.
\begin{enumerate}
\item[(i)] If $\Sigma$ is of zero measure and $\Delta_\mu|_{\mathcal{A}(\mathring{M})}$ is essentially self-adjoint, then $\mathcal{H}^{N-4+\varepsilon}(\Sigma)=0$ for all $\varepsilon>0$ (if $N>4$) respectively $\mathcal{H}^h(\Sigma)=0$ for all Hausdorff functions $h$ satisfying \eqref{E:condh} (if $N=4$).
\item[(ii)] If $\mathcal{H}^{N-4}(\Sigma)<+\infty$ (if $N>4$) respectively $\mathcal{H}^{h_2}(\Sigma)<+\infty$ (if $N=4$), then $\Delta_\mu|_{\mathcal{A}(\mathring{M})}$ is essentially self-adjoint. 
\end{enumerate}
The statements remain valid with $\mathcal{A}_c(\mathring{M})$ in place of $\mathcal{A}(\mathring{M})$.
\end{theorem}

Under the hypotheses of Theorem \ref{T:HausdorffRCD} condition \eqref{A:approxid} holds. 
\begin{lemma}\label{L:approxid}
Let $(M,\varrho,\mu)$ be an {\rm RCD}$^*(K,N)$ space with $K\leq 0$ and $2< N<+\infty$ and suppose that $\mu$ is $N$-regular as in \eqref{E:globalNreg}. Then condition \eqref{A:approxid} holds.
\end{lemma}
\begin{proof}
By Remark \ref{R:trivialifcompact} we may assume that $\diam(M)=+\infty$. Integrating the bounds in \eqref{E:LYRCD} and using (\ref{E:globalNreg}), we see that $\frac{1}{c}\:\varrho(x,y)^{2-N}\leq g_\lambda(x,y)\leq c\:\varrho(x,y)^{2-N}$ for all $x,y\in M$ with $c>1$ independent of $x$ or $y$. Fix some $x_0\in M$. For any $r>0$ and $x\in B(x_0,r)$ it follows that
\[G_\lambda \mathbf{1}_{B(x_0,r)}(x)\geq \frac{1}{c}\int_{B(x_0,r)}\varrho(x,y)^{2-N}\mu(dy)\geq \frac{(2r)^{2-N}}{c}\mu(B(x_0,r))\geq c\:r^2\]
with $c>0$ in the last expression independent of $r$. Let $0<t_0<1$ and let $F$ be a $C^2$-truncation such that $F(t)=0$ for $t\leq t_0$ and $F(t)=1$, $t\geq 1$. Then for any large enough $r$ we have $F\circ G_\lambda \mathbf{1}_{B(x_0,r)}(x)=1$, $x\in B(x_0,r)$. Let $\gamma>\frac{N}{N-2}$. For any $x\in M\setminus B(x_0,r^\gamma+r)$ we similarly obtain
\[G_\lambda \mathbf{1}_{B(x_0,r)}(x)\leq c\:\int_{B(x_0,r)}\varrho(x,y)^{2-N}\mu(dy)\geq c\:r^{-\gamma(N-2)}\mu(B(x_0,r))\leq c\:r^{-\gamma(N-2)+N}\]
with $c>0$ independent of $r$, so that $F\circ G_\lambda \mathbf{1}_{B(x_0,r)}(x)=0$, $x\in M\setminus B(x_0,r^\gamma+r)$,
whenever $r$ is large enough. By completeness the support of $F\circ G_\lambda \mathbf{1}_{B(x_0,r)}$ is compact.

Now set $h_n:=F\circ G_\lambda \mathbf{1}_{B(x_0,n)}$, $n\geq n_0$, with large enough $n_0$ for the preceding to hold. By construction $0\leq h_n\leq 1$ and $h_n\uparrow 1$ as $n\to \infty$. Clearly also $h_n\in C_b(M)\cap \mathcal{D}(\mathcal{L})\cap \mathcal{D}(\mathcal{L}^{(1)})$. The global validity of \eqref{E:globalNreg} implies \eqref{E:doubling} and therefore \eqref{E:loghk}, what gives 
\begin{align}
\sqrt{\Gamma(h_n)}(x)&\leq \|F'\|_{\sup}\sqrt{\Gamma(G_\lambda\mathbf{1}_{B(x_0,n)})}(x)\notag\\
&\leq \|F'\|_{\sup}\int_0^\infty e^{-\lambda t}\int_M \sqrt{\Gamma(p_t(x,\cdot))}(y)\mathbf{1}_{B(x_0,n)}(y)\mu(dy)\:dt\notag\\
&\leq c_1\:\|F'\|_{\sup}\int_0^\infty t^{-1/2}e^{-(\lambda-c_2)t}\int_M p_{\alpha t}\mathbf{1}_{B(x_0,n)}(y)\mu(dy)\:dt\notag\\
&\leq  c_1\:\|F'\|_{\sup}\int_0^\infty t^{-1/2}e^{-(\lambda-c_2)t}dt,\notag
\end{align}
and therefore $\sup_n \big\|\Gamma(h_n)^{1/2}\big\|_{L^\infty(M)}<+\infty$. From Theorem \ref{T:MazjaRCD} and its generalization as pointed out in Remark \ref{R:moregeneral} we obtain
\[\left\|\mathcal{L}h_n\right\|_{L^\infty(M)}\leq \left\|(\lambda-\mathcal{L})h_n\right\|_{L^\infty(M)}+\lambda\: \left\|h_n\right\|_{L^\infty(M)}\leq c_3\left\|\mathbf{1}_{B(x_0,n)}\right\|_{L^\infty(M)}+\lambda=c_3+\lambda\]
and therefore $\sup_n \left\|\mathcal{L}h_n\right\|_{L^\infty(M)}<+\infty$. 
\end{proof}

We prove  Theorem \ref{T:HausdorffRCD}.

\begin{proof}
By Lemma \ref{L:compactsupps} and Lemma \ref{L:approxid} we can make full use of Theorem \ref{T:esa}.  By the arguments used in the proof of Theorem \ref{T:Hausdorff} we may assume that $\Sigma$ is compact. Writing $T:=\diam(\Sigma)^2+1$, integrating the upper bound in \eqref{E:LYRCD} and using \eqref{E:globalNreg}, we see that 
\[g_\lambda(x,y)\leq c\int_0^Te^{-\lambda t} t^{-N/2}\exp\big(-\frac{\varrho(x,y)^2}{c\:t}\big)dt+c\:T^{-N/2}\int_T^\infty e^{-(\lambda-c_{10})t}dt\leq c\:\varrho(x,y)^{2-N}\]
for all $x$ and $y$ from a neighborhood of $\Sigma$, note that for such $x$ and $y$ we have $\varrho(x,y)\leq \sqrt{T}$ and therefore can adjust the constants. Since $\cpct_{2,2}(\Sigma)=0$ by Theorem \ref{T:RCDLpuni} (i) we can now follow the arguments in the proof of Lemma \ref{L:resolventupper} to conclude (i). To see (ii) note that integrating the lower bound in \eqref{E:LYRCD} from $0$ to $T$ and using \eqref{E:globalNreg}, we obtain $g^T_\lambda(x,y)\geq c\:\varrho(x,y)^{2-N}$ for all $x,y\in M$. A slight variation of Lemma \ref{L:resolventlower} shows that $\cpct_{2,2}^T(\Sigma)=0$, and this implies (ii) by Theorem \ref{T:RCDLpuni} (ii).
\end{proof}

\subsection{Essential self-adjointness on {\rm CAT} spaces}
For ${\rm CAT}(0)$ spaces the explicit assumption \eqref{E:globalNreg} of $N$-regularity can be omitted.  Given a geodesic triangle $\tau=\tau(x_0, x_1, x_2)$ in $M$, its {\it comparison triangle} is the unique (up to isometry) geodesic triangle $\bar{\tau}={\tau}(\bar{x}_0, \bar{x}_1, \bar{x}_2)$ in $\mathbb{R}^2$ such that $\|\bar{x}_i - \bar{x}_j\|_{\mathbb{R}^2} = \varrho(x_i, x_j)$ for all $i, j$.
Let $[x_i,x_j]$ denote the geodesic ('side of $\tau$') in $M$ connecting $x_i$ and $x_j$, and let $[\bar{x}_i, \bar{x}_j]$ denote the line segment in $\mathbb{R}^2$ connecting $\bar{x}_i$ and $\bar{x}_j$.  Given a point $p \in [{x}_i,{x}_j]$, a point $\bar{p} \in [\bar{x}_i, \bar{x}_j]$ is called {\it a comparison point} for $p$ if $\varrho(x_i, p) = \|\bar{x}_i - \bar{p}\|_{\mathbb{R}^2}$. 
The geodesic triangle $\tau\subset M$ is said to have the {\it ${\rm CAT}(0)$ property} if for any two points $p$, $q$ on different sides of $\tau$ and corresponding comparison points $\bar{p}, \bar{q}\in \bar{\tau}$ we have $\varrho(p, q) \le \|\bar{p}-\bar{q}\|_{\mathbb{R}^2}$. The space $M$ is said to be a {\it ${\rm CAT}(0)$ space} (or \emph{Hadamard space}) if all its geodesic triangles have the ${\rm CAT}(0)$ property. We say that $M$ is an {\it {\rm RCD}$^*(K,N) \cap {\rm CAT}(0)$ space} if it is both a {\rm RCD}$^*(K,N)$ and a ${\rm CAT}(0)$ space.

\begin{theorem}\label{T:HausdorffRCDCAT}\mbox{}
Let $(M,\varrho,\mu)$ be an {\rm RCD}$^*(K,N) \cap {\rm CAT}(0)$ space with $K\leq 0$ and $4\leq N<+\infty$.
\begin{enumerate}
\item[(i)] If $\Sigma$ is of zero measure and $\Delta_\mu|_{\mathcal{A}(\mathring{M})}$ is essentially self-adjoint, then $\mathcal{H}^{N-4+\varepsilon}(\Sigma)=0$ for all $\varepsilon>0$ (if $N>4$) respectively $\mathcal{H}^h(\Sigma)=0$ for all Hausdorff functions $h$ satisfying \eqref{E:condh} (if $N=4$).
\item[(ii)] If $\mathcal{H}^{N-4}(\Sigma)<+\infty$ (if $N>4$) respectively $\mathcal{H}^{h_2}(\Sigma)<+\infty$ (if $N=4$), then $\Delta_\mu|_{\mathcal{A}(\mathring{M})}$ is essentially self-adjoint. 
\end{enumerate}
The statements remain valid with $\mathcal{A}_c(\mathring{M})$ in place of $\mathcal{A}(\mathring{M})$.
\end{theorem}

\begin{proof}
One can follow the proof of Theorem \ref{T:Hausdorff}: Every metric ball in ${\rm CAT}(0)$ is geodesically convex, see \cite[Prop.\ 2.2]{BridsonHaefliger99}. It follows that every closed ball $\overline{B} \subset M$, equipped with the restrictions to $\overline{B}$ of the metric and the measure, is a compact {\rm RCD}$^*(K,N)$ space, see for instance \cite[Prop.\ 1.4]{Sturm06} and \cite[Thm.\ 4.19]{AGS14b}. By volume comparison the measure on $\overline{B}$ is $N$-regular. Now Corollary \ref{C:loghkT} and the arguments preceding it give estimates \eqref{E:localEuclid}--\eqref{E:loggradestZh} on $\overline{B}$. One can then pass from closed balls $\overline{B}$ to all of $M$ by a variant of the localization argument in Theorem \ref{T:Hausdorff}, under the present assumptions the existence of suitable cut-off functions is ensured by \cite[Prop. 6.9]{AMS16}: For any closed ball $\overline{B}$ and any compact subset $K$ of its interior $B$, we can find a  cutoff function $\varphi\in \mathcal D(\mathcal L)$ with $0 \le \varphi \le 1$, $\supp \varphi \subset B$, $\varphi=1$ on $K$ and such that both $\mathcal{L}\varphi$ and $\Gamma(\varphi)$ are in $L^\infty(M)$. This also allows the construction of suitable partitions of unity by the usual procedure.
\end{proof}

\subsection{Limit spaces of non-collapsed manifolds}
We observe consequences of Theorem \ref{T:HausdorffRCD} (ii) for limits of non-collapsed manifolds. Suppose that $1\leq d<+\infty$, $v>0$ and that $((M_i, g_i, p_i))_i$ is a sequence of pointed $d$-dimensional Riemannian manifolds $(M_i, g_i, p_i)$ having bounded Ricci curvature 
\begin{equation}\label{E:uniRiccibound}
|{\rm Ric}_{M_i}| \le d-1,
\end{equation}
and satisfying the \emph{non-collapsing condition}
\begin{equation}\label{E:noncollaps}
\mu_{M_i}(B(p_i,1))>v>0,
\end{equation}
where $\mu_{M_i}$ denotes the Riemannian volume on $M_i$. Suppose that $(M, \varrho, \mu, p)$ is the limit of this sequence 
with respect to pointed measured Gromov-Hausdorff convergence. The {\rm RCD}$^\ast$ conditions are stable under pointed measured 
Gromov-Hausdorff convergence, see \cite{GMS15, LV09, Sturm06} and \cite[Remark 10.7]{AH16}. Therefore also $(M, \varrho, \mu, p)$ is an  {\rm RCD}$^\ast(d-1,d)$-space, and by the Cheeger-Colding  volume convergence theorem, \cite[Theorem 5.9]{ChC97}, the measure $\mu$ is the $d$-dimensional Hausdorff measure on $M$, $\mu=\mathcal{H}^d$. By \cite[p.11 and Corollary 5.8]{ChN15} there is a closed set $\mathcal S \subset M$, called the \emph{singular set}, such that $\mathcal{R}:=M \setminus \mathcal S$ is a smooth manifold with a $C^{1,\alpha}$ Riemannian metric and $\dim_H(\mathcal S) \le d-4$.  By \cite[Theorem 7.1]{JN21} we have $\mathcal{H}^{d-4}(\mathcal S)<+\infty$.

On $\mathcal{R}$ we can consider the classical Laplacian $\Delta_\mu|_{C_c^\infty(M)}$ with respect to $\mu|_\mathcal{R}$, note that the $C^1$-regularity of the Riemannian metric is sufficient to introduce it, cf. \cite[Section 3.6]{Grigoryan2009}. The following observation tells that from the point of view of self-adjoint extensions the singular set $\mathcal{S}$ 'can be neglected'. 

\begin{theorem}\label{T:codim4}
Let $d\geq 4$ and let $(M, \varrho, \mu, p)$ be the pointed measured Gromov-Hausdorff limit of a sequence $((M_i, g_i, p_i))_i$ of pointed $d$-dimensional Riemannian manifolds satisfying (\ref{E:uniRiccibound}) and (\ref{E:noncollaps}) and let $\mathcal{S}\subset M$ denote the singular set. 
\begin{enumerate}
\item[(i)] We have $C_c^\infty(\mathcal{R})\subset \mathcal{A}_c(\mathcal{R})$, and the operator $\mathcal{L}|_{\mathcal{A}_c(\mathcal{R})}$ is an extension of the classical Laplace operator $\Delta_\mu|_{C_c^\infty(\mathcal{R})}$ on the Riemannian manifold $\mathcal{R}$.
\item[(ii)] The operator $\mathcal{L}|_{\mathcal{A}_c(\mathcal{R})}$ is essentially self-adjoint on $L^2(M)$, and its unique self-adjoint extension is $(\mathcal{L},\mathcal{D}(\mathcal{L}))$.  
\item[(iii)] The space $C_c^\infty(\mathcal{R})$ is dense in $\mathcal{A}_c(\mathcal{R})$ with respect to $\|\cdot\|_{\mathcal{D}(\mathcal{L})}$. The operator $\Delta_\mu|_{C_c^\infty(\mathcal{R})}$ is essentially self-adjoint on $L^2(M)$, and its unique self-adjoint extension is $(\mathcal{L},\mathcal{D}(\mathcal{L}))$.  
\end{enumerate}
\end{theorem}

Given $\lambda>0$ we write 
\begin{equation}\label{E:asusual} 
\mathrm{Ch}_\lambda(u,v):=\mathrm{Ch}(u,v)+\lambda \left\langle u,v\right\rangle_{L^2(M)},\quad u,v\in W^{1}(M),
\end{equation}
as usual. By 
\[D(u,v):=\int_{\mathcal{R}}\left\langle \nabla f, \nabla g\right\rangle_{T\mathcal{R}}d\mu,\quad u,v\in W^{1}(\mathcal{R}),\]
we denote the classical Dirichlet integral on the Riemannian manifold $\mathcal{R}$, and we use the notation $D_\lambda$ with analogous meaning as in \eqref{E:asusual}.

\begin{proof} Any $u\in C_c^\infty(\mathcal{R})$ is also in $C_c(M)\cap \Lip(M)$ and therefore in $W^{1}(M)$. By \eqref{E:gradientscoincide} 
we have 
\begin{equation}\label{E:energiescoincide}
\mathrm{Ch}(u,v)=D(u,v),\quad u\in C_c^\infty(\mathcal{R}),\quad v\in C_c(M)\cap W^{1}(M).
\end{equation}  
Given $u\in C_c^\infty(\mathcal{R})$ we write $f:=(\lambda-\Delta_\mu)u$. Since $f\in C_c(M)$ by locality, we have $G_\lambda f\in \mathcal{A}$ by the proof of Proposition \ref{P:Aisnice}. Using \eqref{E:energiescoincide} it follows that
\[\mathrm{Ch}_\lambda(u,v)=D_\lambda(u,v)=\left\langle f,v\right\rangle_{L^2(\mathcal{R})}=\left\langle f,v\right\rangle_{L^2(M)}=\mathrm{Ch}_\lambda(G_\lambda f,v)\]
for any $v\in C_c(M)\cap W^{1}(M)$, and the density of such functions in $W^{1}(M)$ implies that $u=G_\lambda f$. Consequently $C_c^\infty(\mathcal{R})\subset \mathcal{A}_c(\mathcal{R})$. If  $u\in C_c^\infty(\mathcal{R})$, then
\[\int_M(\mathcal{L}u)v\:d\mu=-\mathrm{Ch}_\lambda(u,v)=-D_\lambda(u,v)=\int_{\mathcal{R}}(\Delta_\mu u)v\:d\mu\]
for all $v\in C_c(M)\cap W^{1}(M)$ and therefore $\mathcal{L}u=\Delta_\mu u$, what shows (i).  Statement  (ii) is immediate from Theorem \ref{T:HausdorffRCD} (ii). To see (iii) let $W^1(\mathcal{R})$, $W^1_0(\mathcal{R})$ and $W^2_0(\mathcal{R})$ be as in Subsection \ref{SS:weighted} and let $(\mathcal{L}^\mathcal{R},W_0^2(\mathcal{R}))$ be 
the Dirichlet Laplacian on $\mathcal{R}$. By \eqref{E:gradientscoincide} we have $\mathcal{A}_c(\mathcal{R})\subset W^1(\mathcal{R})\cap C_c(\mathcal{R})$. On the other hand $W^1(\mathcal{R})\cap C_c(\mathcal{R})\subset W_0^1(\mathcal{R})$: If $u\in W^1(\mathcal{R})\cap C_c(\mathcal{R})$, $\varphi\in C_c^\infty(\mathcal{R})$ is a bump function equal to one on $\supp u$ and $(P_t^\mathcal{R})_ {t>0}$ is the Dirichlet heat semigroup on $L^2(\mathcal{R})$ generated by $(\mathcal{L}^\mathcal{R},W_0^2(\mathcal{R}))$, then $\varphi P_t^\mathcal{R}u\in C_c^\infty(\mathcal{R})$ for each $t>0$ and
\begin{equation}\label{E:semigroupsmoothing}
\lim_{t\to 0}\varphi P_t^\mathcal{R}u=u
\end{equation} 
in $W^1(\mathcal{R})$. Given $u\in \mathcal{A}_c(\mathcal{R})\subset W_0^1(\mathcal{R})$, we have $D(u,v)=\mathrm{Ch}(u,v)$
for all $v\in W_0^1(\mathcal{R})$ and as a consequence, $|D(u,v)|\leq \left\|\mathcal{L}u\right\|_{L^2(M)}\left\|v\right\|_{L^2(\mathcal{R})}$. This implies that $u\in W_0^2(\mathcal{R})$ and $\mathcal{L}^\mathcal{R}u=\mathcal{L}u$. Proceeding similarly as in \eqref{E:passover}, we obtain \eqref{E:semigroupsmoothing} in $W_0^2(\mathcal{R})$, endowed with the graph norm for $\mathcal{L}^\mathcal{R}$, and therefore in $\mathcal{D}(\mathcal{L})$. This shows the claimed density, which implies that also the closure of $\Delta_\mu|_{C_c^\infty(\mathcal{R})}$ on $L^2(M)$ equals $(\mathcal{L},\mathcal{D}(\mathcal{L}))$.

\end{proof}

\end{document}